\documentclass[reqno]{memo-l}
\usepackage{float,rotating,amsfonts, amsbsy, amsmath, amsthm, amssymb, latexsym, verbatim, enumerate}
\usepackage{mathrsfs}
\usepackage{pdfsync}
\usepackage{stmaryrd}
\usepackage{bm}
\usepackage{fancyhdr}

\makeatletter
\newcommand{\imod}[1]{\allowbreak\mkern4mu({\operator@font mod}\,\,#1)}
\makeatother

\renewcommand{\a}{\alpha}
\renewcommand{\b}{\beta}

 \newcommand{\e}{\epsilon}
 \renewcommand{\L}{\Lambda}
\renewcommand{\l}{\lambda} 
 
 \renewcommand{\to}{\rightarrow}
 \newcommand{\s}{\sigma}
\renewcommand{\o}{\omega} 
 \newcommand{\C}{\mathcal{C}}

\newcommand{\leqs}{\leqslant}

 \newcommand{\vs}{\vspace{3mm}}
\newcommand{\la}{\langle}
\newcommand{\ra}{\rangle}
\newcommand{\R}{\mathbb{R}}
\newcommand{\Z}{\mathbb{Z}}

\newtheorem{theorem}{Theorem}

\newtheorem{remark}{Remark}

\newtheorem{thm}{Theorem}[section]
\newtheorem{prop}[thm]{Proposition}
\newtheorem{lem}[thm]{Lemma}
\newtheorem{cor}[thm]{Corollary}

\newtheorem{rmk}[thm]{Remark}

\theoremstyle{definition}

\theoremstyle{remark}
\newtheorem{remk}[thm]{Remark}

\numberwithin{section}{chapter}
\numberwithin{equation}{chapter}

\begin{document}

\frontmatter

\title{Irreducible geometric subgroups of \\ classical algebraic groups}

\author{Timothy C. Burness}
\address{T.C. Burness, School of Mathematics, University of Bristol, Bristol BS8 1TW, United Kingdom}
\email{\texttt{t.burness@bristol.ac.uk}}

\author{Souma\"{i}a Ghandour}
\address{S. Ghandour, Facult\'{e} des Sciences, Section V, Universit\'{e} Libanaise, Nabatieh, Lebanon}
\email{\texttt{soumaia.ghandour@gmail.com}} 

\author{Donna M. Testerman}
\address{D.M. Testerman, Section de Math\'ematiques, Station 8, \'{E}cole Polytechnique F\'{e}d\'{e}rale de Lausanne, CH-1015 Lausanne, Switzerland}
\email{\texttt{donna.testerman@epfl.ch}}

\date{November 25th, 2010}



\begin{abstract}
Let $G$ be a simple classical algebraic group over an algebraically closed field $K$ of characteristic $p \ge 0$ with natural module $W$. Let $H$ be a closed subgroup of $G$ and let $V$ be a non-trivial irreducible tensor-indecomposable $p$-restricted rational $KG$-module such that the restriction of $V$ to $H$ is irreducible. In this paper we classify all such triples $(G,H,V)$, where $H$ is a maximal closed disconnected positive-dimensional subgroup of $G$, and $H$ preserves a natural geometric structure on $W$. 
\end{abstract}

\maketitle

\tableofcontents

\mainmatter

\chapter{Introduction}\label{s:intro}

In the 1950s, Dynkin \cite{Dynkin1} determined the maximal closed connected subgroups of the classical matrix groups over $\mathbb{C}$. In the course of his analysis, he observed that if $G$ is a semisimple algebraic group over $\mathbb{C}$ and if $\phi:G \to {\rm SL}(V)$ is an irreducible rational representation, then with specified exceptions the image of $G$ is maximal among closed connected subgroups in one of the classical groups ${\rm SL}(V)$, ${\rm Sp}(V)$ or ${\rm SO}(V)$. In particular, he determined all
triples $(G,H,V)$ where $G$ is a simple closed irreducible subgroup of ${\rm SL}(V)$ and $H$ is a positive-dimensional closed connected subgroup of $G$ such that the restriction of $V$ to $H$, denoted by $V|_{H}$, is also irreducible. Naturally, one is interested in investigating the more general situation where $\mathbb{C}$ is replaced by an arbitrary algebraically closed field.

In the 1980s, Seitz \cite{Seitz2} initiated the investigation of such triples in the positive characteristic setting as part of a wider study of the subgroup structure of finite and algebraic simple groups. By introducing a variety of new techniques, which differed greatly from those employed by Dynkin, he determined all the triples $(G,H,V)$ where $G$ is a simply connected simple classical algebraic group over any algebraically closed field $K$ of characteristic $p \ge 0$, and $H$ is a closed connected subgroup of $G$. For exceptional algebraic groups $G$, the detailed analysis of Testerman \cite{Test1} handles the case where $H$ is connected, and the case where $H$ is a positive-dimensional disconnected subgroup of an exceptional group has been settled very recently by Ghandour \cite{g_paper}.

Therefore, in order to complete the analysis for simple algebraic groups it remains to consider the case where $G$ is classical and $H$ is a positive-dimensional disconnected subgroup. Here a partial analysis has been undertaken by Ford. In  \cite{Ford1} and \cite{Ford2} he classifies all triples $(G,H,V)$ where $G$ is classical and $H$ is disconnected, under the additional assumption that the connected component $H^0$ is simple and, more importantly, that the $KH^0$-composition factors of $V$ are $p$-restricted as $KH^0$-modules (with the convention that every dominant weight is $p$-restricted when $p=0$). These extra assumptions help to simplify the analysis. Nevertheless, under these hypotheses Ford discovered a very interesting family of triples $(G,H,V)$ with $G=B_n$ and $H=D_n.2$ (see \cite[Section 3]{Ford1}). Furthermore, these examples were found to have applications to the representation theory of the symmetric groups, and led to a proof of the Mullineux conjecture (see \cite{FK}). However, for future applications it is desirable to remove the additional conditions on $H$ and $V$.

Some special cases have been studied by various authors. For instance, in \cite{GT}, Guralnick and Tiep consider irreducible triples in the special case $G={\rm SL}(W)$, $V=S^k(W)$ (the $k$-th symmetric power of the natural module $W$ for $G$) and $H$ is any closed (possibly finite) subgroup of $G$. A similar analysis of the exterior powers $\Lambda^k(W)$ is in progress. These results have found interesting applications in the study of holonomy groups of stable vector bundles on  smooth projective varieties (see \cite{BK}). For finite groups, a related problem for subgroups of ${\rm GL}_{n}(q)$ is investigated by Kleshchev and Tiep in \cite{KT}.

Let $G$ be a simple classical algebraic group over an algebraically closed field $K$ of characteristic $p \ge 0$ with natural module $W$. More precisely,
let $G = {\rm Isom}(W)'$, where ${\rm Isom}(W)$ is the full isometry 
group of a suitable form $f$ on $W$, namely, the zero bilinear form, a 
symplectic form, or a non-degenerate quadratic form.
We write $G=Cl(W)$ to denote the respective simple classical groups ${\rm SL}(W)$, ${\rm Sp}(W)$ and ${\rm SO}(W)$ defined in this way. Note that
$G = {\rm Isom}(W)\cap {\rm SL}(W)$, with the exception
that if $p=2$, $f$ is quadratic and $\dim W$ is even, then $G$ has index 2 in 
${\rm Isom}(W)\cap {\rm SL}(W)$. 

The main theorem on the subgroup structure of $G$ is due to Liebeck and Seitz. In \cite{LS}, six collections of natural, geometrically defined closed subgroups of $G$ are presented, labelled $\C_i$ for $1 \le i \le 6$. These collections include the stabilizers of appropriate subspaces of $W$, and the stabilizers of certain direct sum and tensor product decompositions of $W$. We set $\mathcal{C}=\bigcup_{i=1}^{6}\C_i$. The main theorem of \cite{LS} states that if $H$ is a closed subgroup of $G$ then either $H$ is contained in a member of $\C$,  or roughly speaking, $H$ is almost simple (modulo scalars) and the unique quasisimple normal subgroup of $H$ (which coincides with the connected component $H^0$ when $H$ is infinite) acts irreducibly on $W$. 
We write $\mathcal{S}$ to denote this additional collection of `non-geometric' maximal subgroups of $G$. This result provides a natural algebraic group analogue of Aschbacher's well-known structure theorem for finite classical groups (see \cite{Asch}), and we refer the reader to Section \ref{ss:cla} for further details on the $\C_i$ collections.

In this paper, we consider the case where $H$ is a maximal disconnected positive-dimensional subgroup in one of the above $\C_i$ collections, a so-called \emph{geometric} maximal subgroup of $G$. Fix a set of fundamental dominant weights $\{\l_1, \ldots, \l_n\}$ for $G$ (in this paper, we adopt the standard labelling of simple roots and fundamental weights given in Bourbaki \cite{Bour}). Let $V$ be an irreducible $KG$-module with highest weight $\l$. As previously remarked, Seitz \cite{Seitz2} handles the case where $H$ is connected (more precisely, the case where 
$H/Z(G)$ is connected), so we will assume $H/Z(G)$ is disconnected. It is also natural to assume that $V$ is tensor-indecomposable as a $KG$-module, and in view of Steinberg's tensor product theorem, we will also assume that $\l$ is a $p$-restricted highest weight for $G$ (where we adopt the convention that every dominant weight is $p$-restricted when $p=0$). Finally, to ensure that the weight lattice of the underlying root system $\Sigma$ of $G$  coincides with the character group of a maximal torus of $G$, we replace $G$ by a  simply connected cover also having root system $\Sigma$.
Our main theorem is the following:

\vs

\begin{theorem}\label{main}
Let $G$ be a simply connected cover of a simple classical algebraic group $Cl(W)$ defined over an algebraically closed field $K$ of characteristic $p \ge 0$. Assume that $(G,p) \neq (B_n,2)$.
Let $V$ be an irreducible tensor-indecomposable $p$-restricted $KG$-module with highest weight $\l$, and let $H \in \C$ be a maximal positive-dimensional subgroup of $G$ such that $H/Z(G)$ is disconnected. Then 
$V|_{H}$ is irreducible if and only if $(G,H,\l)$ is one of the cases recorded in Table \ref{t:main}.
\end{theorem}

\vs

\begin{remark}\label{r:conds}
\emph{
Let us make some remarks on the statement of Theorem \ref{main}.
\begin{enumerate}\addtolength{\itemsep}{0.3\baselineskip}
\item[{\rm (a)}] Since $A_1 \cong B_1 \cong C_1$, $B_2 \cong C_2$ and $A_3 \cong D_3$ (as algebraic groups), the
conditions on $n$ recorded in the first column of Table \ref{t:main} avoid an unnecessary repetition of cases. Also note that $D_n$ is simple if and only if $n \ge 3$. 
\item[{\rm (b)}] Note that in the statement of Theorem \ref{main}, and for the remainder of this paper, we assume that $(G,p) \neq (B_n,2)$. The relevant irreducible triples for $(G,p) = (B_n,2)$ can be quickly deduced from the corresponding list of cases presented in Table \ref{t:main} for the dual group of type $C_n$ (there is an exceptional isogeny between groups of type $B_n$ and $C_n$ when $p=2$).
\item[\rm{(c)}] The final column of Table \ref{t:main} records necessary and sufficient conditions for the irreducibility of the corresponding triple $(G,H,V)$, as well as ensuring the existence and maximality of $H$ in $G$ (see Table \ref{t:disc} in Section \ref{ss:cla}).
\item[\rm{(d)}] The required conditions for the case $(G,H)=(B_n,D_n.2)$ in Table \ref{t:main} are as follows:
\vspace{2mm}
\begin{quote}
$a_n=1$; if $a_i,a_j \neq 0$, where $i<j<n$ and $a_k=0$ for all $i<k<j$, then $a_i+a_j \equiv i-j \imod{p}$; if $i<n$ is maximal such that $a_i \neq 0$ then $2a_i \equiv -2(n-i)-1 \imod{p}$.
\end{quote}
\vspace{2mm}
In particular, if $p=0$ then $\l=\l_n$ is the only possibility. This interesting family of examples was found by Ford (see case ${\rm U}_{2}$ in \cite[Table II]{Ford1}). 
\item[{\rm(e)}] The required conditions for the case $(G,H)=(C_n,C_l^t.S_t)$ in Table \ref{t:main} (with $n=lt$ and $\l = \l_{n-1}+a\l_n$) are as follows:
\vspace{2mm}
\begin{quote}
$t=2$ and either $(l,a) = (1,0)$, or $0 \le a <p$ and $2a+3 \equiv 0 \imod{p}$.
\end{quote}
\vspace{2mm}
\item[{\rm(f)}] In Table \ref{t:main} we write $T_k$ to denote a $k$-dimensional torus.
\item[{\rm (g)}] Let $(G,H,\l)$ be one of the following cases from Table \ref{t:main}:
$$\hspace{14mm} (A_{n}, D_m.2,\l_j),\; (C_n,C_l^t.S_t,\l_n), \; (C_n, D_n.2,\sum_{i<n}a_i\l_i),\;(D_{n},(2 \times B_l^2).2,\l_k)$$ 
(where $1 \le j \le n$, $j \neq m$, and $k=n-1,n$). Here the connected group $H^0$ acts irreducibly on $V$, so these cases appear in \cite[Table 1]{Seitz2}.  
\item[{\rm (h)}] Consider the example $\l=\l_7$ for $(G,H)=(D_n, C_l^t.S_t)$ with $n=8$ and $(l,t)=(1,4)$ or $(2,2)$. If $\tilde{H}$ denotes the image of $H$ under a non-trivial graph automorphism of $G$ then $\l=\l_8$ is an example for 
the pair $(G,\tilde{H})$. Similarly, $\l=\l_4$ is an example for $(G,\tilde{H})$ when $G=D_4$ and $H$ is of type $A_3T_1.2$ or $C_1^3.S_3$ (with $p=2$). In this latter case, $\l=\l_1+\l_3$ is an additional example for $(G,\tilde{H})$.
\item[{\rm (i)}] For each case $(G,H,\l)$ listed in Table \ref{t:main}, the restriction of $\l$ to a suitable maximal torus of $(H^0)'$ is given in a later table, according to the particular $\C_i$ family to which $H$ belongs (see Table \ref{t:c1,3,6} if $H \in \C_1 \cup \C_3 \cup \C_6$, Table \ref{t:c2} if $H \in \C_2$ and Tables \ref{t:c4i} and \ref{t:c4ii} if $H \in \C_4$). In these tables we also record the number $\kappa$ of $KH^0$-composition factors in $V|_{H^0}$.
\item[{\rm (j)}] Let $G,H$ and $V$ be given as in the statement of Theorem \ref{main}, and assume   that $V|_{H}$ is irreducible. Then $H^0$ is reductive. Indeed, the unipotent radical of $H^0$ acts completely reducibly on $V$, which implies that it acts trivially on $V$.
\end{enumerate}}
\end{remark}

\renewcommand{\arraystretch}{1.2}
\begin{table}\small
\vspace{10mm}
$$\begin{array}{llcll} \hline
G & H &  & \l & \mbox{Conditions} \\ \hline
A_n & A_l^tT_{t-1}.S_t & n+1=(l+1)t & \l_1, \, \l_n & \mbox{$l \ge 0$, $t \ge 2$} \\
n \ge 1 & & &  \l_k &  \mbox{$l=0$, $t \ge 2$, $1<k<n$} \\
 & A_l^t.S_t & n+1=(l+1)^t& \l_1, \, \l_n & \mbox{$l \ge 2$, $t \ge 2$}  \\
& & & \l_2,\, \l_{n-1} & \mbox{$l \ge 2$, $t=2$, $p\neq 2$}\\
& D_m.2 & n+1 = 2m & \l_k & \mbox{$1 \le k \le n$, $n \ge 3$, $p \neq 2$} \\
&&&&\\
B_n & D_lB_{n-l}.2 & & \l_n & 1\le l < n \\
n \ge 3 & (2^{t-1}\times B_l^t).S_t & 2n+1=(2l+1)t & \l_1, \, \l_n & \mbox{$l \ge 1$, $t \ge 3$ odd} \\
& B_l^t.S_t & 2n+1=(2l+1)^t & \l_1 & \mbox{$l \ge 1$, $t \ge 2$} \\
& & & \l_4 & \mbox{$(l,t)=(1,2)$, $p \neq 3$} \\
  & D_n.2 & &  \sum_{i=1}^{n} a_i\l_i & 
\mbox{See Remark \ref{r:conds}}{\rm (d)} \\
&&&& \\
C_n & C_l^t.S_t & n=lt &  \l_1 & \mbox{$l \ge 1$, $t \ge 2$} \\
n \ge 2 & & & \l_n & \mbox{$l \ge 1$, $t \ge 2$, $p=2$} \\
 & & & \l_{n-1}+a\l_n & \mbox{See Remark \ref{r:conds}(e)} \\
& A_{n-1}T_1.2 & & \l_1 &  p \neq 2 \\
& C_aD_b.2 & n=2ab & \l_1 & \mbox{$b \ge 2$, $p \neq 2$} \\ 
& C_l^t.S_t & 2n=(2l)^t&  \l_1 & \mbox{$l \ge 1$, $t \ge 3$ odd, $p \neq 2$} \\
& & & \l_2			& \mbox{$(l,t)=(1,3)$, $p \neq 2$}\\
	&	& & \l_3			& \mbox{$(l,t)=(1,3)$, $p \neq 2,3$}\\
 & D_n.2 & & \l_n,\,   \sum_{i=1}^{n-1} a_i\l_i& p=2 \\
&&&& \\
D_n & D_lD_{n-l}.2 && \l_{n-1},\,\l_{n} & 1 \leq l < n/2 \\
n \ge 4 & (2^{t-1}\times B_l^t).S_t & 2n=(2l+1)t & \l_1, \, \l_{n-1},\,\l_{n} & \mbox{$l \ge 1$, $t \ge 2$ even, $p \neq 2$} \\
& (D_l^t.2^{t-1}).S_t & n=lt & \l_1,\, \l_{n-1},\,\l_{n} & \mbox{$l \ge 1$, $t \ge 2$} \\
& & & \l_{1}+\l_{n-1},\, \l_{1}+\l_{n} & (t,p)=(2,2),\; \mbox{$l \ge 3$ odd} \\
 & A_{n-1}T_1.2 & & \l_1 & \mbox{$n$ even} \\
 & & & \l_3 &	n=4 \\
 & D_aD_b.2^2 & n=2ab & \l_1 & \mbox{$a>b \ge 2$, $p \neq 2$} \\
 & C_l^t.S_t & 2n=(2l)^t & \l_1 &  \mbox{$l \ge 1$, $t \ge 2$ even or $p=2$} \\
 & & & \l_3,\, \l_1+\l_4,\, \l_3+\l_4			& (l,t,p) = (1,3,2)\\
	&	& & \l_7			& \mbox{$(l,t)=(1,4)$, $p\ne3$}\\
	&	& & \l_7			& \mbox{$(l,t)=(2,2)$, $p\ne5$}\\
 & (D_l^t.2^t).S_t & 2n=(2l)^t & \l_1 & \mbox{$l \ge 3$, $t \ge 2$, $p \neq 2$} \\ \hline
\end{array}$$
\caption{The maximal closed positive-dimensional disconnected irreducible geometric subgroups of classical algebraic groups}
\label{t:main}
\end{table}
\renewcommand{\arraystretch}{1}

\vs

Let $G, H$ and $V$ be given as in the statement of Theorem \ref{main}. Since $H \in \mathcal{C}$ we have a concrete description of the embedding of $H$ in $G$, and we can directly calculate the restriction of $\l$ to a suitable maximal torus of $(H^0)'$ in terms of a set of fundamental weights for $(H^0)'$. If $V|_{H^0}$ is irreducible then the possibilities for $(G,H,V)$ can be deduced from the work of Seitz \cite{Seitz2}, so we focus on the situation where $V|_{H}$ is irreducible, but $V|_{H^0}$ is reducible. By Clifford theory, $V|_{H^0}$ is completely reducible and the highest weights of $KH^0$-composition factors of $V$ are $H$-conjugate; we can exploit this to severely restrict the possibilities for $\l$. This is similar to the approach adopted by Ghandour \cite{g_paper} in her work on exceptional algebraic groups; the challenge here is to extend her combinatorial analysis of weights to classical groups of arbitrary rank.

In \cite{BGMT} we complete the analysis of disconnected maximal positive-dimensional subgroups of classical groups by dealing with the relevant subgroups in the aforementioned $\mathcal{S}$ collection. Here $H^0$ is simple (modulo scalars) and so it remains to overcome Ford's $p$-restricted hypothesis in \cite{Ford2} on the highest weights of the composition factors of $V|_{H^0}$. This requires completely different methods to those used in the present paper.

\vs

Finally, let us make some comments on the organization of this paper. In Section \ref{s:prel} we recall some preliminary results concerning weights and their multiplicities, and we discuss the  maximal subgroups in the various $\C_i$ collections. In addition, we calculate the dimension of some specific irreducible $KG$-modules, which we will need in the proof of Theorem \ref{main} (see Table \ref{t:dims} for a summary). Some of these results are new, and may be of independent interest. We begin the proof of Theorem \ref{main} in Section \ref{s:c136}, where we deal with the disconnected subgroups in the $\C_1$, $\C_3$ and $\C_6$ collections. Next, in Section  \ref{s:c2} we consider the imprimitive subgroups comprising the collection $\C_2$. Finally, in Sections \ref{s:c4i} and \ref{s:c4ii} we deal with the disconnected tensor product subgroups in $\C_4$. Note that the subgroups in the collection $\C_5$ are finite, so our work in Sections \ref{s:c136} -- \ref{s:c4ii} will complete the proof of Theorem \ref{main}.

\section*{Acknowledgments} 

The first author acknowledges the support of EPSRC grant EP/I019545/1, a Collaborative Small Grant from the London Mathematical Society and he thanks FIM at the ETH Zurich and the Section de Math\'{e}matiques at EPFL for their kind hospitality. The second author
acknowledges the support of Swiss National Science Foundation grant number
PP0022-114794, held by Professor Karin Baur at the ETHZ. The third author
acknowledges the support
of a Swiss National Science Foundation grant (number 200021-122267/1). The authors thank Gunter Malle for helpful comments, and they thank an anonymous referee for his or her careful reading of the paper, which led to many improvements. 

\chapter{Preliminaries}\label{s:prel}

\numberwithin{table}{chapter}

\section{Notation}\label{ss:notation}

First, let us fix some notation that we will use for the rest of the paper. As in the statement of Theorem \ref{main}, let $G$ be a simply connected cover of a simple classical algebraic group $Cl(W)$ defined over an algebraically closed field $K$ of characteristic $p \ge 0$. 
Here $Cl(W) = {\rm Isom}(W)'$, where ${\rm Isom}(W)$ is the full isometry 
group of a form $f$ on $W$, which is either the zero bilinear form, a 
symplectic form, or a non-degenerate quadratic form. In particular, we note that $Cl(W) = {\rm Isom}(W)\cap {\rm SL}(W)$, with the exception
that if $p=2$, $f$ is quadratic and $\dim W$ is even, in which case $Cl(W)$ has index two in 
${\rm Isom}(W)\cap {\rm SL}(W)$. It is convenient to adopt the familiar Lie notation $A_n$, $B_n$, $C_n$ and $D_n$ to denote the various possibilities for $G$, where $n$ is the rank of $G$. As in the statement of Theorem \ref{main}, in this paper we will assume that $(G,p) \neq (B_n,2)$.

Let $B=UT$ be a Borel subgroup of $G$ containing a fixed maximal torus $T$ of $G$, where $U$ denotes the unipotent radical of $B$. Let $\Pi(G)=\{\a_1, \ldots, \a_{n}\}$ be a corresponding base of the root system $\Sigma(G)=\Sigma^+(G) \cup \Sigma^{-}(G)$ of $G$, where $\Sigma^+(G)$ and $\Sigma^{-}(G)$ denote the positive and negative roots of $G$, respectively. Let $X(T) \cong \Z^n$ denote the character group of $T$ and let $\{\l_1, \ldots, \l_{n}\}$ be the fundamental dominant weights for $T$ corresponding to our choice of base $\Pi(G)$, so $\la \l_i,\a_j\ra=\delta_{i,j}$ for all $i$ and $j$, where
$$\la \l, \a \ra = 2\frac{(\l,\a)}{(\a,\a)},$$
$(\, ,\,)$ is the usual inner product on $X(T)_{\R}=X(T) \otimes_{\Z} \R$, and $\delta_{i,j}$ is the familiar Kronecker delta. In addition,  let $s_{\a}:X(T)_{\R} \to X(T)_{\R}$  be the reflection relative to $\a \in \Sigma(G)$, defined by $s_{\a}(\l)=\l-\la \l,\a \ra\a$; the corresponding finite group $\la s_{i} \mid 1 \le i \le n\ra$ generated by the fundamental reflections $s_i=s_{\a_i}$ is the Weyl group of $G$, denoted by $W(G)$.

We use the notation $U_{\a}=\{x_{\a}(c) \mid c \in K\}$ to denote the $T$-root subgroup of $G$ corresponding  to $\a \in \Sigma(G)$, and we write $\mathcal{L}(G)$ for the Lie algebra of $G$ (with Lie bracket $[\, , \,]$). For each positive root $\a \in \Sigma^+(G)$ we fix $e_{\a} \in \mathcal{L}(U_{\a})$ and $f_{\a} \in \mathcal{L}(U_{-\a})$ such that 
$$\mathcal{L}(\la U_{\a},U_{-\a} \ra) = {\rm span}_{K}\{e_{\a},f_{\a},[e_{\a},f_{\a}] \} \cong \mathfrak{sl}_{2}(K).$$ 

If $H$ is a closed subgroup of $G$ and $T_{H^0}$ is a maximal torus of $H^0$ contained in $T$ then we abuse notation by writing $\l|_{H^0}$ to denote the restriction of $\l \in X(T)$ to the subtorus $T_{H^0}$. We will write $\mathbb{N}_0 = \mathbb{N}\cup\{0\}$ for the set of non-negative integers. Finally, recall that we adopt the standard labelling of simple roots (and corresponding fundamental dominant weights) given in Bourbaki \cite{Bour}.

\section{Weights and multiplicities}\label{ss:wm}

Let $V$ be a finite-dimensional $KG$-module. The action of $T$ on $V$ can be diagonalized, giving a decomposition
$$V = \bigoplus_{\mu \in X(T)}{V_{\mu}},$$
where
$$V_{\mu}=\{v \in V \mid t\cdot v = \mu(t)v \mbox{ for all $t \in T$}\}.$$
A character $\mu \in X(T)$ with $V_{\mu} \neq 0$ is called a weight (or $T$-weight) of $V$, and $V_{\mu}$ is its corresponding weight space. The dimension of $V_{\mu}$, denoted by $m_{V}(\mu)$, is called the multiplicity of $\mu$. We write $\L(V)$ for the set of weights of $V$. For any weight $\mu \in \L(V)$ and any root $\a \in \Sigma(G)$ we have 
\begin{equation}\label{e:mt}
U_{\a}v \subseteq v +\sum_{m \in \mathbb{N}}V_{\mu+m\a}
\end{equation}
for all $v \in V_{\mu}$ (see \cite[Lemma 15.4]{MT}, for example). There is a natural action of the Weyl group $W(G)$ on $X(T)$, which in turn induces an action on $\L(V)$. In particular, $\L(V)$ is a union of $W(G)$-orbits, and all weights in a $W(G)$-orbit have the same multiplicity.

By the Lie-Kolchin theorem, the Borel subgroup $B$ stabilizes a $1$-dimensional subspace $\la v^{+} \ra$ of $V$, and the action of $B$ on $\la v^{+} \ra$ affords a homomorphism $\chi:B \to K^*$ with kernel $U$. Therefore $\chi$ can be identified with a character $\l \in X(T)$, which is a weight of $V$. If $V$ is an irreducible $KG$-module then $V=\la Gv^{+} \ra$, $m_{V}(\l)=1$ and each 
weight $\mu \in \L(V)$ is obtained from $\l$ by subtracting some positive roots. Consequently, we say that $\l$ is the highest weight of $V$, and $v^+$ is a maximal vector. 

Since $G$ is simply connected, the fundamental dominant weights form a $\Z$-basis for the additive group of all weights for $T$, and a weight $\l$ is said to be dominant if $\l=\sum_{i=1}^{n}a_i\l_i$ and each $a_i$ is a non-negative integer. If $V$ is a finite-dimensional irreducible $KG$-module then its highest weight is dominant. Conversely, given any dominant weight $\l$ one can construct a finite-dimensional irreducible $KG$-module with highest weight $\l$. Moreover, this correspondence defines a bijection between the set of dominant weights of $G$ and the set of isomorphism classes of finite-dimensional irreducible $KG$-modules. For a dominant weight $\l=\sum_{i=1}^{n}a_i\l_i$ we write $L_G(\l)$ (or just $L(\l)$) for the unique irreducible $KG$-module with highest weight $\l$, and $W_G(\l)$ denotes the corresponding Weyl module (recall that $W_G(\l)$ has a unique maximal submodule $M$ such that $W_G(\l)/M \cong L_G(\l)$, and $M$ is trivial if $p=0$). In general, there is no known formula for $\dim L_G(\l)$, but $\dim W_G(\l)$ is given by \emph{Weyl's dimension formula}
$$\dim W_G(\l) = \frac{\Pi_{\a \in \Sigma^+(G)}(\a,\l+\rho)}{\Pi_{\a \in \Sigma^+(G)}(\a,\rho)},$$
where $\rho = \frac{1}{2}\sum_{\a\in \Sigma^+(G)}\a$.
In addition, we say that $L_G(\l)$ is $p$-restricted if $a_i<p$ for all $i$. By a slight abuse of terminology, it is convenient to say that every dominant weight is $p$-restricted when $p=0$. 

Suppose $p>0$. The Frobenius automorphism $F_p:K \to K$, $c \mapsto c^p$, of $K$ induces an endomorphism $F:G \to G$ of algebraic groups defined by $x_{\a}(c) \mapsto x_{\a}(c^p)$, for all $\a \in \Sigma(G)$, $c \in K$. Given a rational representation $\rho:G \to {\rm GL}(V)$ 
and an integer $i \ge 1$, we can use $F$ to define a new rational representation $\rho^{(p^i)}$ on $V$; the corresponding $KG$-module is denoted by $V^{(p^i)}$, and the action is given by $\rho^{(p^i)}(g)v = \rho(F^i(g))v$ for $g \in G, v \in V$. We say that $V^{(p^i)}$ is a Frobenius twist of $V$. 

By Steinberg's tensor product theorem, every irreducible $KG$-module is a tensor product of Frobenius twists of $p$-restricted $KG$-modules, so naturally we focus on the $p$-restricted modules.  For a detailed account of the representation theory of algebraic groups, we refer the reader to \cite{Jantzen}.

We begin by recording some results on the existence and multiplicity of certain weights in various $KG$-modules. The first result, known as \emph{Freudenthal's formula}, provides an effective recursive algorithm for calculating the multiplicity of weights in a Weyl module $V=W_G(\l)$, starting from the fact that $m_{V}(\l)=1$. See \cite[\S 22.3]{Hu1} for a proof.

\begin{thm}\label{t:freud}
Let $V=W_G(\l)$, where $\l$ is a dominant weight for $T$, and let $\mu$ be a weight of $V$. Then $m_{V}(\mu)$ is given recursively as follows:
$$((\l+\rho,\l+\rho)-(\mu+\rho,\mu+\rho)) \cdot m_{V}(\mu) = 2\sum_{\a \in \Sigma^+(G)}\sum_{i \in \mathbb{N}}(\mu+i\a,\a)\cdot m_{V}(\mu+i\a),$$ 
where $\rho = \frac{1}{2}\sum_{\a\in \Sigma^+(G)}\a$.
\end{thm}

For the remainder of Section \ref{ss:wm}, let  $V$ be an irreducible $p$-restricted $KG$-module with highest weight $\l = \sum_{i=1}^na_{i}\l_{i}$. Let $e(G)$ be the maximum of the squares of the ratios of the lengths of the roots in $\Sigma(G)$.

\begin{lem}\label{l:t1}
If $a_{i} \neq 0$ then $\mu=\l-d\a_{i} \in \L(V)$ for all $1 \le d \le a_{i}$. Moreover $m_{V}(\mu)=1$.
\end{lem}

\begin{proof}
This follows from \cite[1.30]{Test1}.
\end{proof}

Recall that a weight $\mu = \l-\sum_{i}{c_{i}\a_{i}} \in \L(V)$ is \emph{subdominant} to $\l$ if $\mu$ is a dominant weight, that is, $\mu = \sum_{i}d_i\l_i$ with $d_i \ge 0$ for all $i$.

\begin{lem}\label{l:pr}
Let $\mu$ be a weight of the Weyl module $W_G(\l)$, and assume that $p=0$ or $p>e(G)$. Then $\mu \in \L(V)$. In particular, if $\mu = \l-\sum_{i=1}^{n}{c_{i}\a_{i}}$ is a subdominant weight, then $\mu \in \L(V)$.
\end{lem}

\begin{proof}
This follows from \cite[Theorem 1]{Pr}.
\end{proof}

\begin{cor}\label{c:sat}
Suppose $p=0$ or $p>e(G)$. If $\mu \in \L(V)$ then $\mu-k\a \in \L(V)$ for all $\a \in \Sigma^{+}(G)$ and all integers $k$ in the range $0 \le k \le \la \mu, \a\ra$.
\end{cor}

\begin{proof}
The set of weights of the Weyl module $W_G(\l)$ is saturated (see \cite[Section 13.4]{Hu1}), so the result follows from Lemma \ref{l:pr}.
\end{proof}

\begin{lem}\label{l:118}
Suppose $\a,\b \in \Pi(G)$ with $(\a,\b)<0$ and $\la \l,\a\ra=c$, $\la \l,\b\ra=d$ for $c,d>0$.
Set $m=m_{V}(\l-\a-\b)$. Then $m \in \{1,2\}$ and the following hold:
\begin{itemize}\addtolength{\itemsep}{0.3\baselineskip}
\item[{\rm (i)}] If $(\a,\a)=(\b,\b)$ then $m=1$ if and only if $c+d=p-1$.
\item[{\rm (ii)}] If $(\a,\a)=2(\b,\b)$ then $m=1$ if and only if $2c+d+2 \equiv 0 \imod{p}$.
\item[{\rm (iii)}] If $(\a,\a)=3(\b,\b)$ then $m=1$ if and only if $3c+d+3 \equiv 0 \imod{p}$.
\end{itemize}
In particular, $m=2$ if $p=0$.
\end{lem}

\begin{proof}
This follows from Theorem \ref{t:freud} and the final proposition of \cite{Bur}.
\end{proof}

\begin{lem}\label{l:s816}
Suppose $G=A_n$ and $\l=a\l_i+b\l_j$, where $n \ge 2$, $i<j$ and $a,b>0$. If $1 \le r \le i$ and $j \le s \le n$ then $\mu=\l-(\a_r+ \cdots +\a_s) \in \L(V)$ and
$$m_{V}(\mu)=\left\{\begin{array}{ll}
j-i & \mbox{if $a+b+j-i \equiv 0 \imod{p}$} \\
j-i+1 & \mbox{otherwise.}
\end{array}\right.$$
In particular, if $p=0$ then $m_{V}(\mu)=j-i+1$.
\end{lem}

\begin{proof}
This is \cite[8.6]{Seitz2}.
\end{proof}

\begin{lem}\label{l:bwt}
Suppose $G=B_n$ and $\l=\l_1+\l_n$, where $n\geq2$ and  $p\neq 2$. 
Let $\mu=\l-\a_1-\cdots-\a_n = \l_n$. 
Then 
$$m_{V}(\mu)=\left\{\begin{array}{ll}
n-1 & \mbox{if $p \mid 2n+1$} \\
n & \mbox{otherwise.}
\end{array}\right.$$
In particular, $m_V(\mu)=n$ if $p=0$.
\end{lem}

\begin{proof}
The result quickly follows from Lemma \ref{l:118} if $n=2$, so we will assume $n\geq 3$.
Since $\lambda$ is a $p$-restricted weight, $V$ is also irreducible as an
$\mathcal{L}(G)$-module, where $\mathcal{L}(G)$ denotes the Lie algebra of $G$. Let $v^+ \in V$ be a maximal vector with respect to the 
standard Borel subgroup of $G$, and recall the notation $e_{\a}, f_{\a}$ (for $\a \in \Sigma^+(G)$) defined in Section \ref{ss:notation}. Using the PBW basis of the universal enveloping algebra of
$\mathcal{L}(G)$ (see \cite[Section 17.3]{Hu1}) we easily deduce that a spanning set for the weight
space $V_{\mu}$ is given by $\{v_i \mid 0\leq i\leq  n-1\}$, where
$v_i=f_{\alpha_1+\cdots+\alpha_i}f_{\alpha_{i+1}+\cdots+\alpha_n}v^+$. 
Then using the basis for $\mathcal{L}(G)$ given in \cite[p.108]{Seitz2},
we see that
$e_{\alpha_j}\sum_{i=0}^{n-1} c_iv_i = 0$ for all $1\leq j\leq n$ if and only if
$c_1=c_2=\cdots=c_{n-1}$, $c_{n-1}=2c_0$ and $c_0+2c_1+c_2+\cdots+c_{n-1}=0$. 
Hence the weight
$\mu$ has multiplicity $n-1$ if $p$ divides $2n+1$, and multiplicity $n$ otherwise.
\end{proof}

\begin{lem}\label{l:sr}
Suppose $\mu = \l - \sum_{\a \in S}c_{\a}\a \in \L(V)$ for some subset $S \subseteq \Pi(G)$. Then $m_{V}(\mu) = m_{V'}(\mu')$, where $V'=L_{X}(\lambda|_{X})$, $\mu' = \mu|_{X}$ and $X=\la U_{\pm \a} \mid \a \in S\ra$.
\end{lem}

\begin{proof}
Let $P$ be the standard parabolic subgroup of $G$ corresponding to
the subset $S$, so that $X$ is the derived subgroup of
a Levi factor of $P$. We see that the $\mu$-weight space $V_{\mu}$ lies in the fixed
point space of the unipotent radical of $P$, which is the irreducible $KX$-module
with highest weight $\lambda|_X$,
by \cite[Proposition 2.11]{Jantzen}. The result now
follows.
\end{proof}

\begin{lem}\label{c2:l3}
Suppose $G=C_n$ and $\lambda=\lambda_{n-1}+a\lambda_n$, where $n \geq 3$, $0 \le a<p$ and $2a+3\equiv 0\imod{p}$. Then
$\mu=\lambda-\a_{n-2}-2\a_{n-1}-\a_n \in \L(V)$ and $m_{V}(\mu)=1$.
\end{lem}

\begin{proof}
It is easy to see that $\mu \in \L(V)$ by repeatedly applying Corollary \ref{c:sat}, so it remains to show that $m_{V}(\mu)=1$. By Lemma \ref{l:sr}, it suffices to consider the case $n=3$. First assume $a=0$, so $p=3$. By inspecting \cite[Table A.32]{Lubeck} we see that $\dim V = 13$, and thus $m_V(\mu)=1$ since $\l$ has $12$ distinct $W(G)$-conjugates (see \cite[1.10]{Seitz2}). 

Now assume $a>0$, so $p>3$. By applying Theorem \ref{t:freud}, we calculate that $\mu$ has multiplicity $3$ in the Weyl module $W_G(\l)$. By Lemma \ref{l:118}, $W_G(\l)$
has a $KG$-composition
factor of highest weight $\nu=\l-\a_2-\a_3$, which is dominant since 
$\nu= \l_1+\l_2+(a-1)\l_3$. Moreover, since $\mu=\nu-\a_1-\a_2$ and $p\ne3$,
Lemma \ref{l:118} implies that $\mu$ occurs with
multiplicity $2$ in this composition factor. The result follows.
\end{proof}

\begin{rmk}\label{r:sz}
\emph{Note that Lemma \ref{c2:l3} is a special case of \cite[Theorem 0.1]{SZ}, which states that $\dim V = (p^n-1)/2$ and $m_V(\mu)=1$ for all weights $\mu \in \L(V)$.}
\end{rmk}

The final lemma in this section records a trivial observation that will be used frequently in the proof of Theorem \ref{main}. 

\begin{lem}\label{l:easy}
Let $G$, $H$ and $V$ be given as in the statement of Theorem \ref{main}. Write $V=V_1 \oplus \cdots \oplus V_r$, where each $V_i$ is a $KH^0$-composition factor, and let $\mu_1, \ldots, \mu_{s} \in \L(V)$ be distinct $T$-weights with the property that $\mu_i|_{H^0} = \nu$ for all $1 \le i \le s$. Then
$$\sum_{i=1}^{s}m_V(\mu_i) \leq m_V(\nu) = \sum_{i=1}^{r}m_{V_i}(\nu).$$
\end{lem}

Here $m_V(\nu)$ denotes the multiplicity in $V$ of the $T_{H^0}$-weight $\nu$, and $m_{V_i}(\nu)$ is its multiplicity in the $KH^0$-composition factor $V_i$, where $T_{H^0}$ is a suitable maximal torus of $(H^0)'$ contained in $T$ (recall that $H^0$ is reductive; see Remark \ref{r:conds}(j)).

\section{Some dimension calculations}\label{ss:dc}

In this section, let $G$ be a classical algebraic group of rank $n$ over an algebraically closed field $K$ of characteristic $p \ge 0$, where we assume $n \ge 2$ if $G=B_n$ or $C_n$, and $n \ge 3$ if $G=D_n$. In addition, we assume $p \neq 2$ if $G = B_n$. In the proof of Theorem \ref{main} we need to compute the dimension of the irreducible $KG$-module $L_G(\l)$ for some specific $p$-restricted highest weights $\l$.
Our main result is the following.

\begin{prop}\label{p:dims}
Let $\l$ be one of the weights listed in Table \ref{t:dims}. Then $\dim L_G(\l)$ is given in the third column of the table.
\end{prop}

\renewcommand{\arraystretch}{1.2}
\begin{table}
$$\begin{array}{llll} \hline
G & \l & \hspace{4.5mm} \dim L_G(\l) & 
\\ \hline
A_n & a\l_1,\, a\l_n\, (a<p) & \hspace{4.5mm} (n+a)!/n!a! & \\
n \ge 1 & & & \\

& & & \\

\hspace{-2mm} \begin{array}{l} 
B_n \\
n \ge 2 
\end{array}    & \hspace{-2mm} \begin{array}{l}  2\l_1 \\ \mbox{  } \end{array} &	
           \left\{\begin{array}{l}
		n(2n+3)-1  \\
        	n(2n+3) \end{array}\right. 
		&  \begin{array}{l}
		\mbox{if $p\mid 2n+1$}\\
		\mbox{otherwise}
		\end{array} \\

    & \l_2 & \left\{\begin{array}{l}
		4  \\
		n(2n+1)
		\end{array}\right. 
		&  \begin{array}{l}
		\mbox{if $n=2$} \\
		\mbox{otherwise}
		\end{array} \\
		
& \l_n & \hspace{4.5mm} 2^n & \\

& \l_1+\l_n  & \left\{\begin{array}{l}
		2^n(2n-1)   \\
		2^{n+1}n 
		\end{array}\right. 
		&  \begin{array}{l}
		\mbox{if $p\mid 2n+1$} \\
		 \mbox{otherwise}
		 \end{array} \\
		
		& & & \\

\hspace{-2mm} \begin{array}{l} 
C_n \\
n \ge 2 \end{array}

& \hspace{-2mm} \begin{array}{l}  2\l_1 \\ \mbox{  } \end{array} &	\left\{\begin{array}{l}
		2n  \\
        	n(2n+1)	  
		\end{array}\right. 
		&  \begin{array}{l}
		\mbox{if $p=2$} \\
		\mbox{otherwise}
		\end{array} \\
		
    & \l_2 & \left\{\begin{array}{l}
		(n-1)(2n+1)-1   \\
		(n-1)(2n+1) 
		\end{array}\right. 
		&  \begin{array}{l}
		\mbox{if $p\mid n$} \\
		\mbox{otherwise}
		\end{array} \\

& & & \\

\hspace{-2mm} \begin{array}{l}  
D_n \\
n \ge 3 \end{array}

& \hspace{-2mm} \begin{array}{l} 2\l_1 \\ \mbox{ } \end{array} & \left\{\begin{array}{l}
       		2n  \\
         	(n+1)(2n-1)-1   \\
		(n+1)(2n-1) 
		\end{array}\right. 
		&  \begin{array}{l}
		\mbox{if $p=2$} \\
		\mbox{if $p\ne2$ and $p\mid n$} \\
		\mbox{otherwise}
		\end{array} \\
		
& \l_2 \; (n \ge 4) & \left\{\begin{array}{l}
		n(2n-1)-2   \\
		n(2n-1)-1   \\
		n(2n-1) 
		\end{array}\right. 
		&  \begin{array}{l}
		\mbox{if $p=2$ and $n$ is even} \\
		\mbox{if $p=2$ and $n$ is odd} \\
		\mbox{otherwise}
		\end{array} \\
		
& \l_{n-1},\, \l_n & \hspace{5mm} 2^{n-1} \\  
  
 & 2\l_{n-1},\, 2\l_n & \left\{\begin{array}{l}
		2^{n-1}    \\
		\frac{1}{2}\binom{2n}{n} 
		\end{array}\right. 
		&  \begin{array}{l}
		 \mbox{if $p=2$} \\
		 \mbox{otherwise}
		 \end{array} \\
  
    & \l_1+\l_{n-1},\, \l_1+\l_{n}  & \left\{\begin{array}{l}
		2^{n}(n-1)    \\
		2^{n-1}(2n-1) 
		\end{array}\right. 
		&  \begin{array}{l}
		 \mbox{if $p \mid n$} \\
		 \mbox{otherwise}
		 \end{array} \\ \hline
\end{array}$$
\caption{The dimensions of some irreducible $KG$-modules}
\label{t:dims}
\end{table}
\renewcommand{\arraystretch}{1}

We partition the proof of Proposition \ref{p:dims} into a series of separate lemmas. Note that $\dim L_G(\l)$ has been computed by L\"{u}beck \cite[Theorems 4.4, 5.1]{Lubeck} in the cases $\l = 2\l_1$ or $\l_2$ (for every classical group $G$), so it remains to deal with the other cases in Table \ref{t:dims}. 

Recall that we can define a partial order on the set of weights for $T$: if $\mu,\l$ are weights then $\mu \preccurlyeq \l$ if and only if $\mu = \l - \sum_{i=1}^{n}c_i\a_i$ and each $c_i$ is a non-negative integer. In this situation, we say that $\mu$ is \emph{under} $\l$. Then following \cite[p.72]{Hu1}, a dominant weight $\l$ for $T$ is \emph{minimal} if for all dominant weights $\mu$, we have $\mu \preccurlyeq \l$ if and only if 
$\mu=\l$. It is easy to verify that the non-zero minimal weights of the classical irreducible root systems are as follows:
\begin{equation}\label{e:mini}
A_{n}: \l_1, \ldots, \l_n, \; B_n: \l_n, \; C_n: \l_1, \; D_n: \l_1, \l_{n-1},\l_{n}.
\end{equation}
The next result is well-known, but we provide a proof for completeness.

\begin{lem}\label{l:dmspin}
$\dim{L_{B_n}(\l_n)}=2^n$ and $\dim{L_{D_n}(\l_{n-1})}=\dim{L_{D_n}(\l_{n})}=2^{n-1}$.
\end{lem}

\begin{proof}
As $\l_n$ is a minimal weight for $B_n$ (see \eqref{e:mini}), all weights in the Weyl module
 $W_{B_n}(\l_n)$ are conjugate to $\l_n$, and so they occur in the spin module $L_{B_n}(\l_n)$ 
in any characteristic.
Therefore, by applying \cite[1.10]{Seitz2} we deduce that $\dim{L_{B_n}(\l_n)}=|W(B_n)|/|W(A_{n-1})|=2^n$, and similarly
$$\dim{L_{D_n}(\l_{n-1})}=\dim{L_{D_n}(\l_n)}=|W(D_n)|/|W(A_{n-1})|=2^{n-1}.$$
\end{proof}

\begin{lem}\label{l:s114}
Suppose $G=A_n$ and $\l=a\l_1$ or $a\l_n$, where $a>0$ and either $p=0$ or $a<p$. Then $m_V(\mu) = 1$ for all $\mu \in \L(V)$, and $\dim L_G(\l)=(n+a)!/n!a!$.
\end{lem}

\begin{proof}
See \cite[1.14]{Seitz2}.
\end{proof}

\begin{lem}\label{l:bn2}
Let $G=B_n$ with $n \ge 2$ and $p \neq 2$. Then 
$$\dim L_{G}(\l_1+\l_n) = \left\{\begin{array}{ll}
		2^n(2n-1) & \mbox{if $p\mid 2n+1$} \\
		2^{n+1}n & \mbox{otherwise.}
		\end{array}\right.$$
\end{lem}

\begin{proof} 
Set $\lambda =\lambda_1+\lambda_n$ and $V= L_{G}(\lambda)$.
There is a unique subdominant weight of $V$, namely $\lambda_n=\lambda-\alpha_1-\alpha_2-\cdots-\alpha_n$, and the multiplicity of this weight is given in Lemma~\ref{l:bwt}.
By counting the $W(G)$-conjugates of $\lambda$ and $\lambda_n$, using \cite[1.10]{Seitz2}, we obtain the dimension
of $V$ as claimed.
\end{proof}

\begin{lem}\label{l:dm}
Let $G=D_n$ with $n \ge 3$. Then
$$\dim L_G(\l_1+\l_{n-1}) = \dim L_G(\l_1+\l_{n}) = \left\{\begin{array}{ll}
2^{n}(n-1) & \mbox{if $p \mid n$} \\
2^{n-1}(2n-1) & \mbox{otherwise.}
\end{array}\right.$$
\end{lem}

\begin{proof}
It suffices to compute $\dim L_G(\l)$, where $\l = \l_{1}+\l_{n}$. We consider the tensor product $M:=L_{G}(\l_1)\otimes L_{G}(\l_n)$,
which has a unique composition factor isomorphic to
$L_{G}(\l)$. One checks that the only subdominant weight of $L_{G}(\l)$ is
$\l_{n-1}=\l-\a_1-\a_2-\cdots-\a_{n-2}-\a_n$. Thus, any other
composition factor of $M$ is isomorphic to $L_{G}(\l_{n-1})$.
Both $\l_1$ and $\l_n$ are minimal weights (see \eqref{e:mini}) and so all
weight spaces in the corresponding irreducible modules
are $1$-dimensional, and Lemma \ref{l:dmspin} implies that $\dim M=2n(2^{n-1})$. In particular, if we set 
$$\mu_k = \l_1 - \sum_{i=1}^{k}\a_i,\;\; \nu_k = \l_n - \a_n - \sum_{i=k+1}^{n-2}\a_i,\;\;\mu_{n-1} = \l_{1} - \a_n - \sum_{i=1}^{n-2}\a_i,\;\;\nu_{n-1} = \l_n$$ 
where $0 \le k \le n-2$, then $\mu_k \in \L(L_G(\l_1))$, $\nu_k \in \L(L_G(\l_n))$ and 
$$\l-\a_1-\cdots-\a_{n-2}-\a_n= \mu_k + \nu_k$$
for all $0 \le k \le n-1$. It follows that $\l-\a_1-\cdots-\a_{n-2}-\a_n$ occurs with multiplicity $n$ in $M$.

It remains to determine the multiplicity of the weight $\l_{n-1}$ in the irreducible module
$L_{G}(\l)$. This is the same as the multiplicity of the zero weight of the irreducible $KA_{n-1}$-module $L_{A_{n-1}}(\l_1+\l_{n-1})$. By Lemma \ref{l:s816}, the zero weight has multiplicity $n-2$ if $p$ divides $n$, otherwise it is $n-1$. It follows that $M$ has two composition factors isomorphic to $L_{G}(\l_{n-1})$ if $p$ divides $n$, otherwise there is only one such factor. 
Therefore
$$\dim M=2n(2^{n-1})=\dim L_{G}(\l)+\epsilon\dim L_{G}(\l_{n-1}),$$
where $\epsilon = 2$ if $p$ divides $n$, otherwise $\epsilon =1$. So
$2n(2^{n-1})=\dim L_{G}(\l) +\epsilon 2^{n-1}$ and the result follows. 
\end{proof}

\begin{lem}\label{l:dmdim}
Let $G=D_n$ with $n \ge 3$. Then 
$$\dim L_G(2\l_{n-1}) = \dim L_G(2\l_{n}) = \left\{\begin{array}{ll}
		2^{n-1}  & \mbox{if $p=2$}  \\
		\frac{1}{2}\binom{2n}{n} & \mbox{otherwise.}
		\end{array}\right.$$
\end{lem}

\begin{proof}
Set $\l=2\l_n$ and $V = L_G(\l)$. If $p=2$ then $V=L_G(\l_n)^{(2)}$, a Frobenius twist of $L_G(\l_n)$, and thus $\dim V =\dim L_G(\l_n)=2^{n-1}$ by Lemma \ref{l:dmspin}. Now assume $p \neq 2$. By considering the corresponding Weyl module $W_G(2\l_n)$, we obtain the upper bound
$$\dim V \leq \dim W_G(2\l_n) = \frac{1}{2}\binom{2n}{n}.$$ 
(See \cite[Exercise 24.43]{Fulton} for more details.)
 
We establish equality by induction on $n$. By inspecting \cite[Table A.7]{Lubeck} we see that $\dim L_{D_3}(2\l_3)=10=\frac{1}{2}\binom{6}{3}$, so let us assume $n>3$.
Let $P=QL$ be the parabolic subgroup of $G$ with $\Pi(L')=\{\a_2,\ldots, \a_n\}$. If $\mu = \l - \sum_{i}b_i\a_i$ is a weight of $V$ then we define the \emph{$Q$-level} of $\mu$ to be the coefficient $b_1$, and then the $i$-th $Q$-level of $V$ is the sum of the weight spaces $V_{\mu}$ for weights $\mu$ with $Q$-level $i$. It is easy to see that $V$ has precisely three $Q$-levels, say $V_0$, $V_1$ and $V_2$, of respective levels 0, 1 and 2. Note that each $V_i$ is a $KL'$-module.  Now $V_i$ has a $KL'$-composition factor with highest weight $\l|_{L'}$, $(\l-\sum_{i=1}^{n-2}\a_i-\a_n)|_{L'}$, respectively $(\l-2\sum_{i=1}^{n-2}\a_i-2\a_{n})|_{L'}$, for $i=0,1$ respectively $2$. Denote these composition factors by $U_0$, $U_1$ and $U_2$, respectively. Now $\l|_{L'}=2\l_n|_{L'}$,  $(\l-\sum_{i=1}^{n-2}\a_i-\a_n)|_{L'}=(\l_{n-1}+\l_n)|_{L'}$ and $(\l-2\sum_{i=1}^{n-2}\a_i-2\a_n)|_{L'}=2\l_{n-1}|_{L'}$. By the inductive hypothesis, we have $\dim U_0=\dim U_2=\frac{1}{2}\binom{2n-2}{n-1}$. By inspecting \cite[Table 1]{Seitz2} we deduce that $\dim U_1=\dim L_{A_{2n-3}}(\lambda_{n-2})$ (see the case labelled ${\rm I}_{5}$), and so $\dim U_1=\binom{2n-2}{n-2}$. We conclude that  
$$\dim V \geq \sum_{i=0}^{2}\dim U_i = \binom{2n-2}{n-1}+\binom{2n-2}{n-2} = \frac{1}{2}\binom{2n}{n}.$$
\end{proof}

This completes the proof of Proposition \ref{p:dims}.

\section{Clifford theory}\label{ss:clifford}

Let $G,H$ and $V$ be given as in the statement of Theorem \ref{main}. Set $X=H^0$. If $V|_{X}$ is irreducible then the possibilities for $(G,H,V)$ are easily deduced from Seitz's main theorem in \cite{Seitz2}, so let us assume $V|_{X}$ is reducible and $V|_{H}$ is irreducible. Then 
Clifford theory implies that 
\begin{equation}\label{e:vx}
V|_{X}=V_1 \oplus \cdots \oplus V_m,
\end{equation}
where $m$ divides the order of $H/X$, and the $V_i$ are transitively permuted under the induced action of $H/X$. We will need the following result (see \cite[Proposition 2.6.2]{BGMT}).

\begin{prop}\label{p:niso}
If $H$ is a cyclic extension of $X$ then the irreducible $KX$-modules $V_i$ in \eqref{e:vx} are pairwise non-isomorphic.
\end{prop}

\section{Subgroup structure}\label{ss:cla}

Let $W$ be a finite-dimensional vector space over an algebraically closed field $K$ of characteristic $p \ge 0$ with $\dim W \ge 2$. Let $G= {\rm Isom}(W)'$, where ${\rm Isom}(W)$ is the full isometry group of a suitable form $f$ on $W$. Here we assume $f$ is either the zero bilinear form, a symplectic form or a non-degenerate quadratic form, so $G$ is one of the groups ${\rm SL}(W)$, ${\rm Sp}(W)$ or ${\rm SO}(W)$, respectively. In the latter case, we will assume that $\dim W \ge 3$ and $\dim W \neq 4$, so $G$ is a simple algebraic group.
Let $n$ denote the rank of $G$, so $G$ is of type $A_n$, $B_n$, $C_n$ or $D_n$ in terms of the usual Lie notation. We write ${\rm GO}(W) = {\rm Isom}(W)$ when $f$ is quadratic; here ${\rm GO}(W)$ is a split extension ${\rm SO}(W)\la \tau \ra$, where $\tau$ acts on $W$ as a reflection when $p \neq  2$, and as a transvection when $p=2$. In particular, if $x \in {\rm GO}(W)$ is an involution and $p \neq 2$ then $x \in {\rm SO}(W)$ if and only if $\det(x)=1$; the analogous criterion when $p=2$ is that the Jordan normal form of $x$ on $W$ comprises an even number of unipotent Jordan blocks of size $2$ (see \cite[Section 8]{AS}).

Following \cite[Section 1]{LS}, we introduce six natural, or \emph{geometric}, collections of closed subgroups of $G$, labelled $\C_i$ for $1 \le i \le 6$, and we set $\C=\bigcup_{i}\C_i$. A rough description of the subgroups in each $\C_i$ collection is given in Table \ref{t:subs}. The main theorem of \cite{LS} provides the following description of the maximal closed subgroups of $G$.

\renewcommand{\arraystretch}{1.1}
\begin{table}
\begin{tabular}{cl} \hline
 & Rough description \\ \hline
$\C_1$ & Stabilizers of subspaces of $W$ \\
$\C_2$ & Stabilizers of orthogonal decompositions $W=\bigoplus_{i}W_i$, $\dim W_i=a$ \\
$\C_3$ & Stabilizers of totally singular decompositions $W=W_1 \oplus W_2$ \\
$\C_4$ & Stabilizers of tensor product decompositions $W=W_1 \otimes W_2$ \\
& Stabilizers of tensor product decompositions $W=\bigotimes_i W_i$, $\dim W_i=a$ \\
$\C_5$ & Normalizers of symplectic-type $r$-groups, $r \neq p $ prime  \\
$\C_6$ & Classical subgroups \\ \hline
\end{tabular}
\caption{The $\C_i$ collections}
\label{t:subs}
\end{table}
\renewcommand{\arraystretch}{1}

\begin{thm}\label{t:ls}
Let $H$ be a closed subgroup of $G$. Then one of the following holds:
\begin{itemize}\addtolength{\itemsep}{0.3\baselineskip}
\item[{\rm (i)}] $H$ is contained in a member of $\C$;
\item[{\rm (ii)}] modulo scalars, $H$ is almost simple and $E(H)$ (the unique quasisimple normal subgroup of $H$) is irreducible on $W$. Further, if $G={\rm SL}(W)$ then $E(H)$ does not fix a non-degenerate form on $W$. In addition, if $H$ is infinite then $E(H)=H^0$ is tensor-indecomposable on $W$.
\end{itemize}
\end{thm}

\begin{proof}
This is \cite[Theorem 1]{LS}.
\end{proof}

We use the symbol $\mathcal{S}$ to denote the collection of maximal closed subgroups of $G$ that arise in part (ii) of Theorem \ref{t:ls}. In this paper we are interested in the maximal positive-dimensional subgroups $H \in \C$ such that $H/Z(G)$ is disconnected. In \cite{BGMT} we deal with the disconnected positive-dimensional subgroups in $\mathcal{S}$, using completely different methods.

For the remainder of this section we focus on the structure of the maximal, 
positive-dimensional subgroups $H$ in $\C$ with the property that $H/Z(G)$ is disconnected. Note that the local subgroups comprising the $\C_5$ family are finite, so we can discard this collection. Also recall that we may assume $H^0$ is reductive (see Remark \ref{r:conds}(j)). 
%We may also assume that $H^0$ is reductive. Indeed, if $V$ is an irreducible $KG$-module and  $V|_{H}$ is irreducible then the unipotent radical of $H^0$ acts completely reducibly on $V$, which implies that it acts trivially on $V$. 
We consider each of the remaining $\C_i$ families in turn, starting with $\C_1$. As in the statement of Theorem \ref{main}, we will assume that $\dim W$ is even if $G={\rm SO}(W)$ and $p=2$.

\subsection{Class $\C_1$: Subspace subgroups}\label{ss:c1}

Here $H=G_{U}$ is the $G$-stabilizer of a proper non-zero subspace $U$ of $W$. Moreover, the maximality of $H$ implies that $U$ is either totally singular or non-degenerate, or $(G,p)=({\rm SO}(W),2)$ and $U$ is a $1$-dimensional non-singular subspace, with respect to the underlying form $f$ on $W$. If $U$ is totally singular then 
$H$ is a maximal parabolic subgroup of $G$, which is connected, so we may assume otherwise. In particular, we have $G={\rm Sp}(W)$ or ${\rm SO}(W)$. Now, if $(G,p)=({\rm SO}(W),2)$ and $U$ is $1$-dimensional and non-singular then $H$ is a connected group of type $B_{n-1}$, so we can assume $U$ is non-degenerate. Note that if $\dim U = \frac{1}{2}\dim W$ then $H$ is properly contained in a $\C_{2}$-subgroup.

If $G={\rm Sp}(W)$ then $H={\rm Sp}(U) \times {\rm Sp}(U^{\perp})$ is connected, so we reduce to the case $G={\rm SO}(W)$. First assume $G$ is of type $B_n$, so $\dim W = 2n+1$ and $p \neq 2$. Since $H=G_{U}=G_{U^{\perp}}$, we may assume $\dim U=2l$ for some $l \ge 1$. If $l=n$ then $H={\rm SO}(U)\langle z \rangle=D_{n}.2$, where $z$ acts as a reflection on $U$. 
Similarly, if $l<n$ then $H=({\rm SO}(U) \times {\rm SO}(U^{\perp}))\langle z \rangle=D_{l}B_{n-l}.2$ is disconnected, where $z$ acts as a reflection on both $U$ and $U^{\perp}$.

Finally, suppose $G$ is of type $D_n$. If $\dim U=2l+1$ is odd then $p \neq 2$ and $H/Z(G)$ is connected. For example, if $l=n-1$ then $H={\rm SO}(U) \times \langle z \rangle$, where $z$ is a central involution, so $H/Z(G)$ is connected of type $B_{n-1}$. On the other hand, if $\dim U=2l$ is even then $H=({\rm SO}(U) \times {\rm SO}(U^{\perp}))\langle z \rangle$, where $z$ acts on both $U$ and $U^{\perp}$ as a reflection if $p \neq 2$ and as a transvection if $p=2$, so $H$ is disconnected of type $D_{l}D_{n-l}.2$. Note that $H=G_{U} = G_{U^{\perp}}$, so we may assume that $l<n/2$.

\subsection{Class $\C_2$: Imprimitive subgroups}\label{ss:c2}

Suppose $W$ admits a direct sum decomposition of the form 
$$W=W_{1}\oplus W_2 \oplus \cdots \oplus W_t,$$ 
where $t>1$ and $\dim W_1 = \dim W_i$ for all $i$. If $G$ is symplectic or orthogonal then assume in addition that the $W_i$ are non-degenerate and pairwise orthogonal with respect to the underlying non-degenerate form $f$ on $W$. A $\C_2$-subgroup of $G$ is the $G$-stabilizer of such a decomposition; such a subgroup has the structure 
$$H = ({\rm Isom}(W_{1}) \wr S_{t}) \cap G,$$ 
where ${\rm Isom}(W_{1})$ denotes the full isometry group of the restriction of the form $f$ to $W_{1}$. 

Suppose $G$ is of type $A_{n}$ and $\dim W_{1}=l+1$, where $l \ge 0$, so $n+1=(l+1)t$. If $l=0$ then $H=({\rm GL}(W_{1}) \wr S_{t}) \cap G=T_{n}.S_{n+1} = N_{G}(T_n)$, where $T_n$ denotes an $n$-dimensional torus in $G$. Similarly, if $l >0$ then 
$H=((A_{l}T_{1})^{t}.S_{t}) \cap G = A_{l}^{t}T_{t-1}.S_{t}$. Also, if $G$ is of type $C_{n}$ then $\dim W_{1}=2l$ is even and $H=C_{l}^{t}.S_{t}$.

Next suppose $G$ is of type $B_{n}$, so $p \neq 2$, $t$ is odd and $\dim W_{1}=2l+1$ with $l \ge 1$ (if $l=0$ then $H$ is finite). Let $J = \langle z_{i,j} \mid 1 \le i<j \le t \rangle$ be the subgroup of $({\rm GO}(W_1) \times \cdots \times {\rm GO}(W_t)) \cap {\rm SO}(W)$ generated by the elements $z_{i,j}=(x_{1}, \ldots, x_{t})$, where $x_i=x_j=-1$ and $x_k=1$ for all $k \neq i,j$. Then $J$ is elementary abelian of order $2^{t-1}$ and $H=
(J \times B_{l}^{t}).S_{t}=(2^{t-1} \times B_{l}^{t}).S_{t}$. Finally, let us assume $G$ is of type $D_n$. If  
$\dim W_{1}=2l+1$ with $l \ge 1$ then $p \neq 2$, $t$ is even and $H=(2^{t-1} \times B_{l}^{t}).S_{t}$, as before. Now suppose $\dim W_{1}=2l$ is even. Write ${\rm GO}(W_{i})={\rm SO}(W_{i})\langle z_{i} \rangle$ for a suitable involution $z_i$, and let $J=\langle z_{i,j} \mid 1 \le i<j \le t \rangle$, where $z_{i,j}=(x_{1}, \ldots, x_{t}) \in ({\rm GO}(W_1) \times \cdots \times {\rm GO}(W_t))\cap {\rm SO}(W)$ with $x_{i}=z_i$, $x_{j}=z_{j}$ and $x_{k}=1$ for all $k \neq i,j$. Then $J$ is elementary abelian of order $2^{t-1}$ and 
$H=(D_{l}^{t}.J).S_{t}=(D_{l}^{t}.2^{t-1}).S_{t}$.

\subsection{Class $\C_3$: Stabilizers of totally singular decompositions}\label{ss:c3}

Here $G={\rm Sp}(W)$ or ${\rm SO}(W)$, and $W$ is even-dimensional. A $\C_3$-subgroup of $G$ is the $G$-stabilizer of a decomposition of the form $W=U \oplus U'$, where $U$ and $U'$ are maximal totally singular subspaces of $W$.

First suppose $G$ is of type $C_{n}$. Then $H={\rm GL}(U)\langle \tau \rangle=A_{n-1}T_{1}.2$ is disconnected, where the action of $\tau$ on $W$ interchanges the subspaces $U$ and $U'$. In particular, $\tau$ induces an involutory graph automorphism on $A_{n-1}$, and inverts the $1$-dimensional torus $T_1$. If $p=2$ then it is easy to see that $H$ is contained in a $\C_6$-subgroup of type ${\rm GO}(W)$ (see \cite[p.101]{KL}, for example), so we require the condition $p \neq 2$ for $H$ to be maximal.

Now assume $G=D_{n}$. If $n$ is odd then \cite[Lemma 2.5.8]{KL} implies that $H=G_{U} \cap G_{U'}={\rm GL}(U)=A_{n-1}T_{1}$ is connected (and non-maximal). Therefore  we may assume $n$ is even, in which case  $H={\rm GL}(U)\langle\tau\rangle=A_{n-1}T_{1}.2$, where $\tau$ is defined as before.

\subsection{Class $\C_4$: Tensor product subgroups}\label{ss:c4}

We partition the subgroups in $\C_4$ into two subcollections, labelled $\C_4(i)$ and $\C_4(ii)$, as indicated by the description of the $\C_4$-subgroups presented in Table \ref{t:subs}. 

First suppose $H \in \C_{4}(i)$. Here $H$ stabilizes a tensor product decomposition of the form $W=W_{1} \otimes W_{2}$, where $\dim W_{i}>1$, and thus $H=N_{G}(Cl(W_{1}) \otimes Cl(W_{2}))$ for some specific classical groups $Cl(W_1) \not\cong Cl(W_2)$. As explained in \cite[Section 1]{LS}, the possibilities for the central product $Cl(W_{1}) \otimes Cl(W_{2})$ when $H$ is maximal are as follows:
$${\rm SL} \otimes {\rm SL} < {\rm SL},\; {\rm Sp} \otimes {\rm SO} < {\rm Sp}\;(p \neq 2),$$
$${\rm Sp} \otimes {\rm Sp} < {\rm SO},\; {\rm SO} \otimes {\rm SO} < {\rm SO}\; (p \neq 2).$$

If $G$ is of type $A_{n}$ then $H$ is of type ${\rm SL}(W_1) \otimes {\rm SL}(W_2)$, where $\dim W_1=a$ and $\dim W_2=b$ for some integers $a,b$ with $a>b>1$ and $n+1=ab$. In particular, $H$ is connected of type $A_{a-1}A_{b-1}T_{1}$. Similarly, we find that $H$ is connected if $G$ is of type $B_n$. Next suppose $G$ is of type $C_n$, so $p \neq 2$ and
$H = N_G({\rm Sp}(W_1) \otimes {\rm SO}(W_2))$. If $\dim W_2$ is odd then $H$ is connected, so assume $\dim W_1=2a$ and $\dim W_2=2b$. Here $H=C_{a}D_{b}\langle z \rangle=C_{a}D_{b}.2$ is disconnected, where $z$ is an involution that centralizes $W_1$ and induces a reflection on $W_2$. Moreover, we may assume $\dim W_2 \ge 4$ (so that $b \ge 2$) since $H$ is contained in a $\C_2$-subgroup when $b=1$ (see \cite[Proposition 4.4.4]{KL}, for example).

Finally, let us assume $G$ is of type $D_n$. There are three cases to consider. If $H=N_G({\rm Sp}(W_1) \otimes {\rm Sp}(W_2))$ then $\dim W_1 \neq \dim W_2$ and $H$ is connected, so let us assume $p \neq 2$ and $H=N_G({\rm SO}(W_1) \otimes {\rm SO}(W_2))$, where  $\dim W_1 \neq \dim W_2$ and $\dim W_2=2b$ is even. If $\dim W_1=2a+1$ is odd then $H/Z(G)=B_{a}D_{b}$ is connected; if $z=(z_1,z_2) \in {\rm GO}(W_1) \otimes {\rm GO}(W_2)$ and $z_2$ is a reflection then $\det(z)=-1$ and thus $z \in {\rm GO}(W) \setminus {\rm SO}(W)$.
Now assume $\dim W_1=2a$ is even. Here $H=D_{a}D_{b}\langle z_1,z_2 \rangle=D_{a}D_{b}.2^2$ is disconnected, where $z_i$ acts as a reflection on $W_i$, and the maximality of $H$ implies that we may assume $a>b \ge 2$.

\vs

Now let us turn to the subgroups in $\C_{4}(ii)$. Here $W$ admits a tensor product decomposition of the form $W=W_1 \otimes \cdots \otimes W_t$, where $t>1$ and the $W_{i}$ are mutually isometric spaces with $\dim W_1 = \dim W_i>1$ for all $i$. A subgroup in  $\C_{4}(ii)$ is the $G$-stabilizer of such a decomposition, so $H$ has the form $H=N_{G}(\prod_{i}{Cl(W_{i})})$ with the central product $\prod_{i}{Cl(W_{i})}$ acting naturally on the tensor product. As noted in \cite[Section 1]{LS}, we may assume that each factor $Cl(W_i)$ is simple. The following cases arise:
$$\prod_{i}{\rm SL}(W_i)< {\rm SL}(W),\; \prod_i{\rm Sp}(W_i) < {\rm Sp}(W)\, (\mbox{$t$ odd, $p \neq 2$}),$$
$$\prod_i{\rm Sp}(W_i) < {\rm SO}(W)\, (\mbox{$t$ even or $p=2$}),$$
$$\prod_i{\rm SO}(W_i) < {\rm SO}(W)\, (\mbox{$p \neq 2$, $\dim W_{i} \neq 2,4$}).$$

First suppose $G$ is of type $A_{n}$ and $H=N_G(\prod_i{\rm SL}(W_i))$, with $\dim W_{1}=l+1$ (so $n+1=(l+1)^t$). If $l=1$ then it is easy to see that $H$ is contained in a $\C_6$-subgroup of type ${\rm Sp}(W)$ or ${\rm SO}(W)$ (according to the parity of $t$), so the maximality of $H$ implies that $l \ge 2$, and we have 
$H=(\prod_{i}{\rm GL}(W_{i}).S_t) \cap G\cong A_{l}^{t}.S_{t}$. Similarly, if $G$ is symplectic or orthogonal (with $\dim W$ even) and $H=N_G(\prod_i{\rm Sp}(W_i))$ then $H=C_{l}^{t}.S_{t}$, while if $G$ is of type $B_{n}$ then $H=B_{l}^{t}.S_{t}$. Finally, suppose $G$ is of type $D_{n}$,  $p \neq 2$ and $H = N_{G}(\prod_i{\rm SO}(W_i))$ with $\dim W_{i}=2l \ge 6$ (since ${\rm SO}(W_i)$ is simple). Then $H=\prod_i{\rm GO}(W_i).S_t = \prod_i{\rm SO}(W_i).2^{t}.S_t$ is of type $(D_l.2^t).S_t$.

\subsection{Class $\C_6$: Classical subgroups}\label{ss:c6}

The members of the $\C_6$ collection are the classical subgroups $H=N_{G}({\rm Sp}(W))$ and $N_{G}({\rm SO}(W))$ in $G={\rm SL}(W)$, and also $H=N_{G}({\rm SO}(W))$ in $G={\rm Sp}(W)$ when $p=2$.
In the latter case, $W$ is even-dimensional and $H=N_{G}({\rm SO}(W)) = {\rm GO}(W) \cap {\rm Sp}(W) = {\rm SO}(W).2$ is a disconnected subgroup of type $D_n.2$. Now suppose $G={\rm SL}(W)$. If 
$H=N_{G}({\rm Sp}(W))$ then $H/Z(G) \cong {\rm Sp}(W)$ is connected, so let us assume $H=N_{G}({\rm SO}(W))$. If $\dim W$ is odd then $p \neq 2$ and $H/Z(G) \cong {\rm SO}(W)$ is connected. 
On the other hand, if $\dim W$ is even and $p \neq 2$ then $H/Z(G) \cong {\rm PSO}(W).2$ is disconnected, and we note that $H$ is contained in $N_{G}({\rm Sp}(W))$ if $p=2$. We also note that $H=N_G(T_1) = T_1.2$ is a $\C_2$-subgroup if $\dim W = 2$, so we may assume that $\dim W \ge 4$.

\subsection{A summary}

The following proposition provides a convenient summary of the above discussion.

\begin{prop}\label{p:cdisc}
Let $H \in \C$ be a positive-dimensional maximal subgroup of $G$ such that $H/Z(G)$ is disconnected. Then the possibilities for $H$ are listed in Table  \ref{t:disc}. 
\end{prop}

\begin{rmk}\label{r:a0}
\emph{In Table \ref{t:disc}, if $G=A_n$ and $H=A_l^{t}T_{t-1}.S_t$ is a $\C_2$-subgroup then we set $A_0=1$ if $l=0$; in this case, $H=N_G(T)$ is the normalizer of a maximal torus $T$ of $G$.} 
\end{rmk}

\renewcommand{\arraystretch}{1.2}
\begin{table}
$$\begin{array}{llll} \hline
 & G & H & \mbox{Conditions} \\ \hline

\C_{1} & B_{n} & D_{n}.2 &  \\
& B_{n} & D_{l}B_{n-l}.2 & 1 \le l <n \\
& D_{n} & D_{l}D_{n-l}.2 & 1 \le l <n/2 \\
& & & \\
\C_{2} 
& A_{n} & A_{l}^{t}T_{t-1}.S_{t} & n+1=(l+1)t,\, l \ge 0,\, t \ge 2 \\
& B_{n} & (2^{t-1} \times B_{l}^{t}).S_{t} & 2n+1=(2l+1)t,\, l \ge 1,\, \mbox{$t \ge 3$ odd} \\
& C_{n} & C_{l}^{t}.S_{t} & n=lt,\, l \ge 1,\, t \ge 2 \\
& D_{n} & (2^{t-1} \times B_{l}^{t}).S_{t} & 2n=(2l+1)t,\, l \ge 1,\, \mbox{$t \ge 2$ even, $p \neq 2$} \\
& D_{n} & (D_{l}^{t}.2^{t-1}).S_{t} & n=lt,\, l \ge 1,\, t \ge 2 \\
& & & \\
\C_{3} & C_{n} & A_{n-1}T_{1}.2 & p \neq 2 \\
& D_{n} & A_{n-1}T_{1}.2 & \mbox{$n$ even} \\
& & & \\
\C_{4}(i) & C_{n} & C_{a}D_{b}.2 & n=2ab,\,  b \ge 2,\, p \neq 2 \\
& D_{n} & D_{a}D_{b}.2^{2} & n=2ab,\, a>b\ge 2,\, p \neq 2, \\
& & & \\
\C_{4}(ii) & A_{n} & A_{l}^{t}.S_{t} & n+1=(l+1)^{t},\, l\ge 2 ,\, t \ge 2 \\
& B_{n} & B_{l}^{t}.S_{t} & \mbox{$2n+1=(2l+1)^t$, $l \ge 1$, $t \ge 2$} \\
& C_{n} & C_{l}^{t}.S_{t} & \mbox{$2n=(2l)^t$, $l \ge 1$, $t \ge 3$ odd, $p \neq 2$} \\
& D_{n} & C_{l}^{t}.S_{t} & \mbox{$2n=(2l)^t$, $l \ge 1$, $t \ge 2$ even or $p=2$} \\
& D_{n} & (D_{l}^{t}.2^t).S_{t} & \mbox{$2n=(2l)^t$, $l \ge 3$, $t \ge 2$, $p \neq 2$} \\
& & & \\
\C_{6} &
A_{n} & D_{m}.2 & \mbox{$n+1=2m$, $n \ge 3$, $p \neq 2$} \\
& C_{n} & D_{n}.2 & p=2 \\ \hline
\end{array}$$
\caption{The disconnected maximal subgroups of $G$ in $\mathcal{C}$}
\label{t:disc}
\end{table}
\renewcommand{\arraystretch}{1}

\chapter{The $\C_1, \C_3$ and $\C_6$ collections}\label{s:c136}

We begin the proof of Theorem \ref{main} by considering the disconnected maximal  subgroups in the $\C_1$, $\C_3$ and $\C_6$ collections. According to Proposition \ref{p:cdisc}, the relevant possibilities for $G$ and $H$ are listed in Table \ref{t:c1,3,6s}. (As in the statement of  Theorem \ref{main}, we will assume that $(G,p) \neq (B_n,2)$ -- see Remark \ref{r:conds}(b).)

\renewcommand{\arraystretch}{1.2}
\begin{table}[h]
$$\begin{array}{lllcl} \hline
& G & H & \mbox{Collection} & \mbox{Conditions} \\ \hline
{\rm (i)} & B_n & D_n.2 & \C_1 & \\
{\rm (ii)} & B_n & D_lB_{n-l}.2 & \C_1 & 1 \leq l < n \\
{\rm (iii)} & D_n & D_lD_{n-l}.2 & \C_1 & 1 \leq l < n/2 \\
{\rm (iv)} & C_n & A_{n-1}T_1.2 & \C_3 & p \neq 2 \\
{\rm (v)} & D_n & A_{n-1}T_1.2 & \C_3 & \mbox{$n$ even} \\
{\rm (vi)} & A_{n} & D_m.2 & \C_6 & \mbox{$n+1=2m$, $n \ge 3$, $p \neq 2$} \\
{\rm (vii)} & C_n & D_n.2 & \C_6 & p=2 \\ \hline
\end{array}$$
\caption{The disconnected maximal  subgroups in the collections $\C_1$, $\C_3$ and $\C_6$}
\label{t:c1,3,6s}
\end{table}
\renewcommand{\arraystretch}{1}

In the proof of Theorem \ref{main} we adopt the notation introduced in Section \ref{ss:notation}. In particular, we fix a maximal torus $T$ and a Borel subgroup $B=UT$ of $G$, which we use to define a base $\{\a_1, \ldots, \a_n\}$ for the root system of $G$, and corresponding fundamental dominant weights $\{\l_1, \ldots, \l_n\}$. We also fix a maximal torus $T_{H^0}$ of $(H^0)'$ contained in $T$, and if $\mu$ is a weight for $T$ then $\mu|_{H^0}$ designates the restriction of $\mu$ to this subtorus $T_{H^0}$. 

\section{The main result}

\begin{prop}\label{TH:C1,3,6}
Let $V$ be an irreducible tensor-indecomposable $p$-restricted $KG$-module with highest weight $\l$, and let $H$ be a maximal $\C_1$, $\C_3$ or $\C_6$-subgroup of $G$ such that $H/Z(G)$ is disconnected. Then $V|_{H}$ is irreducible if and only if  $(G,H,\l)$ is one of the cases recorded in Table \ref{t:c1,3,6}.
\end{prop}

\begin{remk}\label{r:conds1}
Some comments on the statement of Proposition \ref{TH:C1,3,6}:
\begin{itemize}\addtolength{\itemsep}{0.3\baselineskip}
\item[(a)] The family of examples arising when $(G,H) = (B_n,D_n.2)$ was found by Ford (see the case labelled ${\rm U}_2$ in \cite[Table II]{Ford1}). Here $p \neq 2$ and the required conditions on the coefficients $a_i$ in the expression $\l=\sum_{i=1}^{n}a_i\l_i$ are as follows:
\vspace{3mm}
\begin{quote}
$a_n=1$; if $a_i,a_j \neq 0$, where $i<j<n$ and $a_k=0$ for all $i<k<j$, then $a_i+a_j \equiv i-j \imod{p}$; if $i<n$ is maximal such that $a_i \neq 0$ then $2a_i \equiv -2(n-i)-1 \imod{p}$.
\end{quote}
\vspace{3mm}
In particular, if $p=0$ then $\l=\l_n$ is the only example.

\item[(b)] Consider the case $(G,H,\l) = (D_4, A_3T_1.2,\l_3)$ appearing in Table \ref{t:c1,3,6}. As noted in Remark \ref{r:conds}(h), if $\tilde{H}$ denotes the image of $H$ under a non-trivial graph automorphism of $G$ then $\l = \l_4$ is an example for the pair $(G,\tilde{H})$.

\item[(c)] The fourth column of Table \ref{t:c1,3,6} gives the restriction of the highest weight $\l$ to a suitable maximal torus of $(H^0)'$; it is convenient to denote this restriction by $\l|_{H^0}$. We adopt the standard labelling $\{\o_{1},\o_2, \ldots\}$ for the fundamental dominant weights of each factor in $H^0$. In the fifth column, $\kappa$ denotes the number of $KH^0$-composition factors in $V|_{H^0}$. Of course, any condition appearing in the final column of Table \ref{t:c1,3,6s} also applies for the relevant examples in Table \ref{t:c1,3,6}.
\end{itemize}
\end{remk}

\renewcommand{\arraystretch}{1.2}
\begin{table}[h]
$$\begin{array}{llllll} \hline
G & H & \l &  \l|_{H^0} & \kappa & \mbox{Conditions} \\  \hline

B_n & D_n.2 & \sum_{i=1}^{n} a_i\l_i & \sum_{i=1}^{n-1} a_i\o_i + (a_{n-1}+1)\o_n &  2 & \mbox{See Remark \ref{r:conds1}(a)} \\

B_n & T_1B_{n-1}.2 & \l_n &  \o_{n-1} & 2 &  \\

B_n & D_lB_{n-l}.2 & \l_n &  \o_{l} \otimes \o_{n-l} & 2 & l \ge 2 \\

D_n & T_1D_{n-1}.2 & \l_{n-\e} &  \o_{n-1-\e} & 2 & \e=0,1 \\
D_n & D_lD_{n-l}.2  & \l_{n-\e} &  \o_{l} \otimes \o_{n-l-\e} & 2 & \mbox{$2 \le l < n/2$, $\e=0,1$} \\

C_n & A_{n-1}T_1.2 & \l_1 & \o_1 & 2 & 	\\

D_n & A_{n-1}T_1.2 & \l_1 &  \o_{1} & 2 &	\\
 &  &  \l_3 & \o_3 & 2 & n=4	\\

A_{n} & D_m.2 & \l_k,\, \l_{n+1-k} & \o_{k} & 1 & 1 \le k < m-1 \\

& (n+1=2m) & \l_{m-1},\, \l_{m+1} & \o_{m-1}+\o_m & 1 & \\
& & \l_m & 2\o_{m} & 2 & \\

C_n & D_n.2 & \l_n & 2\o_{n} & 2 & \\

& &  \sum_{i=1}^{n-1} a_i\l_i &  \sum_{i=1}^{n-1}a_i\o_i +a_{n-1}\o_n & 1 & \\ \hline
\end{array}$$
\caption{The $\C_1$, $\C_3$ and $\C_6$ examples}
\label{t:c1,3,6}
\end{table}
\renewcommand{\arraystretch}{1}

\section{Proof of Proposition \ref{TH:C1,3,6}}

We will first consider case (i) of Table \ref{t:c1,3,6s}. Here we use Ford's main theorem \cite[Theorem 1]{Ford1} in a crucial way. Indeed, it is not difficult to reduce to the case where $V|_{H^0}$ has composition factors with $p$-restricted highest weights. The latter case is handled in \cite[Section 3]{Ford1}, where a long and detailed analysis is given. 

\begin{lem}\label{c1,3,6:p1}
Proposition \ref{TH:C1,3,6} holds in case (i) of Table \ref{t:c1,3,6s}.
\end{lem}

\begin{proof}
Here $G=B_n$, $H^0=D_{n}$ and up to conjugacy we have
$$H^0=\la U_{\pm \a_{1}},\ldots,U_{\pm \a_{n-1}},U_{\pm (\a_{n-1}+2\a_n)} \ra$$
and $H=H^0\la s_{\a_n} \ra$, where the simple reflection $s_{\a_n}$ interchanges the roots $\a_{n-1}$ and $\a_{n-1}+2\a_{n}$, inducing an involutory graph automorphism on $H^0$. Set $\Pi(H^0)=\{\b_1, \ldots, \b_n\}$, where $\b_i = \a_i$ for all $i<n$, and $\b_{n} = \a_{n-1}+2\a_n$. Let $\{\o_{1},\ldots,\o_{n}\}$ be corresponding fundamental dominant weights of $H^0$. Note that if $\mu = \sum_{i=1}^{n}{b_i\l_i}$ is a weight for $T$ then
$$\mu|_{H^0}
= \sum_{i=1}^{n-1}\la \mu,\a_{i} \ra \o_{i} + \la \mu,\a_{n-1}+2\a_n \ra \o_{n}
= \sum_{i=1}^{n-1} b_{i}\o_{i}+(b_{n-1}+b_{n})\o_{n}.$$

Suppose $V|_{H}$ is irreducible and $\mu \in \L(V)$ affords the highest weight of a $KH^0$-composition factor.
Then Clifford theory implies that either $\mu|_{H^0} = \l|_{H^0}$,
or 
$$\mu|_{H^0} = (s_{\a_n}.\l)|_{H^0}=\l|_{H^0}+a_n(\o_{n-1}-\o_{n}).$$

If $a_n=0$ then $\l|_{H^0}$ is a $p$-restricted weight for $H^0$, so the main theorem of \cite{Ford1} applies; by inspecting \cite[Tables I, II]{Ford1} we quickly deduce that there are no compatible examples with $a_n=0$. Now assume
$a_n \geq 1$.
Then $\mu=\l-\a_{n} \in \L(V)$ (see Lemma \ref{l:t1}) and $\mu$ affords the highest weight of a $KH^0$-composition factor (indeed, by \eqref{e:mt} we observe that a non-zero $T$-weight vector $v \in V_{\mu}$ is fixed by the Borel subgroup $B_{H^0} = \la T, U_{\b_i} \mid 1 \le i \le n \ra$ of $H^0$).
Since $\mu|_{H^0}=\l|_{H^0}+\o_{n-1}-\o_{n}$, we deduce that $\mu|_{H^0}=(s_{\a_n}.\l)|_{H^0}$ and $a_n=1$.

If $p=0$ or $a_{n-1}<p-1$, then $\l|_{H^0}$ is $p$-restricted and Ford's main theorem in \cite{Ford1} implies that the only example is the case labelled ${\rm U}_2$ in \cite[Table II]{Ford1}. We record this case in Table \ref{t:c1,3,6}, with the precise conditions on the coefficients $a_i$ given in Remark \ref{r:conds1}(a).

Finally, suppose $p \neq 0$ and $a_{n-1}=p-1$. Set $\nu = \l-\a_{n-1}-\a_{n} \in \L(V)$ and note that
$$\nu|_{H^0} = \l|_{H^0}+\o_{n-2}-\o_{n-1}-\o_{n} = \mu|_{H^0}+\o_{n-2}-2\o_{n-1} = \mu|_{H^0}-\b_{n-1}.$$
Then Lemma \ref{l:118} yields $m_{V}(\nu)=2$, but the $T_{H^0}$-weight $\nu|_{H^0}$ has multiplicity $1$ in the $KH^0$-composition factor with highest weight $\mu|_{H^0}$, and it does not occur in the factor with highest weight $\l|_{H^0}$. 
Therefore $\nu|_{H^0}$ must occur in a third $KH^0$-composition factor (see Lemma \ref{l:easy}), which is impossible as $|H:H^0|=2$.
\end{proof}

\begin{lem}\label{c1,3,6:p2}
Proposition \ref{TH:C1,3,6} holds in case (ii) of Table \ref{t:c1,3,6s}.
\end{lem}

\begin{proof}
Here $G=B_n$ and $H^0 = X_1X_2$, where $X_{1}=D_{l}$ and $X_{2}=B_{n-l}$. To begin with, let us assume $l=1$, in which case $H=H^0\la \sigma \ra$ with $\sigma$ an involution that inverts the $1$-dimensional central torus $X_1$, and centralizes the $X_2$ factor. Up to conjugacy, we may assume that $X_2=\la U_{\pm\a_2},\ldots,U_{\pm\a_n}\ra$. 

Now $X_{1}$ acts as scalars on the $KH^0$-composition factors of $V$, each of which is an irreducible $KX_2$-module. Therefore, $V|_{X_2}$ has precisely two composition factors, which are interchanged by $\sigma$ (so they are isomorphic as $KX_2$-modules).
For each $j \in \mathbb{N}_0$, $X_2$ preserves the subspace $V_j$ of $V$, which is defined to be the sum of the $T$-weight spaces in $V$ 
corresponding to weights of the form $\lambda-j\a_1-\sum_{i=2}^nc_i\a_i$ with $c_i \in \mathbb{N}_0$. By Lemma \ref{l:pr} and saturation (see \cite[Section 13.4]{Hu1}), if $j>0$ and $V_j \neq 0$ then $V_{j-1} \neq 0$, and we deduce that every $T$-weight in $\L(V)$ is of this form, with $j=0$ or $1$. But $w_0\lambda=-\lambda=\lambda-2\lambda$ is a weight of $V$, where $w_0$ is the longest word in the Weyl group of $G$, and so by writing the fundamental dominant weights as linear combinations of the simple roots (see \cite[Table 1]{Hu1}, for example), we deduce that $\lambda = 
\lambda_n$ is the only possibility. Then $V|_{X_2}$ has two composition factors, each of highest weight $\lambda_n|_{X_2}$, and the central torus $X_1$ acts with weight $1$ on one of the factors and with weight $-1$ on the second factor. Hence, $V$ is indeed an irreducible $KH$-module, and this case is recorded in Table \ref{t:c1,3,6}.

For the remainder we may assume $l \ge 2$.
Up to conjugacy, we have
$$X_{1}=\la U_{\pm \a_{1}},\dots,U_{\pm \a_{l-1}},U_{\pm (\a_{l-1}+2(\a_l+\dots+\a_n))} \ra,
\; X_{2}=\la U_{\pm \a_{l+1}},\dots,U_{\pm \a_{n}} \ra$$
and $H=H^0\la \s \ra$,
where $\s=s_{\a_{l}} \cdots s_{\a_{n-1}}s_{\a_{n}}s_{\a_{n-1}} \cdots s_{\a_{l}}$ induces an involutory graph automorphism on $X_1$.

Let $\{\o_{1,1},\ldots,\o_{1,l}\}$ and $\{\o_{2,1},\ldots,\o_{2,n-l}\}$ be fundamental dominant weights for $X_1$ and $X_2$ corresponding to the  bases of the respective root systems given above. In particular, if
$\mu = \sum_{i=1}^{n}{b_i\l_i}$ is a weight for $T$ then
$$\mu|_{H^0} = \sum_{i=1}^{l-1} b_{i}\o_{1,i}+(b_{l-1}+2b_{l}+\dots+2b_{n-1}+b_{n})\o_{1,l}+\sum_{i=1}^{n-l} b_{l+i}\o_{2,i}.$$

Suppose $V|_{H}$ is irreducible and $\mu \in \L(V)$ affords the highest weight of a $KH^0$-composition factor.
Then either $\mu|_{H^0} = \l|_{H^0}$
or
$$\mu|_{H^0} = (\s.\l)|_{H^0}=\l|_{H^0}+(2a_{l}+\dots+2a_{n-1}+a_{n})(\o_{1,l-1}-\o_{1,l}).$$

If $a_l \geq 1$ then $\mu=\l-\a_{l} \in \L(V)$ and $\mu$ affords the highest weight of a $KH^0$-composition factor, but this is not possible since $\mu|_{H^0}=\l|_{H^0}+\o_{1,l-1}-\o_{1,l}+\o_{2,1}$. It follows that $a_l=0$.

Next let $i \in \{1,\dots,l-1\}$ be maximal such that $a_i \neq 0$.
Then $\mu_i=\l-\a_{i}-\a_{i+1}-\dots-\a_{l} \in \L(V)$ and $\mu_i$ affords the highest weight of a $KH^0$-composition factor. Again, this is a contradiction since
$$\mu_{i}|_{H^0}=\l|_{H^0}+\sum_{j=1}^{l} c_{1,j}\o_{1,j}+\sum_{j=1}^{n-l} c_{2,j}\o_{2,j}$$
with $c_{2,1}=1$. Similarly, if $i \in \{l+1,\dots,n-1\}$ is minimal such that $a_i \neq 0$ then $\mu_i=\l-\a_{l}-\a_{l+1}-\dots-\a_i \in \L(V)$ and $\mu_i$ affords the highest weight of a $KH^0$-composition factor, but this is also impossible since
$$\mu_{i}|_{H^0}=\l|_{H^0}+\sum_{j=1}^{l} c_{1,j}\o_{1,j}+\sum_{j=1}^{n-l} c_{2,j}\o_{2,j}$$
with $c_{2,i-l}=-1$.

We have now reduced to the case $\l=a_n\l_n$.
Here $\mu=\l-\a_{l}-\a_{l+1}-\dots-\a_n \in \L(V)$ and $\mu$ affords the highest weight of a $KH^0$-composition factor.
Since $\mu|_{H^0} = \l|_{H^0}+\o_{1,l-1}-\o_{1,l}$, we deduce that $\mu|_{H^0}=(\s.\l)|_{H^0}$ and $a_n=1$.
Hence $\l=\l_n$ and $V|_{H}$ is irreducible as we have
$$\dim{L_{B_n}(\l_n)}=2\cdot \dim{L_{D_l}(\l_l)} \cdot \dim{L_{B_{n-l}}(\l_{n-l})}$$
by Lemma \ref{l:dmspin}. This case is recorded in Table \ref{t:c1,3,6}.
\end{proof}

\begin{lem}\label{c1,3,6:p3}
Proposition \ref{TH:C1,3,6} holds in case (iii) of Table \ref{t:c1,3,6s}.
\end{lem}

\begin{proof}
This is very similar to the previous lemma. Here $G=D_n$ and $H^0 = X_1X_2$, where $X_{1}=D_{l}$, $X_{2}=D_{n-l}$ and $1 \le l < n/2$. First assume $l=1$, so $H=H^0\la \sigma \ra$ with $\sigma$ an involution that inverts the $1$-dimensional central torus $X_1$, and acts as a graph automorphism on the $X_2$ factor. Up to conjugacy, we may assume that $X_2=\la U_{\pm\a_2},\dots,U_{\pm\a_n}\ra$.  

Now $X_{1}$ acts as scalars on the $KH^0$-composition factors of $V$, each of which is an irreducible $KX_2$-module, so $V|_{X_2}$ has precisely two composition factors, which are interchanged by $\sigma$. By arguing as in the proof of previous lemma (the case $l=1$), we quickly deduce that $\lambda = \lambda_{n-1}, \l_n$ are the only possibilities. Now $L_G(\lambda_{n-1})|_{X_2}$ has precisely two composition
factors; namely, the two distinct spin modules for $X_2$, which are interchanged by $\sigma$. Therefore $L_G(\lambda_{n-1})$ is an irreducible $KH$-module. The same argument applies for $L_G(\lambda_{n})$. These cases are recorded in Table \ref{t:c1,3,6}.

Now suppose $l \ge 2$. Up to conjugacy, we have
\begin{align*}
X_{1} & =\la U_{\pm \a_{1}},\dots,U_{\pm \a_{l-1}},U_{\pm (\a_{l-1}+2(\a_l+\dots+\a_{n-2})+\a_{n-1}+\a_{n})} \ra \\
X_{2} & =\la U_{\pm \a_{l+1}},\dots,U_{\pm \a_{n}} \ra
\end{align*}
and $H=H^0\la \s \ra$,
where $\s=s_{\a_{l}} \cdots s_{\a_{n-2}}s_{\a_{n-1}}s_{\a_{n}}s_{\a_{n-2}} \cdots s_{\a_{l}}$ induces a graph automorphism on both $X_1$ and $X_2$. Let $\{\o_{1,1},\ldots,\o_{1,l}\}$ and $\{\o_{2,1},\ldots,\o_{2,n-l}\}$ be the fundamental dominant weights corresponding to the above bases of the root systems of $X_1$ and $X_2$, respectively, so if $\mu = \sum_{i=1}^{n}{b_i\l_i}$ is a weight for $T$ then
$$\mu|_{H^0} = \sum_{i=1}^{l-1} b_{i}\o_{1,i}+(b_{l-1}+2b_{l}+\dots+2b_{n-2}+b_{n-1}+b_{n})\o_{1,l}+\sum_{i=1}^{n-l} b_{l+i}\o_{2,i}.$$

Suppose $V|_{H}$ is irreducible and $\mu \in \L(V)$ affords the highest weight of a $KH^0$-composition factor.
Then either $\mu|_{H^0} = \l|_{H^0}$
or
\begin{align*}
\mu|_{H^0} = (\s.\l)|_{H^0}= & \; \l|_{H^0}+(2a_{l}+\dots+2a_{n-2}+a_{n-1}+a_{n})(\o_{1,l-1}-\o_{1,l}) \\
& \; +(a_n-a_{n-1})(\o_{2,n-l-1}-\o_{2,n-l}).
\end{align*}

If $a_l \neq 0$ then $\mu=\l-\a_{l} \in \L(V)$ and $\mu$ affords the highest weight of a $KH^0$-composition factor, but this is not possible since $\mu|_{H^0}=\l|_{H^0}+\o_{1,l-1}-\o_{1,l}+\o_{2,1}$. Therefore $a_l=0$.

Next let $i \in \{1,\dots,l-1\}$ be maximal such that $a_i \neq 0$.
Then $\mu_i=\l-\a_{i}-\a_{i+1}-\dots-\a_l \in \L(V)$ and $\mu_i$ affords the highest weight of a $KH^0$-composition factor. This is a contradiction since
$$\mu_{i}|_{H^0}=\l|_{H^0}+\sum_{j=1}^{l} c_{1,j}\o_{1,j}+\sum_{j=1}^{n-l} c_{2,j}\o_{2,j}$$
with $c_{2,1}=1$. Similarly, suppose $i \in \{l+1,\dots,n-2\}$ is minimal such that $a_i \neq 0$.
Then $\mu_i=\l-\a_{l}-\a_{l+1}-\dots-\a_i \in \L(V)$ and $\mu_i$ affords the highest weight of a $KH^0$-composition factor. Again, this is impossible as
$$\mu_{i}|_{H^0}=\l|_{H^0}+\sum_{j=1}^{l} c_{1,j}\o_{1,j}+\sum_{j=1}^{n-l} c_{2,j}\o_{2,j}$$
with $c_{2,i-l}=-1$. We have now reduced to the case $\l=a_{n-1}\l_{n-1}+a_n\l_n$.

If $a_{n-1} \neq 0$, then $\mu=\l-\a_{l}-\a_{l+1}-\dots-\a_{n-2}-\a_{n-1} \in \L(V)$ and $\mu$ affords the highest weight of a $KH^0$-composition factor.
Since 
$$\mu|_{H^0} = \l|_{H^0}+\o_{1,l-1}-\o_{1,l}-\o_{2,n-l-1}+\o_{2,n-l},$$ 
we deduce that $\mu|_{H^0}=(\s.\l)|_{H^0}$, $a_{n-1}=1$ and $a_n=0$.
Hence $\l=\l_{n-1}$ and thus $V|_{H}$ is irreducible since Lemma \ref{l:dmspin} yields
$$\dim{L_{D_n}(\l_{n-1})}=2\cdot \dim{L_{D_l}(\l_l)}\cdot \dim{L_{D_{n-l}}(\l_{n-l-1})}.$$
Similarly, if $a_{n} \neq 0$ then $\l=\l_{n}$ is the only possibility, and once again we deduce that $V|_{H}$ is irreducible. The examples $\l=\l_{n-1}$ and $\l_n$ are listed in Table \ref{t:c1,3,6}.
\end{proof}

\begin{lem}\label{c1,3,6:p4,5}
Proposition \ref{TH:C1,3,6} holds in cases (iv) and (v) of Table \ref{t:c1,3,6s}.
\end{lem}

\begin{proof}
Here $G=C_n$ or $D_n$, and $(H^0)' =A_{n-1}$. According to Proposition \ref{p:cdisc}, we may assume $n$ is even if $G=D_n$. Set $J=(H^0)'$ and observe that $J =\la U_{\pm \a_{1}},\dots,U_{\pm \a_{n-1}} \ra$, up to conjugacy in ${\rm Aut}(G)$. Referring to Remark \ref{r:conds}(h) (or Remark \ref{r:conds1}(b)), we may assume that $J$ is as given.
The $1$-dimensional central torus $T_1<H^0$ acts as scalars on the $KH^0$-composition factors of $V$, each of which is an irreducible $KJ$-module. Therefore $V|_J$ has exactly two composition factors, and by arguing as in the proof of Lemma \ref{c1,3,6:p2} (the case $l=1$) we deduce that either $\l=\l_1$, or $G=D_4$ and $\lambda =\lambda_3$. (Note that if $G = D_4$ and we take $J = \la U_{\pm \a_{1}}, U_{\pm \a_{2}}, U_{\pm \a_4} \ra$, then the relevant weights are $\lambda_1$ and $\lambda_4$; see Remark \ref{r:conds1}(b).)
In the latter case, the two $KJ$-composition factors are interchanged by the outer automorphism in $N_G(H^0)$, and so $V|_{H}$ is indeed irreducible. This case is recorded in Table \ref{t:c1,3,6}.
\end{proof}

Next let us consider case (vi) in Table \ref{t:c1,3,6s}. Here $G=A_n$ and $H=D_m.2$, where $n = 2m-1 \ge 3$ and $p \neq 2$. 
Suppose $V|_{H}$ is irreducible, but $V|_{H^0}$ is reducible. Then $V = V_1 \oplus V_2$, where $V_1$ and $V_2$ are irreducible $KH^0$-modules with respective $T_{H^0}$-highest weights 
\begin{equation}\label{e:mu1mu2}
\mu_1 = \sum_{i=1}^{m}c_i\o_i,\;\; \mu_2 = \sum_{i=1}^{m-2}c_i\o_i + c_{m}\o_{m-1}+c_{m-1}\o_{m}.
\end{equation}
Without loss of generality, we may assume that $\mu_1 = \l|_{H^0}$. The following lemma applies in this situation, where we take $\{\b_1, \ldots, \b_m\}$ to be a set of simple roots for $H^0=D_m$, corresponding to the fundamental dominant weights $\{\o_1, \ldots, \o_m\}$.

\begin{lem}\label{l:andm}
Suppose $\mu \in \Lambda (V)$ and $\mu|_{H^0} = \mu_1-\beta_m+\beta_{m-1}$.  Then $\mu_2=\mu|_{H^0}$. 
\end{lem}

\begin{proof}
Let $\nu=\mu|_{H^0}$ and observe that $\nu \not \in \Lambda(V_1)$ since $\nu$ is not under $\mu_1$ (that is, $\nu$ is not of the form $\mu_1 - \sum_{i}d_i\b_i$ with $d_i \in \mathbb{N}_0$). Therefore  $\nu\in \Lambda(V_2)$, so $\nu=\mu_2-\sum_{i} k_i\beta_i$ for some $k_i \in \mathbb{N}_0$, whence
$$\mu_2-\mu_1=\sum_{i=1}^{m-2} k_i\beta_i+(k_{m-1}+1)\beta_{m-1}+(k_{m}-1)\beta_{m}.$$ 
Now $\o_{m-1}-\o_m=(\beta_{m-1}-\beta_m)/2$ (see \cite[Table 1]{Hu1}), so using \eqref{e:mu1mu2} we deduce that  
$$\mu_2-\mu_1=\frac{1}{2}(c_m-c_{m-1})(\beta_{m-1}-\beta_m).$$
Therefore $k_{m-1}+1=-(k_m-1)$, and $k_i=0$ for all $1\leq i \leq m-2$, so $k_i=0$ for all $i$, and the result is proved.
\end{proof}

\begin{lem}\label{c1,3,6:p6}
Proposition \ref{TH:C1,3,6} holds in case (vi) of Table \ref{t:c1,3,6s}.
\end{lem}

\begin{proof}
Here $G=A_{2m-1}$ and $H^0 = D_m$, with $m \ge 2$ and $p \neq 2$. As above, let $\{\b_1, \ldots, \b_{m}\}$ be a set of simple roots for $H^0$ and let $\{\o_1, \ldots, \o_m\}$ be the  corresponding fundamental dominant weights. The $A_{m-1}$ parabolic subgroup of $H^0$, corresponding to the simple roots
$\{\beta_1,\dots,\beta_{m-1}\}$, embeds in an $A_{m-1}\times A_{m-1}$
parabolic subgroup
of $G$, and up to conjugacy we may assume that this gives the root
restrictions $\a_i|_{H^0} = \b_i$ and $\a_{m+i}|_{H^0} = \b_{m-i}$ for all $1 \le i \le m-1$. By considering the action of the Levi factors of these parabolics on $W$ (the natural $KG$-module), we deduce that the weight $\lambda_1-\alpha_1-\cdots-\alpha_m$ must restrict
to $\lambda_1|_{H^0}-\beta_1-\cdots-\beta_{m-2}-\beta_m$, which yields 
$\alpha_m|_{H^0} = \beta_m-\beta_{m-1}$.

Now $\lambda_i=i\lambda_1-\sum_{j=1}^{i-1}(i-j)\alpha_j$ for all $1 \le i \le 2m-1$ (see \cite[Table 1]{Hu1}), so using the above root restrictions (together with the fact that $\l_1|_{H^0} = \o_1$) we get
$$\l_i|_{H^0} = \o_{i},\; \l_{m+j}|_{H^0} = \o_{m-j},\; \l_{m-1}|_{H^0} = \l_{m+1}|_{H^0} = \o_{m-1}+\o_{m},\; \l_{m}|_{H^0} = 2\o_m,$$
for all $1 \le i \le m-2$ and $2 \le j \le m-1$. It follows that 
$$\lambda|_{H^0} = \sum_{j=1}^{m-2}(a_j+a_{2m-j})\o_j+(a_{m-1}+a_{m+1})\o_{m-1}+(a_{m-1}+a_{m+1}+2a_m)\o_{m}.$$
 
Suppose $V|_{H}$ is irreducible. If $V|_{H^0}$ is irreducible then the main theorem of \cite{Seitz2} implies that $\l = \l_k$ with $k \neq m$, and these cases are recorded in Table \ref{t:c1,3,6}. For the remainder let us assume $V|_{H^0}$ is reducible. As above, write $V = V_1 \oplus V_2$, where each $V_i$ is an irreducible $KH^0$-module with highest weight $\mu_i$ (as given in \eqref{e:mu1mu2}), and the $V_i$ are interchanged under the action of the graph automorphism of $H^0$. Without loss of generality, we may assume that $\mu_1 = \l|_{H^0}$, so 
$$\mu_2=\sum_{j=1}^{m-2}(a_j+a_{2m-j})\o_j+(a_{m-1}+a_{m+1}+2a_m)\o_{m-1}+(a_{m-1}+a_{m+1})\o_m.$$
By Proposition \ref{p:niso}, $V_1$ and $V_2$ are non-isomorphic $KH^0$-modules, so $\mu_1 \neq \mu_2$ and thus $a_m \neq 0$. Therefore $\lambda-\alpha_m \in \Lambda(V)$ and so Lemma \ref{l:andm} implies that $\mu_2=\mu_1-\beta_m+\beta_{m-1}$ since 
$(\lambda-\alpha_m)|_{H^0}=\mu_1-\beta_m+\beta_{m-1}$. This yields
$$\mu_2=\sum_{j=1}^{m-2}(a_j+a_{2m-j})\o_j+(a_{m-1}+a_{m+1}+2)\o_{m-1}+(a_{m-1}+a_{m+1}+2a_m-2)\o_m$$ 
and we deduce that $a_m=1$.   

Next we claim that $a_{i}=a_{2m-i}=0$ for all $1\leq i\leq m-1$, so we reduce to the case $\l=\l_m$. We proceed by induction on $m-i$. To establish the base case we need to show that $a_{m-1}=a_{m+1}=0$. Suppose $a_{m-1}\neq 0$. If $a_{m+1} \neq 0$ then $\lambda-\alpha_{m-1}-2\alpha_m$ and $\lambda-2\alpha_m-\alpha_{m+1}$ are weights of $V$, which both restrict to the same 
$T_{H^0}$-weight $\nu=\mu_1-2\beta_m+\beta_{m-1}=\mu_2-\beta_m$. Now $\nu$ does not occur in $V_1$ (since $\nu$ is not under $\mu_1$), and its multiplicity in $V_2$ is at most $1$. By Lemma \ref{l:easy}, this is incompatible with $V=V_1\oplus V_2$, so we must have $a_{m+1}=0$. 
Now consider the weights $\lambda-\alpha_{m-1}-\alpha_m$ and $\lambda-\alpha_m-\alpha_{m+1}$ in $\Lambda(V)$. Since $a_{m+1}=0$, $\lambda-\alpha_m-\alpha_{m+1}$ is conjugate to $\lambda-\alpha_m$ and thus Lemma \ref{l:t1} implies that this weight has multiplicity $1$ in $V$. Similarly, Lemma \ref{l:s816} implies that  $\lambda-\alpha_{m-1}-\alpha_m$ has multiplicity $2-\e$ in $V$, where $\e=1$ if $p$ divides $a_{m-1}+2$, otherwise $\e=0$. Both of these weights restrict to the $T_{H^0}$-weight $\nu=\mu_1-\beta_m=\mu_2-\beta_{m-1}$, so $\nu$ occurs in $V$ with multiplicity at least $3-\e$. Now $\nu$ has multiplicity at most $1$ in both $V_1$ and $V_2$, so Lemma \ref{l:easy} implies that $\e=1$ (since $V=V_1\oplus V_2$). Therefore $p$ divides $a_{m-1}+2$, hence 
$a_{m-1}=p-2$ since $\l$ is $p$-restricted. Therefore $p>0$ and $c_m=p$ (see \eqref{e:mu1mu2}), so $\nu \not \in \Lambda(V_1)$ and the multiplicity of $\nu$ in $V_1 \oplus V_2$ is at most $1$. This final contradiction implies that $a_{m-1}=0$, and a symmetric argument yields $a_{m+1}=0$. This establishes the base case in the induction. In particular, we have reduced to the case $\l=\l_2$ when $m=2$.

Now assume $m \ge 3$ and $2\leq k \leq m-1$, and suppose that $a_{m-k+j}=a_{m+k-j}=0$ for all $1\leq j \leq k-1$. To complete the proof of the claim, we need to show that $a_{m-k}=a_{m+k}=0$. Seeking a contradiction, let us assume $a_{m-k}\neq 0$. If $a_{m+k} \neq 0$ then 
$\lambda-\sum_{j=m-k}^{m-1}\alpha_j-2\alpha_m$ and $\lambda-2\alpha_m-\sum_{j=m+1}^{m+k}\alpha_j$ are weights of $V$, which both restrict to the same $T_{H^0}$-weight $\nu=\mu_1-\sum_{j=m-k}^{m-2}\beta_j+\beta_{m-1}-2\beta_{m}=\mu_2-\sum_{j=m-k}^{m-2}\beta_j-\beta_m$.  Now $\nu$ does not occur in $V_1$ (since $\nu$ is not under $\mu_1$). Furthermore, since  $a_{m-k}\neq 0$ and $\mu_2=\sum_{j=1}^{m-k}(a_j+a_{2m-j})\o_j+2\o_{m-1}$ by the induction hypothesis, $\nu$ is conjugate to $\mu_2-\beta_{m-k}$, so the multiplicity of $\nu$ in $V_2$ is at most $1$. Therefore $m_{V}(\nu) \ge 2$ but the multiplicity of $\nu$ in $V_1\oplus V_2$ is at most $1$. This is a contradiction, and we conclude that $a_{m+k}=0$. 
  
We have $a_{m-k}\neq 0$ and $a_m=1$, so $\nu_i=\lambda-\sum_{j=m}^{m+k-i}\alpha_j-\sum_{j=1}^{i} \alpha_{m-k+j-1}$ is a weight of $V$ for all $0\leq i \leq k$, and each $\nu_i$ restricts to the same $T_{H^0}$-weight $\nu=\mu_1-\sum_{j=m-k}^{m-2}\beta_j-\beta_m=\mu_2-\sum_{j=m-k}^{m-1}\beta_j$.  Since $a_m=1$, the induction hypothesis implies that 
each weight $\nu_i$ with $i<k$ is conjugate to $\lambda-\alpha_{m-k}-\alpha_m$, hence $m_{V}(\nu_i)=1$ for all $i<k$. In addition, Lemma \ref{l:s816} implies that $\nu_k$ occurs in $V$ with multiplicity $k$ if $k+1+a_{m-k}$ is divisible by $p$, otherwise with multiplicity $k+1$. It follows that $m_V(\nu) \ge 2k$ if $p$ divides $k+1+a_{m-k}$, otherwise $m_V(\nu) \ge 2k+1$. We now calculate the multiplicity of $\nu$ in $V_1 \oplus V_2$. By  Lemma \ref{l:s816}, the weight $\mu_1-\sum_{j=m-k}^{m-2}\beta_j-\beta_m \in \Lambda(V_1)$ occurs in $V_1$ with multiplicity $k-1$ if $p$ divides $k+1+a_{m-k}$, and multiplicity $k$ otherwise.  Similarly, $\mu_2-\sum_{j=m-k}^{m-1}\beta_j$ has multiplicity $k-1$ in $V_2$ if $k+1+a_{m-k}$ is divisible by $p$, otherwise the multiplicity is $k$. We conclude that $m_{V_1\oplus V_2}(\nu) < m_{V}(\nu)$, which is a contradiction. Therefore  
$a_{m-k}=0$, and a symmetric argument yields $a_{m+k}=0$. This completes the proof of the claim.
  
We have now reduced to the case $\l=\l_m$, so $V=\Lambda^m(W)$, $\mu_1=2\o_m$ and $\mu_2=2\o_{m-1}$. To complete the proof of the lemma it remains to show that $V=V_1 \oplus V_2$. By Lemma \ref{l:dmdim} we have
$$\dim V = \dim \Lambda^m(W)=\binom{2m}{m} = 2\cdot \dim L_{D_m}(2\o_m) = \dim V_1 + \dim V_2$$
and the result follows. We record the case $\l=\l_m$ in Table \ref{t:c1,3,6}.
\end{proof}

\begin{lem}\label{c1,3,6:p7}
Proposition \ref{TH:C1,3,6} holds in case (vii) of Table \ref{t:c1,3,6s}.
\end{lem}

\begin{proof}
Here $G=C_n$, $p=2$ and $H^0 = D_{n}$. Up to conjugacy we have 
$$H^0=\la U_{\pm \a_{1}},\dots,U_{\pm \a_{n-1}},U_{\pm (\a_{n-1}+\a_n)} \ra$$ 
and $H=H^0\la s_{\a_n} \ra$.
Let $\{\o_{1},\ldots,\o_{n}\}$ be the set of fundamental dominant weights for $H^0$ corresponding to this base of its root system, so if
$\mu = \sum_{i=1}^{n}{b_i\l_i}$ is a weight for $T$ then
$$\mu|_{H^0}= \sum_{i=1}^{n-1} b_{i}\o_{i}+(b_{n-1}+2b_{n})\o_{n}.$$
(Note that $\a_{n-1}+\a_{n} \in \Sigma(G)$ is a short root.)

Suppose $V|_{H}$ is irreducible. Since $V$ is a tensor-indecomposable $KG$-module,  \cite[1.6]{Seitz2} implies that either  $a_n=0$ or $\l=\l_n$. If $a_n=0$ then all composition factors of $V|_{H^0}$ have $p$-restricted highest weights, so the configuration $(C_n,D_n,V)$ must arise as one of the examples in \cite{Ford1}. By considering the case labelled ${\rm MR}_4$ in \cite[Table I]{Ford1}, we see that $V|_{H}$ is irreducible for any highest weight of the form $\l=\sum_{i}a_i\l_i$ with $a_n=0$. 
On the other hand, if
$\l=\l_n$ then $V$ is a spin module for $G$ and $V|_{H^0}$ is a sum of two Frobenius twists of spin modules for $H^0$; this example is labelled ${\rm U}_6$ in \cite[Table II]{Ford1}. These cases are listed in Table \ref{t:c1,3,6}.
\end{proof}

This completes the proof of Proposition \ref{TH:C1,3,6}. In particular, we have established Theorem \ref{main} for the relevant subgroups in the $\C_1$, $\C_3$ and $\C_6$ collections.

\chapter{Imprimitive subgroups}\label{s:c2}

Let us now turn to the imprimitive subgroups comprising the $\C_2$ collection. Recall from Section \ref{ss:c2} that such a subgroup arises as the stabilizer  in $G$ of a direct sum decomposition
$$W = W_{1}\oplus W_{2} \oplus \cdots \oplus W_{t}$$
of the natural $KG$-module $W$, where $\dim W_{1}=\dim W_{i}$ for all $i$. Moreover, if $G$ is symplectic or orthogonal then we require the $W_{i}$ to be non-degenerate and pairwise orthogonal with respect to the underlying non-degenerate form on $W$.
The particular cases to be considered are listed in Table \ref{t:c2s} (see Proposition \ref{p:cdisc}).

\renewcommand{\arraystretch}{1.2}
\begin{table}[h]
$$\begin{array}{lllcl} \hline
& G & H & & \mbox{Conditions} \\ \hline
{\rm (i)} & A_n & A_l^tT_{t-1}.S_t & n+1 = (l+1)t & \mbox{$l \ge 0$, $t \ge 2$} \\
{\rm (ii)} & B_n & (2^{t-1}\times B_l^t).S_t & 2n+1=(2l+1)t & \mbox{$l \ge 1$, $t \ge 3$ odd} \\
{\rm (iii)} & C_n & C_l^t.S_t & n=lt & \mbox{$l \ge 1$, $t \ge 2$} \\
{\rm (iv)} & D_n & (2^{t-1}\times B_l^t).S_t & 2n=(2l+1)t & \mbox{$l \ge 1$, $t \ge 2$ even, $p \neq 2$} \\
{\rm (v)} & D_n & (D_l^t.2^{t-1}).S_t & n=lt & \mbox{$l \ge 1$, $t \ge 2$} \\ \hline
\end{array}$$
\caption{The collection $\C_2$}
\label{t:c2s}
\end{table}
\renewcommand{\arraystretch}{1}

\section{The main result} 

\begin{prop}\label{TH:C2}
Let $V$ be an irreducible tensor-indecomposable $p$-restricted $KG$-module with highest weight $\l$ and let $H$ be a maximal $\C_{2}$-subgroup of $G$. Then $V|_{H}$ is irreducible if and only if $(G,H,\l)$ is one of the cases recorded in Table \ref{t:c2}.
\end{prop}

\begin{remk}\label{r:condds}
Let us make a couple of comments on the statement of Proposition \ref{TH:C2}.
\begin{itemize}\addtolength{\itemsep}{0.3\baselineskip}
\item[(a)] The required conditions for the case $(G,H)=(C_n,C_l^t.S_t)$ in Table \ref{t:c2} (with $\l = \l_{n-1} + a\l_n$ and $n=lt$) are as follows:
\vspace{3mm}
\begin{quote}
$t=2$ and either $(l,a) = (1,0)$, or $0 \le a <p$ and $2a+3 \equiv 0 \imod{p}$.
\end{quote}
\vspace{3mm}
Note that if $l=1$ then $\l|_{H^0} = (a+1)\o_{1,1}+a\o_{2,1}$.
\item[(b)] In the fourth column of Table \ref{t:c2}, we give the restriction of $\l$ to a suitable maximal torus of $(H^0)'$ (as before, we denote this restriction by $\l|_{H^0}$) in terms of a set of fundamental dominant weights $\{\o_{i,1}, \ldots, \o_{i,l}\}$ for the $i$-th factor $X_i$ in $(H^0)' = X_1\cdots X_t$. In addition, in the fifth column $\kappa$ denotes the number of $KH^0$-composition factors in $V|_{H^0}$. Any condition appearing in the final column of Table \ref{t:c2s} also applies for any of the relevant examples in Table \ref{t:c2}.
\end{itemize}
\end{remk}

\renewcommand{\arraystretch}{1.2}
\begin{table}[h]\small
$$\begin{array}{llllll} \hline
G & H & \l & \l|_{H^0} & \kappa & \mbox{Conditions} \\  \hline
A_n & A_l^tT_{t-1}.S_t & \l_1 & \omega_{1,1} & t & \\
& & \l_n & \omega_{t,l} & t & \\
& & \l_k & - & \binom{n+1}{k} & l = 0,\, 2 \le k \le n-1 \\

B_n & (2^{t-1}\times B_l^t).S_t & \l_1 & \omega_{1,1} & t & \\
& & \l_n & \omega_{1,l} + \cdots + \omega_{t,l} & 2^{(t-1)/2} & \\

C_n & C_l^t.S_t & \l_1 & \omega_{1,1} & t & \\
& & \l_n & \omega_{1,l} + \cdots + \omega_{t,l} & 1 & p = 2 \\
& & \l_{n-1}+a\l_n & (a+1)\o_{1,l} +\o_{2,l-1}+a\o_{2,l} & 2 & \mbox{See Remark \ref{r:condds}(a)} \\

D_n & (2^{t-1}\times B_l^t).S_t & \l_1 & \o_{1,1} & t & \\
& & \l_{n-1},\,\l_{n} & \o_{1,l} + \cdots + \o_{t,l} & 2^{(t-2)/2} &  \\
D_n & (D_l^t.2^{t-1}).S_t & \l_1 & \o_{1,1} & t & \\
 & & \l_{n-1},\,\l_n  & \o_{1,l} + \cdots + \o_{t,l} & 2^{t-1} & \\
& & \l_{1}+\l_{n-1} & \o_{1,1}+\o_{1,l} + \o_{2,l-1} & 4 & (t,p)=(2,2),\, \mbox{$l \ge 3$ odd} \\
& & \l_{1}+\l_{n} & \o_{1,1}+\o_{1,l} + \o_{2,l} & 4 & (t,p)=(2,2),\, \mbox{$l \ge 3$ odd} \\ \hline
\end{array}$$
\caption{The $\C_2$ examples}
\label{t:c2}
\end{table}
\renewcommand{\arraystretch}{1}

\section{Preliminaries}\label{ss:prel_c2}

Let $H$ be a $\C_{2}$-subgroup of $G$ and assume that $(H^0)'$ is semisimple (that is, assume $l \ge 1$ in case (i) of Table \ref{t:c2s}, and $l \ge 2$ in case (v)). Write $(H^0)' = X_{1}\cdots X_{t}$. Let $\{\b_{i,1}, \ldots, \b_{i,l}\}$ be a set of simple roots for $X_i$ and let $\{\o_{i,1},\ldots,\o_{i,l}\}$ be a corresponding set of fundamental dominant weights for $X_i$. Let $\l=\sum_{i=1}^{n}a_{i}\l_{i}$ denote the highest weight of $V$ and suppose $\l|_{X_{i}} = \sum_{j=1}^{l}{a_{i,j}\o_{i,j}}$ for each $i$, so
$$\l|_{H^0} = \sum_{i=1}^{t}{\sum_{j=1}^{l}{a_{i,j}\o_{i,j}}}.$$
Suppose $V|_{H}$ is irreducible and assume for now that we are not in case (v) of Table \ref{t:c2s}. 
If $\mu = \sum_{i=1}^{n}b_{i}\l_{i}$ is a weight of $V$ that affords the highest weight of a composition factor of $V|_{H^0}$ then Clifford theory implies that there exists a permutation $\s \in S_t$ such that
\begin{equation}\label{e:star}
\mu|_{H^0} = \sum_{i=1}^{t}\sum_{j=1}^{l}{a_{i,j}\o_{\s(i),j}}.
\end{equation}
We call $\s$ the \emph{associated permutation} of $\mu$. Write $\mu|_{X_{i}} = \sum_{j=1}^{l}{b_{i,j}\o_{i,j}}$ and define $h(\mu|_{H^0}) = \sum_{i=1}^t{h(\mu|_{X_{i}})}$, where
$$h(\mu|_{X_{i}}) = \sum_{j=1}^{l}{b_{i,j}}$$
is the sum of the coefficients of the fundamental dominant weights of $X_i$ in the above expression for $\mu|_{X_{i}}$.

It will be useful in our later analysis to observe that if a subset $S$ of $\{1, \ldots, t\}$ is $\s$-invariant (that is, if $\s$ fixes $S$ setwise) then we have an equality of multisets
\begin{equation}\label{e:S}
\{h(\mu|_{X_{i}}) \mid i \in S\} = \{h(\l|_{X_{i}}) \mid i \in S\}.
\end{equation}
In particular, we have
\begin{equation}\label{e:h}
\sum_{i \in S}{h(\mu|_{X_{i}})} = \sum_{i \in S}{h(\l|_{X_{i}})},
\end{equation}
and so in the special case
$S=\{1, \ldots, t\}$, we get $h(\mu|_{H^0}) = h(\l|_{H^0})$.

The situation in case (v) of Table \ref{t:c2s} is very similar. Suppose $l \ge 2$ and fix a labelling of simple roots so that the standard graph automorphism of $X_{i}$ swaps the simple roots $\b_{i,l-1}$ and $\b_{i,l}$. If we assume $V|_{H}$ is irreducible and $\mu \in \L(V)$ affords the highest weight of a
$KH^0$-composition factor then there exists $\s \in S_t$ and a collection of permutations $\{\rho_1, \ldots, \rho_t\}$ in $S_{l}$ such that $\rho_i = (l-1,l)$ if $\rho_{i}\neq 1$, and precisely an even number $k \ge 0$ of the $\rho_i$ are non-trivial, and we have
$$\mu|_{X_{\s(i)}} = \sum_{j=1}^{l-2}{a_{i,j}\o_{\s(i),j}}+a_{i,l-1}\o_{\s(i),\rho_i(l-1)}+a_{i,l}\o_{\s(i),\rho_i(l)}$$
for all $i$. In particular, we note that \eqref{e:S} and \eqref{e:h} hold for any $\s$-invariant subset $S$ of $\{1, \ldots, t\}$.

For easy reference, let us record this general observation.

\begin{lem}\label{l:eh}
Suppose $V|_{H}$ is irreducible and $\mu \in \L(V)$ affords the highest weight of a composition factor of $V|_{H^0}$. If $\s \in S_t$ is the associated permutation of $\mu$ and $S$ is a $\s$-invariant subset of $\{1, \ldots, t\}$ then
$$\sum_{i \in S}{h(\mu|_{X_{i}})} = \sum_{i \in S}{h(\l|_{X_{i}})}.$$
In particular, we have
$h(\mu|_{H^0}) = h(\l|_{H^0})$.
\end{lem}

\section{Proof of Proposition \ref{TH:C2}}\label{ss:pc2}

\begin{lem}\label{c2:p1}
Proposition \ref{TH:C2} holds in case (i) of Table \ref{t:c2s}.
\end{lem}

\begin{proof}
Here $G=A_n$, $n+1=(l+1)t$ and $l \ge 0$. If $l=0$ then $H$ is the stabilizer of a direct sum decomposition of $W$ into $1$-spaces and thus $H=N_{G}(T)$ is the normalizer of a maximal torus $T$ of $G$. By \cite[Lemma 2.4]{g_paper}, $V|_{H}$ is irreducible if and only if $\l$ is \emph{minimal} (see the paragraph preceding Lemma \ref{l:dmspin}), and so the only examples are $\l=\l_k$ for all $1 \le k \le n$ (see \eqref{e:mini}).

For the remainder we may assume $l \ge 1$, so $H^0 = X_1 \cdots X_tT_{t-1}$ and up to conjugacy we have
$$X_{i}=\la U_{\pm \a_{(i-1)(l+1)+1}}, \ldots, U_{\pm \a_{(i-1)(l+1)+l}} \ra \cong A_{l}.$$
Let $\{\o_{i,1},\ldots,\o_{i,l}\}$ be the fundamental dominant weights of $X_i$ corresponding to this base of its root system, so if
$\mu = \sum_{i=1}^{n}{b_i\l_i}$ is a weight for $T$ then
$$\mu|_{X_{i}} = \sum_{j=1}^{l}\la \mu,\a_{(i-1)(l+1)+j}\ra \o_{i,j} = \sum_{j=1}^{l}{b_{(i-1)(l+1)+j}\o_{i,j}}$$
and thus
$$\mu|_{H^0} = \sum_{i=1}^{t}{\sum_{j=1}^{l}{b_{(i-1)(l+1)+j}\o_{i,j}}}.$$
Let $\l = \sum_{i=1}^n{a_i\l_i}$ denote the highest weight of $V$ and assume $V|_{H}$ is irreducible.

First suppose $a_{m(l+1)} \neq 0$ for some $m \in \{1, \ldots, t-1\}$. Then
$$\mu = \l - \a_{m(l+1)}  = \l+\l_{m(l+1)-1}-2\l_{m(l+1)}+\l_{m(l+1)+1} \in \L(V)$$
affords the highest weight of a $KH^0$-composition factor and
$$\mu|_{H^0} = \l|_{H^0}+\o_{m,l}+\o_{m+1,1}.$$
Therefore $h(\mu|_{H^0}) = h(\l|_{H^0})+2$, which contradicts Lemma \ref{l:eh}.
Similarly, if $l=1$ and $a_{2m+1}\neq 0$ with $m \in \{1, \ldots, t-2\}$ then $\mu = \l-\a_{2m}-\a_{2m+1}-\a_{2m+2} \in \L(V)$ affords the highest weight of a 
$KH^0$-composition factor and once again we reach a contradiction since $h(\mu|_{H^0}) = h(\l|_{H^0})+2$.

Now assume $l \ge 2$. Let $r \in \{1, \ldots, l-1\}$ be minimal such that $a_{m(l+1)-r} \neq 0$ for some $m \in \{1, \ldots, t-1\}$. Then
$$\mu = \l - \a_{m(l+1)-r} - \a_{m(l+1)-r+1} -  \cdots - \a_{m(l+1)} \in \L(V)$$
affords the highest weight of a $KH^0$-composition factor, but
$h(\mu|_{H^0}) = h(\l|_{H^0})+1$. Similarly, if $r \in \{1, \ldots, l-1\}$ is minimal such that $a_{m(l+1)+r} \neq 0$ for some $m \in \{1, \ldots, t-1\}$  then
$\mu  = \l -  \a_{m(l+1)} - \cdots - \a_{m(l+1)+r} \in \L(V)$
affords the highest weight of a $KH^0$-composition factor, but again we have $h(\mu|_{H^0}) = h(\l|_{H^0})+1$, which is not possible. Notice that for any $l \ge 1$ we have now reduced to the case 
$$\l=a_1\l_1+a_n\l_n.$$

Suppose $a_1 \neq 0$. Then $\mu = \l - \a_1-\a_2 - \cdots - \a_{l+1} \in \L(V)$
affords the highest weight of a $KH^0$-composition factor, with associated permutation $\s \in S_t$, say (see \eqref{e:star}). If $l \ge 2$ then $\s(1)=2$ is the only possibility and thus $a_1=1$. Similarly, if $l=1$ and $t \ge 3$ then either $\s(1)=2$,  or $\s(1)=t$ and $\s(i)=1$ for some $i \in \{2, \ldots, t-1\}$. Once again, we conclude that $a_1=1$. Finally, if $(l,t)=(1,2)$ and $a_1 \ge 2$ then $\nu = \l - 2\a_1-2\a_2 \in \L(V)$ affords the highest weight of a $KH^0$-composition factor, which implies that $V|_{H^0}$ has at least three distinct composition factors (since $\l$ and $\l-\a_1-\a_2$ also afford highest weights of $KH^0$-composition factors). This is a contradiction since $|H:H^0|=2$.

Similarly, if $a_n \neq 0$ then
$\mu  = \l-\a_{n-l} - \cdots - \a_{n} \in \L(V)$
affords the highest weight of a $KH^0$-composition factor, and by considering $\mu|_{H^0}$ we quickly deduce that $a_n=1$ is the only possibility when $(l,t) \neq (1,2)$. On the other hand, if $(l,t)=(1,2)$ and $a_3 \ge 2$ then $\nu = \l - 2\a_2-2\a_3 \in \L(V)$ also affords the highest weight of a $KH^0$-composition factor, so there are at least three composition factors, but this cannot happen since $|H:H^0|=2$. We have now reduced to the cases $\l=\l_1$, $\l_n$ and $\l_1+\l_n$.

Of course, the highest weights $\l=\l_1$ and $\l_n$ are genuine examples since $L_G(\l_1)$ is the natural $KG$-module, and $L_G(\l_{n})$ its dual. Finally, if $\l=\l_1+\l_n$ then $\mu = \l -\a_1 - \cdots - \a_{n-l} \in \L(V)$ affords the highest weight of a $KH^0$-composition factor, but it is easy to see that there is no compatible associated permutation.
\end{proof}

Before we continue the general analysis of the cases in Table \ref{t:c2s}, let us first deal with two special cases that arise when $H$ is of type $(2^{t-1} \times B_l^t).S_t$.

\begin{lem}\label{l:c2_special}
Suppose $G=B_n$ or $D_n$, and let $H$ be a $\C_2$-subgroup of type $(2^{t-1} \times B_l^t).S_t$ as in cases (ii) and (iv) of Table \ref{t:c2s}. Let $V$ be a spin module for $G$, so 
$V=L_G(\l_n)$ if $G = B_n$, and $V=L_G(\l_{n-1})$ or $L_G(\l_n)$ if $G=D_n$. Then $V|_{H}$ is irreducible, and $V|_{H^0}$ is the sum of $2^{\lfloor (t-1)/2\rfloor}$ irreducible $KH^0$-modules
each with highest weight $\sum_{i=1}^t\o_{i,l}$, that is, the tensor product of the spin
modules for each of the simple factors of $H^0$.
\end{lem}

\begin{proof}
Here $p \neq 2$, $l \ge 1$, $t \ge 2$ and $H$ stabilizes a direct sum decomposition
$$W = W_1 \oplus W_2 \oplus \cdots \oplus W_t$$
of the natural $KG$-module $W$, where $\dim W_i = 2l+1$ for all $i$. 
Moreover, $G=B_n$ if $t$ is odd, otherwise $G=D_n$ (see Table \ref{t:c2s}). 
We proceed by induction on $t$.

If $t=2$ then the main theorem of \cite{Seitz2} implies that $V$ is an irreducible $KH^0$-module, with $T_{H^0}$-highest weight $\o_{1,l}+\o_{2,l}$ (see the case labelled ${\rm IV}_{2}$ in \cite[Table 1]{Seitz2}), so let us assume $t>2$. Set 
$$U_1 = W_1 \oplus W_2 \oplus \cdots \oplus W_{t-1},\;\; U_2 = W_t$$
and note that 
$$H^0 = B_{l}^{t-1}.B_{l} < {\rm SO}(U_1) \times {\rm SO}(U_2),$$
where $\dim U_1 = \dim W - 2l-1$ and $\dim U_2 = 2l+1$. 

Suppose $t \ge 4$ is even. Here $\dim U_1$ is odd and $N_G({\rm SO}(U_1) \times {\rm SO}(U_2)) = {\rm SO}(U_1) \times {\rm SO}(U_2)$ acts irreducibly on $V$ by the main theorem of \cite{Seitz2}. More precisely, ${\rm SO}(U_1) \times {\rm SO}(U_2)$ acts on $V$ as the tensor product of two spin modules. By induction we see that $N_{{\rm SO}(U_1)}(B_l^{t-1})$ acts irreducibly on the spin module for ${\rm SO}(U_1)$, and there are precisely 
$2^{t/2-1}$ distinct $K(B_l^{t-1})$-composition factors on this module, each with highest weight $\sum_{i=1}^{t-1}\o_{i,l}$. Since ${\rm SO}(U_2)=B_l$, the irreducibility of $V|_{H}$ follows immediately, and so does the desired description of the $KH^0$-composition factors of $V$.

Finally, suppose $t$ is odd. Now $\dim U_1$ is even and Proposition \ref{TH:C1,3,6} implies that $N_G({\rm SO}(U_1) \times {\rm SO}(U_2)) = ({\rm SO}(U_1)\times {\rm SO}(U_2))\la \tau \ra$ acts irreducibly on $V$, where $\tau$ acts as a reflection on both $U_1$ and $U_2$. Moreover, the connected component ${\rm SO}(U_1)\times {\rm SO}(U_2)$ has precisely two composition factors on $V$ (each of which is the tensor product of appropriate spin modules), which are interchanged by $\tau$. 
We can choose $\tau \in H$, so it suffices
to show that $N_{{\rm SO}(U_1)}(B_l^{t-1})$ acts irreducibly on the two spin modules for ${\rm SO}(U_1)$, and that there are precisely 
$2^{(t-3)/2}$ distinct $K(B_l^{t-1})$-composition factors on each of these modules, all with highest weight $\sum_{i=1}^{t-1}\o_{i,l}$. But this follows from the inductive hypothesis.
\end{proof}

\begin{lem}\label{l:c2_special_2}
Suppose $G=B_n$ or $D_n$, and let $H$ be a $\C_2$-subgroup of type $(2^{t-1} \times B_l^t).S_t$ as in cases (ii) and (iv) of Table \ref{t:c2s}. Let $V = L_G(\l_1+\l_n)$. Then $V|_{H^0}$ has at least $(t-2)2^{\lfloor (t-1)/2\rfloor}$ composition factors of highest weight $\sum_{i=1}^t\o_{i,l}$.
\end{lem}

\begin{proof}
We consider the tensor product $M=L_G(\l_1)\otimes L_G(\l_n)$. Set $\mu =  \l_1+\l_n-\sum_{i=1}^n\a_i$ if $G=B_n$, and $\mu =  \l_1+\l_n-\sum_{i=1}^{n-2}\a_i-\a_n$ if $G = D_n$. (Note that $\mu = \l_n$ if $G=B_n$, and $\mu=\l_{n-1}$ if $G=D_n$.)
Then $\mu$ occurs with multiplicity $n+\delta$ in $M$, where $\delta=1$ if $G=B_n$, and $\delta=0$ if $G=D_n$. By applying Lemmas \ref{l:s816} -- \ref{l:sr}, we see that $m_V(\mu) = n-\epsilon$, where
$$\epsilon=\left\{\begin{array}{ll} 
0& \mbox{if $G = B_n$, $p \nmid 2n+1$}\\
1&\mbox{if $G=B_n$,  $p \mid 2n+1$}\\
1&\mbox{if $G=D_n$, $p \nmid n$}\\
2&\mbox{if $G=D_n$,  $p \mid n$}\\
\end{array}\right.$$

Note that $\mu$ is the only subdominant weight to $\l_1+\l_n$ occurring in $M$.
By comparing the above multiplicities, we deduce that $M$ has $1$ or $2$ $KG$-composition factors with highest weight $\mu$, 
in addition to the composition factor isomorphic to $V$. Moreover, there are no other $KG$-composition factors. 

Now we consider the action of $H^0$ on the two tensor factors of $M$. The natural $KG$-module $L_G(\l_1)$ decomposes as a sum of $t$ $(2l+1)$-dimensional (natural) modules for each of the factors $B_l$. Let $W_i$ denote the natural module for the $i$-th factor of $H^0$.
By Lemma \ref{l:c2_special}, the spin module $L_G(\l_n)$ decomposes as the sum of $2^{\lfloor (t-1)/2 \rfloor}$ irreducible
$KH^0$-modules with highest weight $\sum_{i=1}^t\omega_{i,l}$.
Now each of the $2^{\lfloor (t-1)/2 \rfloor}$ tensor factors 
$W_i\otimes L_{H^0}(\sum_{i=1}^t\omega_{i,l})$ has a $KH^0$-composition factor with highest weight 
$\sum_{j=1}^t\omega_{j,l}$. So $M$ has at least 
$t\cdot2^{\lfloor (t-1)/2 \rfloor}$ such composition factors. We conclude that $V|_{H^0}$ has at least 
$(t-2)2^{\lfloor (t-1)/2 \rfloor}$
such composition factors, as required.
\end{proof}

\begin{lem}\label{c2:p2}
Proposition \ref{TH:C2} holds in case (ii) of Table \ref{t:c2s}.
\end{lem}

\begin{proof}
Here $G=B_n$, $t \ge 3$ is odd, $H=(2^{t-1}\times B_l^t).S_t$, $n = lt+(t-1)/2$ and $p \neq 2$. 
We may obtain expressions for the root subgroups of $H^0$ by taking
a maximal rank subsystem subgroup of $G$ of type $D_{n-l}B_l$, to which we 
apply the construction of \cite[Claim 8]{Test2} and induction within the $D_{n-l}$ factor.

We begin with the case $l=1$. Here $n=(3t-1)/2$ and $H^0 = X_1 \cdots X_t$ with $X_{i}\cong B_1$. We may choose simple roots $\{\b_1,\ldots, \b_t\}$ for $H^0$ so that $x_{\pm\b_{t}}(c) = x_{\pm\a_n}(c)$ and $x_{\pm\b_{i}}(c) = x_{\pm\gamma_{i}}(c)x_{\pm\delta_{i}}(c)$ for all $i<t$ and $c \in K$, where
$$\gamma_{2k-1} = \a_{3k-2}+\a_{3k-1}+2\sum_{j=3k}^{n}\a_j,\;\; \delta_{2k-1} = \a_{3k-2}+\a_{3k-1}$$
and
$$\gamma_{2k} = \a_{3k-1}+2\sum_{j=3k}^{n}\a_j,\;\; \delta_{2k} = \a_{3k-1}.$$

Let $\o_{i}$ be the fundamental dominant weight of $X_{i}$ corresponding to the simple root $\b_i$. With the above choice of embedding, we deduce that if $\mu = \sum_{i=1}^{n}{b_i\l_i}$ is a weight for $T$ then $\mu|_{X_{i}} = v_{i}(\mu)\o_i$, where $v_{t}(\mu) = b_n$ and
$$v_{2k-1}(\mu) = b_n+2\sum_{j=3k-2}^{n-1}b_j,\;\; v_{2k}(\mu) = b_n+2\sum_{j=3k-1}^{n-1}b_j.$$
Let $\l=\sum_{i=1}^{n}a_i\l_i$ be the highest weight of $V$ and set $v_{i}=v_{i}(\l)$ for all $1 \le i \le t$. Assume $V|_{H}$ is irreducible.

Suppose $a_{3k} \neq 0$, where $k \ge 1$. Then $\mu = \l - \a_{3k} \in \L(V)$ affords the highest weight of a $KH^0$-composition factor, but this contradicts Lemma \ref{l:eh} since $\mu|_{H^0} = \l|_{H^0}+2\o_{2k+1}$. Similarly, if $a_{3k+1} \neq 0$, where $k \ge 1$ and $3k+1<n$ then $\l - \a_{3k}-\a_{3k+1} \in \L(V)$ affords the highest weight of a $KH^0$-composition factor, but once again this is not possible since this weight restricts to $\l|_{H^0} +2\o_{2k+2}$. Also, if $a_{3k-1} \neq 0$ and $k \ge 2$ then $\mu = \l - \a_{3k-3} - \a_{3k-2} - \a_{3k-1}-\a_{3k} \in \L(V)$ affords the highest weight of a $KH^0$-composition factor, but this contradicts Lemma \ref{l:eh} since $\mu|_{H^0} = \l|_{H^0}+2\o_{2k+1}$. We have now reduced to the case
$$\l = a_1\l_1 + a_2\l_2 + a_{n}\l_n.$$

Suppose $a_n \ge 2$. Then $\mu =
\lambda-\alpha_{n-1} - 2\alpha_n \in \L(V)$ and we claim that $m_{V}(\mu)=2$. To see 
this, let $L$ be the Levi subgroup of $G$ with derived subgroup 
$$Y=\la U_{\pm \a_{n-1}},U_{\pm\a_{n}}\ra \cong B_2.$$ 
By Lemma \ref{l:sr}, we have $m_{V}(\mu) = m_{V'}(\mu')$, where $V'=L_{Y}(a_n\eta_2)$, 
$\mu' = \mu|_{Y}$  and $\eta_1,\eta_2$ are the fundamental dominant weights for $Y$ 
corresponding to $\a_{n-1}$ and $\a_n$, respectively. 
By \cite[Table 1]{Seitz2}, the irreducible $KD_3$-module with highest
weight $\zeta = a_n\zeta_3$ restricts irreducibly to the natural $B_2$ subgroup
to give the module with highest weight $a_n\eta_2$, where $\zeta_1,\zeta_2,\zeta_3$ are 
fundamental dominant weights for $D_3$. Now, the weights of $L_{D_3}(\zeta)$ that
 restrict to $B_2$ to give $\mu'$ are 
$\zeta-\alpha'_1-2\alpha'_3$
and $\zeta-\alpha'_1-\alpha'_2-\alpha'_3$, where the $\alpha'_i$ are appropriate simple 
roots of $D_3$. By Lemma \ref{l:s114}, both of these weights occur with multiplicity $1$, whence $m_{V}(\mu)=2$ as claimed. 

By the PBW theorem (see \cite[Section 17.3]{Hu1}), a basis for the weight space $V_{\mu}$ is given by $\{f_{\a_{n-1}+\a_n}f_{\a_n}v^+,f_{\a_{n-1}+2\a_n}v^+\}$, where $v^+ \in V$ is a maximal vector for the fixed Borel subgroup of $G$. Set $w = af_{\a_{n-1}+\a_n}f_{\a_n}v^++bf_{\a_{n-1}+2\a_n}v^+$ for scalars $a,b \in K$, and apply the generating elements of the Borel subgroup $B_{H^0} = \la T_{H^0}, U_{\b_i} \mid 1 \le i \le t\ra$ of $H^0$.
We have $U_{\b_i}w=w$ for $i<t$, and 
\begin{align*}
x_{\b_t}(c)w & = w+ ce_{\a_n}w \\
& =  w+c(a[e_{\a_n}, f_{\a_{n-1}+\a_n}]f_{\a_n}v^+ +
af_{\a_{n-1}+\a_n}a_nv^+ +b[e_{\a_n}, f_{\a_{n-1}+2\a_n}]v^+)
\end{align*}
for all $c \in K$. Therefore
$$x_{\b_t}(c)w =w+c(a\gamma f_{\a_{n-1}+\a_n}v^++af_{\a_{n-1}+\a_n}a_nv^++b\delta f_{\a_{n-1}+\a_n}v^+)$$ 
for some non-zero scalars $\gamma,\delta$. In particular, if we set $b = -a(\gamma+a_n)/\delta$ then $U_{\b_t}w = w$, so $w$ is a maximal vector for $B_{H^0}$ of $T_{H^0}$-weight $\mu|_{H^0}$ and thus $\mu$ affords the highest weight of a $KH^0$-composition factor of $V$. This contradicts Lemma \ref{l:eh} since $\mu|_{H^0} = \l|_{H^0} - 2\o_t$. We conclude that $a_n \le 1$.

Next suppose $a_1+a_2 \neq 0$. Then $\mu = \l-\a_1-\a_2-\a_3 \in \L(V)$ affords the
 highest weight of a $KH^0$-composition factor and $\mu|_{H^0} = \l|_{H^0} - 2\o_1+2\o_3$. 
By considering the possibilities for the associated permutation we quickly deduce that 
$a_1+a_2=1$.

Now assume $a_1a_n \neq 0$, so $\l=\l_1+\l_n$. Then $\l|_{H^0} = 3\o_1+\sum_{i=2}^t\o_i$. However, Lemma \ref{l:c2_special_2} implies that $V|_{H^0}$ has at least one composition factor with highest weight $\sum_{i=1}^{t}\o_i$ (recall that $t \ge 3$), which is not conjugate to $\l|_{H^0}$. This contradiction eliminates the case $\l=\l_1+\l_n$.

Next suppose $a_2a_n\ne 0$, so $\lambda=\lambda_2+\lambda_n$ and $\lambda|_{H^0} = 3\omega_1+3\omega_2+\sum_{i=3}^t\omega_i$. Set $J = \langle U_{\pm\b_i} \mid 2\le i\le t\rangle$. We will consider the restriction $V|_J$. Note that $J\leqs L=\langle U_{\pm\a_k} \mid 2\le k\le n\rangle$, the derived subgroup of a Levi factor of a parabolic subgroup of $G$. More precisely, we have $J\leqs Y\leqs L$, where $Y=D_{n-1}$ is the stabilizer in $L$ of a non-degenerate $1$-space in the natural module for $L$. To simplify the notation, let us write 
$\{\eta_1,\ldots,\eta_{n-1}\}$ for the base of the root system of $Y$, and $\{\rho_1,\ldots,\rho_{n-1}\}$ for the associated fundamental dominant weights. By \cite[Proposition 2.11]{Jantzen}, $V|_L$ has a composition factor with highest weight $(\l_2+\l_n)|_L$, which yields a $KY$-composition factor of $V|_{Y}$ with highest weight $\rho_1+\rho_{n-1}$. We can now restrict this to the 
subgroup $J$, which is the connected component of a $\C_2$-subgroup of $Y$ of type $(2^{t-2} \times B_1^{t-1}).S_{t-1}$. By applying Lemma \ref{l:c2_special_2}, we see that 
$L_Y(\rho_1+\rho_{n-1})|_J$ has a composition factor with highest weight $\sum_{i=2}^t\omega_{i}$, but this contradicts the form of the highest weights of the  $KH^0$-composition factors of $V$.

We have now reduced to the cases $\l=\l_1$, $\l_2$ and $\l_n$. The case $\l=\l_1$ is an example since $V$ is simply the natural module for $G$. Similarly, $V|_{H}$ is irreducible if $\l=\l_n$ by Lemma \ref{l:c2_special}. Finally, suppose $\l=\l_2$. Here $\dim V = n(2n+1)$ by Proposition \ref{p:dims},
and each $KH^0$-composition factor is $9$-dimensional. By considering the permuting action of the symmetric group $S_t$ we see that there are at least $\binom{t}{2}$ distinct composition factors. However, there cannot be exactly $\binom{t}{2}$ factors since $\dim V > 9\binom{t}{2}$, but if there are more then $H$ must permute the homogeneous summands of $V|_{H^0}$ and this is not possible since $\dim V < 18\binom{t}{2}$. We conclude that $V|_{H}$ is reducible if $\l=\l_2$.

\vs

For the remainder we may assume $l \ge 2$. Write $H^0 = X_1 \cdots X_t$, where $X_i \cong B_l$, and let $\{\b_{i,1}, \ldots, \b_{i,l}\}$ be a set of simple roots for $X_{i}$. As before, we obtain expressions for the root groups of $H^0$ via induction and the construction of \cite[Claim 8]{Test2}. 
In this way we get
$$x_{\pm\b_{i,j}}(c) = x_{\pm\a_{(i-1)l+\lfloor (i-1)/2 \rfloor +j}}(c),\;\; x_{\pm\b_{i,l}}(c) = x_{\pm\gamma_{i}}(c)x_{\pm\delta_{i}}(c),$$
for all $c \in K$, $i<t$ and $1 \le j \le l-1$ with
$$\gamma_{2k-1} = \sum_{j=0}^{l}\a_{(2k-1)l+k+j-1},\;\; \delta_{2k-1} = \gamma_{2k-1}+2\sum_{j=k(2l+1)}^n\a_j$$
and
$$\gamma_{2k} = \a_{k(2l+1)-1},\;\; \delta_{2k} = \gamma_{2k}+2\sum_{j=k(2l+1)}^n\a_j.$$
For $i=t$ we have
$$x_{\pm\b_{t,j}}(c) = x_{\pm\a_{n-l+j}}(c)$$
for all $c \in K$ and $1 \le j \le l$.
It follows that if $\mu=\sum_{i=1}^{n}{b_{i}\l_{i}} \in \L(V)$ then
$$\mu|_{X_i} = \sum_{j=1}^{l-1}b_{(i-1)l+\lfloor (i-1)/2 \rfloor +j}\o_{i,j}+v_i(\mu)\o_{i,l},\;\; \mu|_{X_t} = \sum_{j=1}^{l}b_{n-l+j}\o_{t,j}$$
for all $i<t$, where
$$v_{2k-1}(\mu) = b_n + 2\sum_{j=(2k-1)l+k-1}^{n-1}b_j,\;\; v_{2k}(\mu) = b_n + 2\sum_{j=k(2l+1)-1}^{n-1}b_j.$$
We set $v_t = a_n$, and $v_i = v_i(\l)$ for all $1 \le i < t$.

Suppose $a_{ml+\lfloor m/2 \rfloor} \neq 0$ for some $m \in \{1, \ldots, t-1\}$. If $m=2k$ is even then $\mu = \l - \a_{2kl+k} \in \L(V)$ affords the highest weight of a $KH^0$-composition factor, but this contradicts Lemma \ref{l:eh} since $\mu|_{H^0} = \l|_{H^0}+\o_{2k+1,1}$. Similarly, if $m=2k-1$ is odd then $\mu = \l - \a_{(2k-1)l+k-1} \in \L(V)$ affords 
the highest weight of a $KH^0$-composition factor and
$$\mu|_{H^0} = \l|_{H^0}+\o_{2k-1,l-1}-2\o_{2k-1,l}+\o_{2k,1}.$$
Let $\s \in S_t$ be the associated permutation. If $\s(i) \ge 2k$ for some $i<2k$ then $v_{2k-1} = v_{2k}$, but this is not possible since we are assuming $a_{(2k-1)l+k-1} \neq 0$. Therefore $\{1, \ldots, 2k-1\}$ is $\s$-invariant, but this is ruled out by Lemma \ref{l:eh}.

For the time being, let us assume $l \ge 3$. Suppose $r \in \{1, \ldots, l-2\}$ is minimal such that $a_{l-r} \neq 0$. Then $\mu = \l - \a_{l-r} - \a_{l-r+1} - \cdots - \a_{l} \in \L(V)$ affords the highest weight of a 
$KH^0$-composition factor and
$$\mu|_{H^0} = \l|_{H^0} + \o_{1,l-r-1} - \o_{1,l-r} + \o_{2,1}.$$
This contradicts Lemma \ref{l:eh}. Similarly, if  $r \in \{1, \ldots, l-2\}$ is minimal such that $a_{n-l+r} \neq 0$ then $\mu = \l - \a_{n-l} - \a_{n-l+1} - \cdots - \a_{n-l+r} \in \L(V)$ affords the highest weight of a $KH^0$-composition factor and
$\mu|_{H^0} = \l|_{H^0} - \o_{t,r}+\o_{t,r+1}$.
If $\s \in S_t$ denotes the associated permutation then clearly $\s(t) \neq t$, whence $v_{t} = v_{t-1}$ but this is not possible since $a_{n-l+r} \neq 0$. Finally, if $a_{n-1} \neq 0$ then $\mu = \l - \a_{n-l} - \cdots - \a_{n-1} \in \L(V)$ affords the highest weight of a $KH^0$-composition factor, but this contradicts Lemma \ref{l:eh} since $\mu|_{H^0} = \l|_{H^0} -\o_{t,l-1}+2\o_{t,l}$.

Next, for each odd integer $m \in \{1, \ldots, t-2\}$, say $m=2k-1$, suppose $s \in \{1, \ldots, l-2\}$ is minimal such that $a_{(2k-1)l+k-1+s} \neq 0$. Then
$$\mu = \l - \a_{(2k-1)l+k-1} - \a_{(2k-1)l+k} - \cdots - \a_{(2k-1)l+k-1+s} \in \L(V)$$
affords the highest weight of a $KH^0$-composition factor, contradicting Lemma \ref{l:eh} since
$$\mu|_{H^0} = \l|_{H^0}+\o_{2k-1,l-1}-2\o_{2k-1,l}-\o_{2k,s}+\o_{2k,s+1}.$$
Similarly, if $s \in \{1, \ldots, l-2\}$ is minimal such that $a_{(2k-1)l+k-1-s} \neq 0$ then
$$\mu = \l - \a_{(2k-1)l+k-1-s} - \a_{(2k-1)l+k-s} - \cdots - \a_{(2k-1)l+k-1} \in \L(V)$$
affords the highest weight of a $KH^0$-composition factor, but this is ruled out by Lemma \ref{l:eh} since
$$\mu|_{H^0} = \l|_{H^0}+\o_{2k-1,l-s-1}-\o_{2k-1,l-s}+\o_{2k,1}.$$
For each even integer $m \in \{1, \ldots, t-1\}$, say $m=2k$, suppose $s \in \{1, \ldots, l-1\}$ is minimal such that $a_{2kl+k-s} \neq 0$ then
$$\mu = \l - \a_{2kl+k-s} - \a_{2kl+k-s+1} - \cdots - \a_{2kl+k} \in \L(V)$$
affords the highest weight of a $KH^0$-composition factor, with associated permutation $\s \in S_t$. If $s \ge 2$ then this contradicts Lemma \ref{l:eh} since 
$$\mu|_{H^0} = \l|_{H^0} +\o_{2k,l-s}-\o_{2k,l-s+1}+\o_{2k+1,1}.$$ 
On the other hand, if $s=1$ then 
$$\mu|_{H^0} = \l|_{H^0} +\o_{2k,l-1}-2\o_{2k,l}+\o_{2k+1,1}.$$ 
However, if $\s(i)>2k$ for some $i \le 2k$ then $v_{2k}=v_{2k+1}$, which is not possible since $a_{2kl+k-1} \neq 0$. Therefore $\{1, \ldots, 2k\}$ is $\s$-invariant, but this contradicts Lemma \ref{l:eh}.

Finally, suppose $t \ge 5$ and $m \in \{2, \ldots, t-3\}$ is the largest even integer, $m=2k$ say, such that $a_{2kl+k +s} \neq 0$ for some $s \in \{1, \ldots, l-2\}$. Assume 
$s \in \{1, \ldots, l-2\}$ is minimal with respect to the property $a_{2kl+k +s} \neq 0$.
Then
$$\mu = \l - \a_{2kl+k} - \a_{2kl+k+1} - \cdots - \a_{2kl+k+s} \in \L(V)$$
affords the highest weight of a $KH^0$-composition factor and
$\mu|_{H^0} = \l|_{H^0} - \o_{2k+1,s} + \o_{2k+1,s+1}$. Let $\s \in S_t$ be the associated permutation. By the maximality of $m$, and our earlier deductions, we have $a_{(i-1)l+\lfloor (i-1)/2\rfloor+s}=0$ for all $i>2k+1$, whence $\s(2k+1)\le 2k$ and thus $v_{2k+1} = v_{2k}$. This is a contradiction since $a_{2kl+k+s} \neq 0$. 
Therefore, for $l \ge 3$, we have reduced to the case $\l = a_1\l_1+a_n\l_n$.

We can also reduce to $\l = a_1\l_1+a_n\l_n$ when $l=2$. To see this, first recall that we have already established $a_{2m+\lfloor m/2 \rfloor}=0$ for all $m \in \{1, \ldots, t-1\}$. Suppose $a_{5k+1} \neq 0$ for some $k \ge 1$. Then $\mu = \l - \a_{5k} - \a_{5k+1} \in \L(V)$ and $\mu|_{H^0} = \l|_{H^0}-\o_{2k+1,1}+2\o_{2k+1,2}$, but this contradicts Lemma \ref{l:eh} since $\mu$ affords the highest weight of a $KH^0$-composition factor. Similarly, if $a_{5k-1} \neq 0$ then $\mu = \l - \a_{5k-1} - \a_{5k} \in \L(V)$ and $\mu|_{H^0} = \l|_{H^0}+\o_{2k,1}-2\o_{2k,2}+\o_{2k+1,1}$ is the highest weight of a $KH^0$-composition factor. If $\s \in S_t$ denotes the associated permutation then $\{1,\ldots, 2k\}$ is $\s$-invariant since $a_{5k-1} \neq 0$, but this contradicts Lemma \ref{l:eh}. Finally, if $a_{5k-2} \neq 0$ for some $k$ then
$$\mu = \l - \a_{5k-3} - \a_{5k-2} - \a_{5k-1} - \a_{5k} \in \L(V)$$
affords the highest weight of a $KH^0$-composition factor and
$$\mu|_{H^0} = \l|_{H^0}+\o_{2k-1,1}-2\o_{2k-1,2}+\o_{2k+1,1}.$$
Here $\{1, \ldots, 2k-1\}$ has to be $\s$-invariant, but once again this contradicts Lemma \ref{l:eh}.

Finally, let us assume $\l = a_1\l_1+a_n\l_n$ and $l \ge 2$. 
If $a_1 \neq 0$ then $\l - \a_1-\cdots-\a_l \in \L(V)$ affords the 
highest weight of a $KH^0$-composition factor and 
$\mu|_{H^0} = \l|_{H^0} - \o_{1,1} + \o_{2,1}$. 
If $\s \in S_t$ is the associated permutation then $\s(1)=2$ is the
 only possibility, whence $a_1=1$. 
Similarly, by arguing as in the case $l=1$, we deduce that $a_n \le 1$. 
In particular, if $a_{1}a_n \neq 0$ then $\l=\l_1+\l_n$ and $\l|_{H^0} = \o_{1,1}+\sum_{i=1}^t\o_{i,l}$. But Lemma \ref{l:c2_special_2} shows that $V|_{H^0}$ has at least one composition factor with highest weight $\sum_{i=1}^t\o_{i,l}$, and this is a contradiction. 

We have now reduced to the case $\l=\l_1$ or $\l_n$. If $\l=\l_1$ then $V$ is the natural module for $G$ and $V|_{H}$ is irreducible. By Lemma \ref{l:c2_special}, the same conclusion also holds when $\l=\l_n$. These cases are recorded in Table \ref{t:c2}.
\end{proof}

\begin{lem}\label{c2:p3}
Let $G=C_n$ and let $H$ be a $\C_2$-subgroup of type $C_{l}^t.S_t$ as in case (iii) of Table \ref{t:c2s}. Suppose $\l \neq \l_{n-1}+a\l_n$ for any $a\ge 0$. Then $V|_{H}$ is irreducible if and only if $\l = \l_1$, or $p=2$ and $\l = \l_n$.
\end{lem}

\begin{proof}
Here $n=lt$ and $H^0 = X_1 \cdots X_t$, where up to conjugacy we have
$$X_{i}=\la U_{\pm \a_{(i-1)l+1}}, \ldots, U_{\pm \a_{(i-1)l+l-1}}, U_{\pm \gamma_{i}} \ra \cong C_{l}$$
with $\gamma_{t}=\a_n$ and
$\gamma_{i}=2(\a_{il}+ \cdots +\a_{n-1})+\a_n$
for $i<t$.
Let $\{\o_{i,1},\ldots,\o_{i,l}\}$ be the fundamental dominant weights corresponding to this base of the root system of $X_i$, so if
$\mu = \sum_{i=1}^{n}{b_i\l_i}$ is a weight for $T$ then
\begin{equation}\label{e:sp}
\mu|_{X_{i}} = \sum_{j=1}^{l-1}\la \mu, \a_{(i-1)l+j}\ra \o_{i,j} + \la \mu, \gamma_{i}\ra\o_{i,l}  = \sum_{j=1}^{l-1}{b_{(i-1)l+j}\o_{i,j}}+v_i(\mu)\o_{i,l},
\end{equation}
where
$$v_{i}(\mu)=  \la \mu, \gamma_{i}\ra = \sum_{j=il}^{n}b_j.$$
Let $\l=\sum_{i=1}^{n}a_i\l_i$ be the highest weight of $V$ and set $v_{i}=v_{i}(\l)$ for all $i$. Assume $V|_{H}$ is irreducible.

Consider the case $l=1$. Here $n=t$ and for now we will assume $t \ge 3$. Let $\o_{i}=\o_{i,1}$, so \eqref{e:sp} reads $\mu|_{X_{i}} = v_{i}(\mu)\o_{i}$, with $v_{i}(\mu) = \sum_{j=i}^{n}{b_j}$.

Let us assume $a_{j} \neq 0$ for some $j<n$. Then $\mu = \l-\a_j \in \L(V)$ affords the highest weight of a $KH^0$-composition factor and $\mu|_{H^0} = \l|_{H^0}-\o_{j}+\o_{j+1}$. By considering the possible associated permutations, we quickly deduce that $a_j=1$.
Next suppose $\l$ is of the form
$$\l = \l_j+\l_{j+k}+a_{j+k+1}\l_{j+k+1} + \cdots + a_n\l_n,$$
where $j,k \ge 1$ and $j+k <n$. Then
$\mu = \l-\a_j - \cdots - \a_{j+k} \in \L(V)$
affords the highest weight of a $KH^0$-composition factor and
$\mu|_{H^0} = \l|_{H^0}-\o_{j}+\o_{j+k+1}$. Let $\s \in S_t$ be the associated permutation and fix $r \le j$ such that $\s(r) \ge j$. If $\s(r)=j$ then $v_r=v_j-1$, which is absurd since $v_r \ge v_j$. If $j<\s(r) \le j+k$ then $v_j = v_{j+k}$, which is impossible since $v_j = v_{j+k}+1$. Therefore $\s(r) \ge j+k+1$, but once again we reach a contradiction since $v_r > v_{j+k+1}+1$. For $l=1$ (with $t \ge 3$) we have now reduced to the case
$\l=a_j\l_j+a_n\l_n$ with $a_j \le 1$.

Let $\l=\l_j+a_n\l_n$. By the hypothesis of the lemma, we may assume $j<n-1$. Suppose $a_n \neq 0$. Since $V$ is tensor-indecomposable, \cite[1.6]{Seitz2} implies that $p \neq 2$, so $\mu = \l - \a_{n-1} -\a_{n} \in \L(V)$ by Lemma \ref{l:pr}, since it is a dominant weight for $T$. Moreover, $\mu$ affords the highest weight of a $KH^0$-composition factor, but $h(\mu|_{H^0}) = h(\l|_{H^0})-2$ and this contradicts Lemma \ref{l:eh}.

To complete the analysis of the case $l=1$ (with $t \ge 3$) we may assume $\l= \l_j$ or $a_n\l_n$, where $1 \le j \le n-2$. If $\l=\l_1$ then $V$ is the natural $KG$-module and thus $V|_{H}$ is irreducible. Next suppose $\l=\l_j$ with $2 \le j \le n-2$. Set $\l_0 = 0$. Then
$$\mu = \l-\a_{j-1}-\a_{n}-2\sum_{i=j}^{n-1}\a_i = \l_{j-2} \in \L(V)$$
(see \cite[Theorem 15]{Zal}) and
$\mu$ affords the highest weight of a $KH^0$-composition factor, but this contradicts Lemma \ref{l:eh} since $h(\mu|_{H^0}) = h(\l|_{H^0})-2$. Finally, let us assume $\l=a_n\l_n$. As before, if $p \neq 2$ then $\mu = \l - \a_{n-1} -\a_{n} \in \L(V)$ affords the highest weight of a $KH^0$-composition factor and once again we reach a contradiction via Lemma \ref{l:eh}. Therefore we may assume $p=2$ and $a_n=1$. Here $V|_{H^0}$ is irreducible (see \cite{Seitz2}), so the case $\l=\l_n$ with $p=2$ is recorded in Table \ref{t:c2}.

Finally, let us assume $(l,t)=(1,2)$. If $a_1 \neq 0$ then $\mu = \l-\a_1\in \L(V)$ affords the highest weight of a $KH^0$-composition factor and we deduce that $a_1=1$ since $\mu|_{H^0} = \l|_{H^0}-\o_1+\o_2$. However, the case $a_1=1$ is excluded by the hypothesis of the lemma, so we have reduced to the case $a_1=0$ and $a_2 \neq 0$. Here $\mu = \l - \a_1 - \a_2$
is dominant. In particular, if $p \neq 2$ then $\mu \in \L(V)$ and $\mu$ affords the highest weight of a $KH^0$-composition factor. However, $h(\mu|_{H^0}) = h(\l|_{H^0})-2$ so we may assume $p=2$ and $a_2=1$. Here $V|_{H^0}$ is irreducible, so $\l=\l_2$ with $p=2$ is an example as before.

\vs

For the remainder, we may assume $l \ge 2$. If $a_{ml}\neq 0$ for some $m \in \{1,\ldots, t-1\}$ then $\mu = \l-\a_{ml}\in \L(V)$ affords the highest weight of a $KH^0$-composition factor, but this contradicts Lemma \ref{l:eh} since $\mu|_{H^0} = \l|_{H^0}+\o_{m,l-1}-\o_{m,l}+\o_{m+1,1}$. Similarly, if
$r \in \{1, \ldots, l-2\}$ is minimal such that $a_{l-r}\neq 0$ then $\mu = \l-\a_{l-r}-\a_{l-r+1} - \cdots - \a_{l}\in \L(V)$
affords the highest weight of a $KH^0$-composition factor, but
$\mu|_{H^0} = \l|_{H^0}+\o_{1,l-r-1}-\o_{1,l-r}+\o_{2,1}$ and this contradicts Lemma \ref{l:eh}. Now if $r \in \{1, \ldots, l-2\}$ is minimal such that $a_{l+r} \neq 0$ then
$\mu = \l-\a_{l}-\a_{l+1} - \cdots - \a_{l+r} \in \L(V)$
affords the highest weight of a $KH^0$-composition factor and
$$\mu|_{H^0} = \l|_{H^0}+\o_{1,l-1}-\o_{1,l}-\o_{2,r}+\o_{2,r+1}.$$
If $\s \in S_t$ is the associated permutation then $\s(1)>1$ and thus $v_{1}=v_{2}$, which is a contradiction since $a_{l+r} \neq 0$. We now have $a_{i}=0$ for all $i \in \{2, \ldots, 2l-2\}$.

Next suppose $s \in \{2l-1, \ldots, n-2\}$ is minimal such that $a_{s} \neq 0$. First assume $l$ does not divide $s+1$, say $s=kl+r$ for some $k \in \{2, \ldots, t-1\}$ and $r \in \{1, \ldots, l-2\}$ (recall that $a_{ml}=0$ for all $m \in \{1, \ldots, t-1\}$, so $r\ge 1$). Then
$\mu = \l-\a_l-\a_{l+1} - \cdots - \a_{kl+r} \in \L(V)$
affords the highest weight of a $KH^0$-composition factor and
$$\mu|_{H^0} = \l|_{H^0}+\o_{1,l-1}-\o_{1,l}-\o_{k+1,r}+\o_{k+1,r+1}.$$
If $i \le k$ and $\s(i)>k$ then $v_{k}=v_{k+1}$, but this is not possible since $a_{kl+r}\neq 0$. Therefore $\{1, \ldots, k\}$ is $\s$-invariant and there exists $i \in \{2, \ldots, k\}$ such that $\s(i)=1$. Then $v_{i}=v_{1}-1$ and thus $a_{l}+\cdots +a_{il-1} =  1$, but this is absurd since $a_{j} = 0$ for all $j \in \{2, \ldots, kl+r-1\}$ by the minimality of $s$.

Now assume $l$ does divide $s+1$, say $s = kl-1$ for some $k \in \{2, \ldots, t-1\}$. Then $\mu = \l-\a_{l} - \cdots - \a_{kl} \in \L(V)$
affords the highest weight of a $KH^0$-composition factor, but $\mu|_{H^0} = \l|_{H^0}+\o_{1,l-1}-\o_{1,l}+\o_{k+1,1}$ and this contradicts Lemma \ref{l:eh}. We have now reduced to the case
$\l = a_1\l_1+a_{n-1}\l_{n-1}+a_n\l_n$.

If $a_1 \neq 0$ then $\mu = \l - \a_1 - \cdots - \a_l \in \L(V)$
affords the highest weight of a $KH^0$-composition factor and $\mu|_{H^0} = \l|_{H^0}-\o_{1,1}+\o_{2,1}$. Then $\s(1)=2$ is the only possibility, where $\s$ is the associated permutation, and thus $a_1=1$.
Similarly, if $a_{n-1}\neq 0$ then
$\mu = \l - \a_{l} - \cdots - \a_{n-1} \in \L(V)$
affords the highest weight of a $KH^0$-composition factor and
$$\mu|_{H^0} = \l|_{H^0}+\o_{1,l-1}-\o_{1,l}-\o_{t,l-1}+\o_{t,l}.$$
Here $\s(t)=1$ is the only possibility, so $a_{n-1}=1$ and $a_{1}=0$.

In view of the hypothesis of the lemma, it remains to deal with the case $\l=a_1\l_1+a_n\l_n$ with $a_1 \le 1$. First assume $a_1=0$, so $a_n \neq 0$ and $\l-\a_{n-1}-\a_{n} = \l+\l_{n-2}-\l_{n}$ is a dominant weight. If $p \neq 2$ then Lemma \ref{l:pr} implies that 
$\l-\a_{n-1}-\a_{n} \in \L(V)$ and we deduce that 
$\mu = \l - \a_{l}- \cdots - \a_{n}\in \L(V)$
affords the highest weight of a $KH^0$-composition factor and
\begin{equation}\label{e:eqa}
\mu|_{H^0} = \l|_{H^0}+\o_{1,l-1}-\o_{1,l}+\o_{t,l-1}-\o_{t,l}.
\end{equation}
Clearly, this weight is not conjugate to $\l|_{H^0}$, which is a 
contradiction. Therefore we may assume $p=2$ and $\l=\l_n$. Here $V|_{H^0}$ is irreducible by \cite[Theorem 1]{Seitz2}, and we record this example in Table \ref{t:c2}.

Finally, suppose $\l=\l_1+a_n\l_n$. If $a_n=0$ then $\l=\l_1$ and $V$ is the natural module for $G$, so let us assume $a_n \neq 0$, in which case $p \neq 2$ since $V$ is tensor-indecomposable (see \cite[1.6]{Seitz2}). As above, $\mu = \l - \a_{l}- \cdots - \a_{n} \in \L(V)$
affords the highest weight of a $KH^0$-composition factor, and \eqref{e:eqa} holds.
Let $\s \in S_t$ denote the associated permutation and fix $i \ge 2$ such that $\s(i)=1$. If $l \ge 3$ then we get $a_{il-1} = 1$, which is a contradiction. Similarly, if $l=2$ then $a_{2i-1} = 2$, which is equally absurd. This completes the proof of the lemma.
\end{proof}

\begin{lem}\label{l:c2cn}
Proposition \ref{TH:C2} holds in case (iii) of Table \ref{t:c2s}.
\end{lem}

\begin{proof}
In view of the previous lemma, we may assume $\l=\l_{n-1}+a\l_n$ for some $a \ge 0$. We need to prove that $V|_{H}$ is irreducible if and only if one of the following holds:
\begin{itemize}\addtolength{\itemsep}{0.3\baselineskip}
\item[(a)] $\l=\l_1$;
\item[(b)] $t=2$, $a <p$, $(l,a) \neq (1,0)$ and $2a+3 \equiv 0 \imod{p}$.
\end{itemize}
This is clear if $\l=\l_1$, so for the remainder let us assume $\l \neq \l_1$. We continue with the notation introduced in the proof of the previous lemma. Note that $H$ is a maximal rank subgroup of $G$; in particular, $\l \in \L(V)$ is the unique $T$-weight restricting to $\l|_{H^0}$.

Suppose $V|_{H}$ is irreducible. There are exactly
$t$ distinct composition factors of $V|_{H^0}$, with highest weights
of the form
$$(a+1)\omega_{1,l}+(a+1)\o_{2,l}+\cdots+(a+1)\omega_{t-1,l}+
\omega_{t,l-1}+a\omega_{t,l}$$
(where $\o_{t,l-1}=0$ if $l=1$), and the distinct permutations of this weight under the natural
action of $S_t$. The highest weights of these composition factors are afforded
by $\l$ and $\l-\a_{il}-\a_{il+1}-\cdots-\a_{n-1}$, where $1\leq i\leq t-1$. 

First assume $(l,a) = (1,0)$. If $t=2$ then $n=2$ and $\l=\l_1$, so we may assume $t \ge 3$ (and thus $n \ge 3$). Set $\nu = \l - \a_{n-2}-2\a_{n-1}-\a_n$. Then $m_V(\nu)$ coincides with the multiplicity of the zero weight of $L_{C_3}(\l_2)$, so $m_V(\nu) = 2-\delta_{3,p}$ since $\dim L_{C_3}(\l_2) = 14-\delta_{3,p}$ (see Table \ref{t:dims}). The composition factors of $V|_{H^0}$ are afforded by the weights 
$$\l, \; \l - \a_1 - \cdots - \a_{n-1}, \; \l - \a_2 - \cdots - \a_{n-1}, \; \ldots, \l - \a_{n-1},$$ 
but by considering the root system for $H^0$ given in the proof of Lemma \ref{c2:p3}, we quickly deduce that $\nu$ does not occur in any of the $KH^0$-composition factors afforded by these weights. This is a contradiction, and we have eliminated the case $(l,a)=(1,0)$ (with $t \ge 3$).

In the remaining cases we claim that $2a+3\equiv0 \imod{p}$. First assume $a \neq 0$.
By Lemma \ref{l:118}, the weight $\nu=\l-\a_{(t-1)l}-\a_{(t-1)l+1}-
\cdots-\a_n$ occurs with multiplicity $2$ in $V$, unless $2a+3\equiv 0 \imod{p}$,
in which case $m_{V}(\nu)=1$. Moreover, this weight can only
occur in the $KH^0$-composition factor afforded by $\mu=\l-\a_{(t-1)l}-\a_{(t-1)l+1}-\cdots-\a_{n-1}$.
Since
$$\mu|_{H^0}=
(a+1)\omega_{1,l}+\cdots+(a+1)\omega_{t-2,l}+ \omega_{t-1,l-1}+
a\omega_{t-1,l}+(a+1)\omega_{t,l}$$
and $\nu=\mu-\a_n$, it follows that $m_{V}(\nu)=1$ and thus 
$2a+3\equiv0 \imod{p}$, as claimed.

Now assume $a=0$. We have already considered the case $l=1$, so let us assume $l \ge 2$ (so $n \ge 4$). As before, let $\nu = \l - \a_{n-2}-2\a_{n-1}-\a_n$ and let $\nu'$ be the conjugate weight $\nu' = \l - \a_{n-l} - \cdots - \a_{n-2} - 2\a_{n-1} - \a_n$, so $m_{V}(\nu') = 2-\delta_{3,p}$. By inspecting the root system of $H^0$ given in the proof of Lemma \ref{c2:p3}, we see that $\nu'$ can only be in the $KH^0$-composition factor afforded by $\eta = \l - \a_{n-l} - \cdots - \a_{n-1}$. Now $\nu' = \eta - \a_{n-1} - \a_n$ and
$$\eta|_{H^0} = \o_{1,l}+ \cdots + \o_{t-2,l}+(\o_{t-1,1}+\o_{t-1,l-1}) + \o_{t,l},$$
so $\nu'$ has multiplicity $1$ in the $KH^0$-composition factor afforded by $\eta$. We conclude that  $m_V(\nu')=1$ and thus $p=3$, as required.

Next we show that $t=2$ if $V|_{H}$ is irreducible. This follows from our earlier analysis when $(l,a) = (1,0)$, so let us assume otherwise. Then $2a+3\equiv0 \imod{p}$, so $p \neq 2$ and all weights occurring in the Weyl module $W_G(\l)$ also occur as weights of $V$, by Lemma \ref{l:pr}. Suppose $t>2$ (and so $n\geq 3$). Then the weight
$$\mu=\l-\a_{n-2l}-2\a_{n-2l+1}-\cdots-2\a_{n-1}-\a_n$$
occurs with
non-zero multiplicity in $V$. As the coefficient of $\a_{n-2l}$ in $\l-\mu$ is $1$, while the 
corresponding coefficient of $\a_{n-2l-1}$ is $0$, we see that $\mu$ must occur in the
$KH^0$-composition factor with highest weight afforded by 
$\nu=\l-\a_{n-2l}-\a_{n-2l+1}-\a_{n-2l+2}-\cdots-\a_{n-1}$. But for any weight $\mu$ in this summand we observe that the coefficient of $\a_{n-l}$ in $\nu-\mu$ is even, which is incompatible with the form of $\nu$. Hence $t=2$ as claimed.

Clearly, if $(l,a,t) = (1,0,2)$ then $\l=\l_1$ and $V|_{H}$ is irreducible, so to complete the proof of the lemma it remains to show that $V|_{H}$ is irreducible when $t=2$ and $2a+3\equiv0 \imod{p}$ (where $l \ge 2$ if $a=0$).
Here $p\ne 2$ and $n=2l$. Also note that $p=3$ if $a=0$. 
Set $J=H^0$ and write $J=X_1X_2$, where
$$X_1=\langle U_{\pm\a_1},U_{\pm\a_2},\ldots, U_{\pm\a_{l-1}},
U_{\pm \gamma}\rangle,\;\; 
X_2=\langle  U_{\pm\a_{l+1}}, \ldots, U_{\pm\a_{n}} \rangle$$
and $\gamma=2(\a_{l}+\cdots +\a_{n-1})+\a_n$. Set
$$\Pi(X_1)=\{\a_i,\gamma\mid 1\leq i\leq l-1\}, \;\; \Pi(X_2)=\{\a_i\mid l+1\leq i\leq n\}$$
and
$$\Pi(J)=\Pi(X_1)\cup \Pi(X_2), \;\; \Sigma(J)=\Z\Pi(J)\cap \Sigma(G), \;\; \Sigma^+(J)=\Z\Pi(J)\cap \Sigma^+(G).$$

Let $v^+ \in V$ be a maximal vector for the fixed Borel subgroup $B$ of $G$.
Then $v^+$ is also a maximal vector for the Borel subgroup $B_J:= B\cap J$
of $J$. Hence, $V$ has a $KJ$-composition factor with highest weight $\l|_{J}$.
Let $\mu=\l-\a_l-\a_{l+1}-\cdots-\a_{n-1}$ and note that $m_{V}(\mu)=1$ since $\mu$ is conjugate to $\l$.  We also note that if $m \in \mathbb{N}$ then $\mu+m\b\not\in\L(V)$ for all $\b\in\Pi(J)$, so if $0\ne w^+\in V$
has weight $\mu$ then $U_\b w^+=w^+$
for all $\b\in\Pi(J)$ (see \eqref{e:mt}).
Hence, $\mu$ affords the highest weight of a second
$KJ$-composition factor of $V$. Let $U\subseteq V$ be the $KJ$-submodule of $V$
generated by $v^+$ and $w^+$. Note that
$$\l|_{J}=(a+1)\omega_{1,l}+\omega_{2,l-1}+a\omega_{2,l},\;\; \mu|_{J}=\omega_{1,l-1}+a\omega_{1,l}+(a+1)\omega_{2,l}$$
and thus
$(\l-\mu)|_{J} = \frac{1}{2}(\gamma-\a_{n})$. In particular, $U$ is the
direct sum of $KJv^+$ and $KJw^+$.

Seeking a contradiction, suppose that $V|_{H}$ is reducible, so $J$ has more than two
composition factors in its action on $V$.
We first show that there exists a maximal vector
$w\in V$, with respect to $B_J$, such that $w\not\in\langle v^+\rangle$ and
$w\not\in\langle w^+\rangle$. Let us suppose that no such 
$w\in V$ exists. Let $w_0$ be the longest word
in the Weyl group of $G$ and note that $w_0$ is represented in the group $J$ (the longest
word is just the element $-1$). In particular,
$w_0v^+$ is a vector in $U$ of weight $-\l$, and $w_0w^+$
is
a vector in $U$ of weight $-\mu$. Now define
$f,g\in V^*$ by setting
$$f(w_0v^+)=g(w_0w^+)=1,\;\; f|_{V_\eta}=g|_{V_{\eta'}} = 0$$
for all weights $\eta\ne-\l$ and $\eta' \ne -\mu$.

Now let $\b\in\Pi(J)$ and $c\in K$. Then
$x_\b(c)f(w_0v^+)=f(x_\b(-c)w_0v^+)=1$, since $x_\b(-c)w_0v^+\in w_0v^++\sum_{m \in \mathbb{N}}V_{-\l+m\b}$, and for $\eta\ne-\l$ and 
$v_\eta\in V_\eta$,
$x_\b(c)f(v_\eta)=f(x_\b(-c)v_\eta)$,
which is equal to the coefficient of $w_0v^+$ in $x_\b(-c)v_\eta$. But as
$-\l=w_0\l$
is the lowest weight of $V$, $x_\b(-c)v_\eta$ has a trivial projection into the
 weight space $V_{-\l}$ and $f(x_\b(-c)v_\eta)=0$. So $x_\b(c)f=f$ and $f\in V^*$
is a maximal vector for $J$ (of weight $\l$) not belonging to ${\rm Ann}_{V^*}(U)$.

Similarly, $x_\b(c)g(w_0w^+)=g(x_\b(-c)w_0w^+)=1$, and for $\eta\ne-\mu$
and $v_\eta\in V_\eta$,
$x_\b(c)g(v_\eta)=g(x_\b(-c)v_\eta)$,
which is equal to the coefficient of $w_0w^+$ in $x_\b(-c)v_\eta$.
If this coefficient is non-zero then $\eta+m\b = -\mu$ for some $m \in \mathbb{N}$, so
$\eta = -\mu-m\b$ and thus $w_0(-\mu-m\b)=\mu+m\b$ is a weight of $V$.
But as remarked above, this is not the
case for $\b\in\Pi(J)$. So $x_\b(c)g=g$ and $g\in V^*$ is a maximal vector for $J$ 
(of weight $\mu$)
not lying in ${\rm Ann}_{V^*}(U)$. As ${\rm Ann}_{V^*}(U)$ is a $KJ$-submodule
of $V^*$, there exists an irreducible $KJ$-submodule and so a maximal
vector. Since $V^*\cong V$, $\l$ and $\mu$ each occur with multiplicity $1$
in $V^*$. Moreover, we have just seen that the weights $\l$ and $\mu$ do not
occur
in ${\rm Ann}_{V^*}(U)$. Hence, $V^*$ has a $KJ$-maximal vector of
 weight
different from
$\l$ and $\mu$. But $V^*\cong V$, contradicting our assumption that no such maximal vector exists in $V$. 

Now, let $w\in V$ be a maximal vector with respect to $B_J$ such that
$w\not\in\langle v^+\rangle$ and $w\not\in\langle w^+\rangle$. 
As $J$ is a maximal
rank subgroup of $G$, $w$ is a weight vector of the $KG$-module $V$; choose 
such a $w$
of weight $\nu=\l-\sum_{i=1}^{n}c_i\a_i$, with $\sum_{i=1}^{n}c_i$ minimal.
(We call  $\sum_{i=1}^{n}c_i$ the {\em level} of $w$.)
As $w\not\in\langle v^+\rangle$, $w$ is not a maximal vector with respect to
$B$. Since $\l$ is a $p$-restricted weight, $V$ is an irreducible
$K\mathcal{L}(G)$-module
and so $v^+$ is the unique maximal vector with respect to the Lie algebra $\mathcal{L}(B)$.
For $\alpha\in\Sigma^+(G)$, let $e_\alpha$ span the root space 
$\mathcal{L}(U_\alpha)$.
As $e_{\a_i}w=0$ for all $i \neq l$,
we have  $e_{\a_l}w\ne0$. 

Choose $\b\in\Sigma^+(G)$ of maximal height such that
$e_\b w\ne0$; write $\b=\sum_{i=1}^{n} d_i\a_i$.
Note that $d_l=1$. Indeed, if $d_l=0$ then $\b\in\Sigma^+(J)$ and $e_\b w=0$, while 
if $d_l=2$, then
$e_\b\in\langle e_{\a_i},
e_\gamma\mid 1\leq i\leq l-1\rangle\subseteq\mathcal{L}(B_J)$ and thus $e_\b w=0$.
Therefore $d_l=1$ and $d_i\in\{0,1\}$ for all $i<l$.

For now let us assume $l=1$, so $n=2$, $a \neq 0$ and 
the two composition factors of $V|_J$ are afforded by $\l|_J$ and $(\l-\alpha_1)|_J$. 
The above remarks indicate that $\beta = \alpha_1$
or $\beta = \alpha_1+\alpha_2$. First assume $\beta = \a_1$, so $v  := e_{\alpha_1}w\ne 0$. If $v\in\langle w^+\rangle$ then $w$ is of weight $\l-2\alpha_1$, which is absurd since $\l-2\alpha_1 \not\in \L(V)$. Similarly, if $v\in \langle v^+\rangle$ then 
$w$ is of weight $\l-\alpha_1$, but $m_V(\l-\alpha_1) = 1$ and thus $w\in\langle w^+\rangle$, contradicting our choice of $w$. Therefore $v\not\in \langle v^+\rangle$ and $v\not\in\langle w^+\rangle$. In particular, since the level of $v$ is smaller than that of $w$, it follows that $v$ is not a maximal vector with respect to $B_J$.
Since $e_{2\alpha_1+\alpha_2}v = 0$ we must
have 
$$e_{\alpha_2}v = e_{\alpha_2}e_{\alpha_1}w = ce_{\alpha_1+\alpha_2}w\ne 0$$ 
for some $c\in K$, but this contradicts our choice of $\beta$. 

Now assume $\beta=\alpha_1+\alpha_2$, so  
$v': = e_{\alpha_1+\alpha_2}w\ne 0$.
Arguing as above we deduce that $v'$ does not lie in $\langle v^+\rangle$ nor in 
$\langle w^+\rangle$. Indeed, if $v' \in \langle v^+\rangle$ then $w$ would have weight $\l-\alpha_1-\alpha_2$, which has
multiplicity $1$ in $V$ and occurs with non-zero multiplicity in the $KJ$-composition
factor afforded by $\l-\alpha_1$. Similarly, if $v' \in \langle w^+\rangle$ then $w$ would be of weight $\l-2\alpha_1-\alpha_2$, which also has multiplicity $1$ in $V$ since it is conjugate to $\l-\alpha_1-\alpha_2$, and it occurs
with non-zero multiplicity in the $KJ$-composition factor afforded by $\l$. In both cases, this contradicts the fact that $w$ is a maximal vector with respect to $B_J$.
As the level of $v'$ is smaller than that of $w$, our choice of $w$ implies that $v'$ is not  a maximal vector with respect to $B_J$. However, we have  
$$e_{\alpha_2}v' = e_{2\alpha_1+\alpha_2}v' = 0$$ 
and this yields the desired contradiction.

For the remainder of the proof, we may assume that $l\geq 2$ and so $n\geq 4$.
Suppose $e_\b w\in\langle v^+\rangle$. Then $w$ is
of weight $\l-\b$, which we observe is conjugate to
$\l-\a_l-d_{l+1}\a_{l+1}-\cdots-d_{n}\a_{n}$. If $d_{n}=0$, then $d_i\in\{0,1\}$ for all $i$ and $\l-\b$ is conjugate to $\l-\a_{n-1}$,
and hence occurs with multiplicity $1$ in $V$. As $\l-\b=\mu-\a_i-\a_{i+1}-\cdots-\a_{l-1}$, 
this weight occurs
with multiplicity $1$
in the $KJ$-composition factor with highest weight $\mu|_J$, hence $w$ cannot
be of weight $\l-\b$ and thus $d_{n}=1$. One checks that in this case, $\l-\b$
is
conjugate to $\l-\a_{n-1}-\a_{n}$ or $\l-\a_{n-2}-2\a_{n-1}-\a_{n}$.
By Lemmas \ref{l:118}(ii) and \ref{c2:l3}, both of these weights
have multiplicity $1$ in $V$.  But $\l-\b$ occurs with multiplicity $1$ in the
$KJ$-composition
factor afforded by $\mu$. So again, we see that $w$ cannot be a
maximal vector as chosen. Hence $e_\b w\not\in\langle v^+\rangle$.

Similarly, let us assume $e_\b w\in\langle w^+\rangle$, so $w$ is of weight $\mu-\b$.
Recall that $\mu=\l-\a_l-\a_{l+1}-\cdots-\a_{n-1}$. Note that $d_{n}\ne0$,
as otherwise
$\mu-\b=\l-\a_l-\cdots-\a_{n-1}-\b$ is not a weight of $V$. (One can easily
see this by restricting to the $A_{n-1}$ Levi factor of $G$.)
Thus $d_{n}=1$ and $d_i\ne0$ for all $l\leq i\leq n$. Moreover, if $d_i=2$, then
$d_{i+1}=2$ for all $l+1\leq i\leq n-2$. One checks that for all
such $\b$, the weight $\mu-\b$ is conjugate to $\l-\a_{n-2}-2\a_{n-1}-\a_{n}$.
By Lemma \ref{c2:l3}, this weight occurs with multiplicity $1$ in $V$.
The weight $\mu-\b$ cannot occur in the $KJ$-composition factor
of highest weight $\mu$, as the coefficient $d_l$ of $\a_l$ in $\b$ is $1$.
On the other hand, $\mu-\b$ does occur in the $KJ$-composition factor of highest weight
$\l$ since
$$\mu-\b=\l-\a_i-\cdots-\a_{l-1}-2\a_l-(d_{l+1}+1)\a_{l+1}-
\cdots-(d_{n-1}+1)\a_{n-1}-\a_{n}$$ 
for some $i\leq l$. Therefore
$$\mu-\b = \l-(\a_i+\cdots+\a_{l-1}+
\gamma)-(1-\delta_{n,k})(\a_k+\cdots+\a_{n-1})$$ 
for some
$k$ with $l+1\leq k\leq n$, and this weight occurs with multiplicity $1$ in the
$KJ$-composition
factor afforded by $\l$. This contradicts the choice of $w$ as a maximal vector with respect to $B_J$. 

We have now established that $e_\b w\not\in\langle v^+\rangle$ and
$e_\b w\not\in\langle w^+\rangle$, where
$$\b=\a_i+\cdots+\a_l+d_{l+1}\a_{l+1}+\cdots+d_{n-1}\a_{n-1}+d_{n}\a_{n}.$$
Note that $\gamma+\b\not\in\Sigma(G)$, and so
$e_\gamma e_\b w = e_\b e_\gamma w = 0$. For $i\ne l$, if
$\a_i+\b\not\in\Sigma(G)$, then as with $\gamma$, we have $e_{\a_i}e_\b w = 0$.
Finally, for $i\ne l$, if $\a_i+\b \in\Sigma(G)$ then $e_{\a_i}e_\b w = e_\b e_{\a_i}w + ce_{\a_i+\b} w = ce_{\a_i+\b} w$ for some $c \in K$. Our choice of $\b$ (of maximal
height such that $e_\b w\ne 0$), implies that $ce_{\a_i+\b} w=0$.

But we have now shown that $e_r e_\b w=0$ for all $r\in\Pi(J)$, that is,
$e_\b w$ is a maximal vector with respect to $B_J$, not lying in
$\langle v^+\rangle$,
nor in $\langle w^+\rangle$, contradicting
our choice of $w$ of minimal level. In view of this final contradiction, we conclude that $V|_{H}$ is irreducible.
\end{proof}

\begin{remk}\label{r:sz2}
We can use \cite{SZ} to give an alternative proof of the irreducibility of $V|_{H}$ in the previous lemma. In \cite{SZ}, it is shown that 
$$\dim V_{C_m}(\lambda_{m-1}+\frac{1}{2}(p-3)\lambda_m) = \frac{1}{2}(p^m-1)$$ 
and 
$$\dim V_{C_m}(\frac{1}{2}(p-1)\lambda_m)= \frac{1}{2}(p^m+1).$$ 
Let $G=C_n$ with $n=2l$, let $H$ be a $\C_2$-subgroup of type $C_lC_l.2$ and set 
$V=V_{G}(\lambda_{n-1}+\frac{1}{2}(p-3)\lambda_n)$. If we restrict $V$ to $H^0=C_lC_l$, we have composition factors 
$$V_{C_l}(\lambda_{l-1}+\frac{1}{2}(p-3)\lambda_l)
\otimes V_{C_l}(\frac{1}{2}(p-1)\lambda_l)$$
and
$$V_{C_l}(\frac{1}{2}(p-1)\lambda_l)
\otimes V_{C_l}(\lambda_{l-1}+\frac{1}{2}(p-3)\lambda_l),$$ 
as described in the proof of Lemma \ref{l:c2cn}. From the above dimension formulae, we deduce that 
\begin{align*}
\dim V = & \dim (V_{C_l}(\lambda_{l-1}+\frac{1}{2}(p-3)\lambda_l)
\otimes V_{C_l}(\frac{1}{2}(p-1)\lambda_l)) \\
& + \dim (V_{C_l}(\frac{1}{2}(p-1)\lambda_l)
\otimes V_{C_l}(\lambda_{l-1}+\frac{1}{2}(p-3)\lambda_l))
\end{align*}
and thus $H$ acts irreducibly on $V$.
\end{remk}

\begin{lem}\label{c2:p4}
Proposition \ref{TH:C2} holds in case (iv) of Table \ref{t:c2s}.
\end{lem}

\begin{proof}
Here $G=D_n$, $t$ is even, $H=(2^{t-1}\times B_l^t).S_t$, $n = lt+t/2$ and $p \neq 2$. 
To obtain expressions for the root subgroups of $H^0$ we use induction and the description of the embedding $B_l^2 < D_{2l+1}$ given in \cite[Claim 8]{Test2}.  

We begin with the case $l=1$. Here $n=3t/2$, $H^0 = X_1 \cdots X_t$ and we may assume $t \ge 4$. We may choose simple roots $\{\b_1,\ldots, \b_t\}$ for $H^0$ such that
$$x_{\pm\b_{t-1}}(c) = x_{\pm(\a_{n-2}+\a_{n-1})}(c)x_{\pm(\a_{n-2}+\a_n)}(c), \;\;
x_{\pm\b_{t}}(c)=x_{\pm\a_{n-1}}(c)x_{\pm\a_n}(c)$$
and $x_{\pm\b_{i}}(c) = x_{\pm\gamma_{i}}(c)x_{\pm\delta_{i}}(c)$ for all $i<t-1$, where $c \in K$,
$$\gamma_{2k-1} = \a_{3k-2}+\a_{3k-1}+\a_{n-1}+\a_{n}+2\sum_{j=3k}^{n-2}\a_j,\;\; \delta_{2k-1} = \a_{3k-2}+\a_{3k-1}$$
and
$$\gamma_{2k} = \a_{3k-1}+\a_{n-1}+\a_n+2\sum_{j=3k}^{n-2}\a_j,\;\; \delta_{2k} = \a_{3k-1}.$$
Let $\o_{i}$ be the fundamental dominant weight of $X_{i}$ corresponding to the simple root $\b_i$. From the above description of the root subgroups of $H^0$, it follows that if $\mu = \sum_{i=1}^{n}{b_i\l_i}$ is a weight for $T$ then $\mu|_{X_{i}} = v_{i}(\mu)\o_i$, where $v_{t}(\mu) = b_{n-1}+b_n$ and
$$v_{2k-1}(\mu) = b_{n-1}+b_{n}+2\sum_{j=3k-2}^{n-2}b_j,\;\; v_{2k}(\mu) = b_{n-1}+b_{n}+ 2\sum_{j=3k-1}^{n-2}b_j.$$
Let $\l=\sum_{i=1}^{n}a_i\l_i$ be the highest weight of $V$, and set $v_{i}=v_{i}(\l)$ for all $i$.

Suppose $a_{3k} \neq 0$, where $3k<n$. Then $\mu = \l-\a_{3k} \in \L(V)$ affords the highest weight of a $KH^0$-composition factor, but this contradicts Lemma \ref{l:eh} since $\mu|_{H^0} = \l|_{H^0} +2\o_{2k+1}$. Similarly, if
$a_{3k+1} \neq 0$ and $k>0$ then
$\mu = \l - \a_{3k} - \a_{3k+1} \in \L(V)$ and
$\mu|_{H^0} = \l|_{H^0} +2\o_{2k+2}$ is the highest weight of a $KH^0$-composition factor. Again, this contradicts Lemma \ref{l:eh}. Also, if  $a_{3k-1} \neq 0$ with $k \ge 2$ and $3k-1 \le n-4$ then
$$\mu = \l - \a_{3k-3}-\a_{3k-2}-\a_{3k-1}-\a_{3k}-\a_{3k+1} \in \L(V)$$
affords the highest weight of a $KH^0$-composition factor, but $\mu|_{H^0} = \l|_{H^0}+2\o_{2k+2}$ and this contradicts Lemma \ref{l:eh}. We have now reduced to the case
$$\l = a_1\l_1+a_2\l_2+a_{n-1}\l_{n-1}+a_n\l_n.$$

If $a_1+a_2 \neq 0$ then $\mu = \l - \a_1 - \a_2 - \a_3 \in \L(V)$ affords the highest weight of a $KH^0$-composition factor and
$\mu|_{H^0} = \l|_{H^0}-2\o_1+2\o_3$. If $\s \in S_t$ is the associated permutation then $\s(1)=3$ is the only possibility, whence $a_1+a_2=1$. 

Next we claim that $a_{n-1}a_{n}=0$. Seeking a contradiction, suppose that $a_{n-1}a_n \neq 0$. Set $w = af_{\alpha_n}v^++bf_{\alpha_{n-1}}v^+ \in V_{\l-\a_n} \oplus V_{\l-\a_{n-1}}$, where $a,b \in K$ are scalars and $v^+ \in V$ is a maximal vector for the fixed Borel subgroup of $G$. Then $U_{\b_i}w = w$ for all $i<t$,
and 
$$x_{\b_t}(c)w = w+c(e_{\a_{n-1}}af_{\a_{n-1}}v^+ + e_{\a_n}bf_{\a_n}v^+) = w+c(aa_{n-1}v^+ + ba_nv^+).$$ 
So choosing $b = -aa_{n-1}/a_n$, we see that $w$ is a maximal vector with respect to the fixed Borel subgroup $B_{H^0}$ of $H^0$ defined in terms of the above root subgroups for $H^0$. Now $(\l-\a_n)|_{H^0} = (\l-\a_{n-1})|_{H^0} = \l|_{H^0}-\b_t$, so it follows that $V|_{H^0}$ has a composition factor with highest weight $\l|_{H^0}-\b_t$. But this contradicts Lemma \ref{l:eh}, hence $a_{n-1}a_{n}=0$ as claimed.

Next suppose $a_{n-1}>1$, so $a_n=0$. Then 
$\lambda-\a_{n-3}-2\a_{n-2}-2\a_{n-1},\l-\a_{n-3}-2\a_{n-2}-\a_{n-1}-\a_n\in\Lambda(V)$, and both weights restrict to $\l|_{H^0}-\b_{t-1}$.
Let $x,y$ be non-zero vectors in the respective $T$-weight spaces. Set  
$\mu = \lambda-\a_{n-3}-\a_{n-2}-\a_{n-1}\in\Lambda(V)$ and note that $m_V(\mu)=1$, say  $V_{\mu} = \la v \ra$. Now $U_{\b_i}x = x$ and $U_{\b_i}y = y$ for all $i \neq t-1$, and 
$x_{\b_{t-1}}(c)x = x_{\a_{n-2}+\a_{n-1}}(c)x$, $x_{\b_{t-1}}(c)y = x_{\a_{n-2}+\a_{n}}(c)y$. 
In particular, since  $x$ and $y$ 
cannot afford maximal vectors of $KH^0$-composition factors of $V$ (by Lemma \ref{l:eh}), it follows that   
$x_{\a_{n-2}+\a_{n-1}}(c)x\neq x$ and
$x_{\a_{n-2}+\a_{n}}(c)y\neq y$. Hence, there exists $a,b\in K^*$ with $x_{\a_{n-2}+\a_{n-1}}(c)x = x+acv$ and 
$x_{\a_{n-2}+\a_{n}}(c)y=y+bcv$. Then $bx-ay$ is a maximal vector with respect to the Borel subgroup $B_{H^0}$, but this contradicts Lemma \ref{l:eh} because $h(\l|_{H^0}) \neq h(\l|_{H^0}-\b_{t-1})$. We conclude that $a_{n-1}\leq 1$. An entirely similar argument shows that $a_n\leq 1$. Therefore  $a_{n-1}+a_n\leq 1$.

If $\l=\l_1+\l_n$ then Lemma \ref{l:c2_special_2} implies that $V|_{H^0}$ has a composition factor with highest weight $\sum_{i=1}^t\o_i$ (recall that $t \ge 4$). But this weight is not conjugate to $\l|_{H^0}$, so we can eliminate the case $\l=\l_1+\l_n$. By symmetry, the same argument also rules out $\l=\l_1+\l_{n-1}$. 

Next suppose $\l = \l_2+\l_{n-1}$. Consider the semisimple subgroup
 $J = \langle U_{\pm\b_i} \mid 2 \leq i\leq t\rangle\leqs H^0$.
On the one hand, by restricting each composition factor of $V|_{H^0}$ to this subgroup, we see that each composition factor of $V|_J$ has highest weight of the form 
\begin{equation}\label{e:vJ}
3\omega_{k}+\sum_{i=2,i \ne k}^t\omega_{i} \;\; \mbox{ or } \;\; 3\omega_r + 3 \omega_s + \sum_{i=2, i \ne r,s}^{t}\omega_i
\end{equation}
where $2\leq k,r,s\leq t$, $r \neq s$. On the other hand, $J\leqs L = \langle U_{\pm\alpha_i} \mid  2\leq i\leq n\rangle \cong D_{n-1}$, the derived subgroup of a Levi
factor of a parabolic subgroup of $G$. Therefore \cite[Proposition 2.11]{Jantzen} implies that $V|_{L}$ has a composition factor afforded by $\lambda$, that is,
a composition factor of highest weight $(\l_2+\l_{n-1})|_L$. Next observe that $J\leqs Y \leqs L$, where $Y=B_{n-2}$ is the stabilizer  in $L$ 
of a non-degenerate $1$-space in the natural module for $L$. To simplify the notation, let us 
write 
$\{\eta_1,\dots,\eta_{n-2}\}$ for the base of the root system of $Y$, and $\{\rho_1,\dots,\rho_{n-2}\}$ 
for the associated fundamental dominant weights.
Then the $KL$-composition factor afforded by $\lambda$ restricted to $Y$ has a composition factor with highest weight 
$\rho_1+\rho_{n-2}$.
In addition, since $J$ is the connected component of a $\C_2$-subgroup $(2^{t-2} \times B_1^{t-1}).S_{t-1}$ of $Y$, Lemma \ref{l:c2_special_2} implies that $V|_J$ has a composition factor with highest weight $\sum_{i=2}^t\omega_{i}$. But this is incompatible with the form of the highest weights of the composition factors of $V|_{J}$ given in \eqref{e:vJ}. 

An entirely similar argument also eliminates the case $\l=\l_2+\l_{n}$. We have now reduced to the cases $\l = \l_1, \l_2,\l_{n-1}$ and $\l_n$.

If $\l=\l_1$ then $V$ is the natural $KG$-module and $V|_{H}$ is irreducible. Next suppose $\l=\l_2$, so $\dim V = n(2n-1)$ (see Proposition \ref{p:dims}, and recall that $p \neq 2$) and each $KH^0$-composition factor is $9$-dimensional. By applying the permutations in $S_t$ we see that there are at least $\binom{t}{2}$ distinct composition factors. There cannot be exactly $\binom{t}{2}$ since $n(2n-1)>9\binom{t}{2}$. Similarly, if there are more than $\binom{t}{2}$ factors then $\dim V \ge 18\binom{t}{2}$ since $H$ permutes the homogeneous summands of $V|_{H^0}$, but this cannot happen since $n(2n-1)< 18\binom{t}{2}$. This eliminates the case $\l=\l_2$. Finally, the cases $\l=\l_{n-1}$ and $\l_n$ provide irreducible examples by Lemma \ref{l:c2_special}.

\vs

For the remainder we will assume $l \ge 2$.
As before, we have $H^0 = X_1 \cdots X_t$, where $X_i \cong B_l$. Let $\{\b_{i,1}, \ldots, \b_{i,l}\}$ be a set of simple roots for $X_{i}$. Up to conjugacy, we may define the corresponding root subgroups of $X_{i}$ as follows:
$$x_{\pm \b_{i,j}}(c) = x_{\pm \a_{(i-1)l+j+\lfloor (i-1)/2\rfloor}}(c),\;\; x_{\pm \b_{i,l}}(c) = x_{\pm \gamma_i}(c)x_{\pm \delta_i}(c)$$
for all $c \in K$, $1 \le i \le t$ and $1 \le j \le l-1$, where
$$\gamma_{2k-1} = \sum_{j=0}^{l}\a_{(2k-1)l+k+j-1},\;\; \gamma_{2k} = \a_{k(2l+1)-1}$$
and
$$\delta_{2k-1} = \left\{\begin{array}{ll}
\gamma_{2k-1} + \displaystyle 2\sum_{j=k(2l+1)}^{n-2}\a_j+\a_{n-1}+\a_{n} & \mbox{if $2k < t$} \\
\displaystyle \sum_{j=0}^{l-1}\a_{(t-1)l+t/2+j-1}+\a_n & \mbox{if $2k=t$,}
\end{array}\right.$$
and
$$\delta_{2k} = \left\{\begin{array}{ll}
\a_{k(2l+1)-1}+\displaystyle 2\sum_{j=k(2l+1)}^{n-2}\a_j+\a_{n-1}+\a_{n} & \mbox{if $2k<t$} \\
\a_n & \mbox{if $2k=t$.}
\end{array}\right.$$

Consequently, if $\mu = \sum_{i}b_i\l_i$ is a weight for $T$ then 
$$\mu|_{X_i} = \sum_{j=1}^{l-1}b_{(i-1)l+j+\lfloor (i-1)/2\rfloor}\o_{i,j}+v_i(\mu)\o_{i,l}$$
for all $1 \le i \le t$, where
$$v_{2k-1}(\mu) = 2\sum_{j=(2k-1)l+k-1}^{n-2}b_j+b_{n-1}+b_n$$
and
$$v_{2k}(\mu) = 2\sum_{j=k(2l+1)-1}^{n-2}b_j+b_{n-1}+b_n$$
(so that $v_{t}(\mu) = b_{n-1}+b_n$). We set $v_i = v_i(\l)$ for all $i$. 

It will be useful to observe that if $t=2$ and $\a=\sum_{i=1}^nc_i\a_i$ with $c_i \in \mathbb{N}_0$ for all $i$, then
\begin{equation}\label{e:rho}
\mbox{$\a|_{H^0} = 0$ if and only if $\a=0$.}
\end{equation}
In particular, if $t=2$ then $\l$ is the unique weight of $\Lambda(V)$ that restricts to $\l|_{H^0}$, so if $V|_{H^0}$ is reducible then $V$ has exactly two non-isomorphic $KH^0$-composition factors (see Proposition \ref{p:niso}).

Suppose $a_{ml+\lfloor m/2 \rfloor} \neq 0$ for some $m \in \{1, \ldots, t-1\}$. If $m=2k$ is even then $\mu = \l - \a_{2kl+k} \in \L(V)$ affords the highest weight of a $KH^0$-composition factor, but this contradicts Lemma \ref{l:eh} since $\mu|_{H^0} = \l|_{H^0}+\o_{2k+1,1}$. Similarly, if $m=2k-1$ is odd then $\mu = \l - \a_{(2k-1)l+k-1} \in \L(V)$ affords the   highest weight of a $KH^0$-composition factor and
$$\mu|_{H^0} = \l|_{H^0}+\o_{2k-1,l-1}-2\o_{2k-1,l}+\o_{2k,1}.$$
Let $\s \in S_t$ be the associated permutation. If $\s(i) \ge 2k$ for some $i<2k$ then $v_{2k-1} = v_{2k}$, but this is not possible since we are assuming $a_{(2k-1)l+k-1} \neq 0$. Therefore $\{1, \ldots, 2k-1\}$ is $\s$-invariant, but this is ruled out by Lemma \ref{l:eh}.

For the time being, let us assume $l=2$, so $n=5t/2$. We deal first with the 
special case $l=t=2$, so $n=5$. As noted above (see \eqref{e:rho}), if $V|_{H^0}$ is reducible then $V$ has exactly two non-isomorphic $KH^0$-composition factors.

By the above argument we have $a_2=0$. Suppose $a_1 \neq 0$. Then 
$\l - \a_1-\a_2 \in \L(V)$
affords the highest weight of a $KH^0$-composition factor and we
quickly deduce that $(a_1,a_3) = (1,0)$ is the only possibility. 
Similarly, 
if $a_3 \neq 0$ then by considering $\l - \a_2-\a_3 \in \L(V)$ we 
see that $(a_1,a_3)=(0,1)$.

Now suppose $a_4+a_5\ne 0$. If $V|_{H^0}$ is irreducible, then 
\cite[Table 1]{Seitz2} indicates that the only examples are $\l=\l_4$ and 
$\l_5$, as recorded in Table \ref{t:c2}. Now assume $V|_{H^0}$ is reducible.
Since 
$$\l|_{H^0} = a_1\o_{1,1}+(2a_3+a_4+a_5)\o_{1,2} + a_3\o_{2,1}+(a_4+a_5)\o_{2,2}$$
the reducibility of $V|_{H^0}$ implies that $a_1+a_3\ne 0$.
Therefore, we may assume that
$\lambda=\lambda_1+a_4\lambda_4+a_5\lambda_5$ or 
$\lambda_3+a_4\lambda_4+a_5\lambda_5$, in which case the second 
$KH^0$-composition factor is afforded
by the restriction of $\mu = \lambda-\alpha_1-\alpha_2$ or 
$\mu = \lambda-\alpha_2-\alpha_3$, respectively.

If $a_4a_5\ne 0$ then $\mu_1=\l-\a_4$ and $\mu_2=\l-\a_5$ are weights that restrict to $\l|_{H^0}-\b_{2,2}$.
Since $\mu-\mu_1$ and $\mu-\mu_2$ do not restrict to a sum of roots in 
$\Sigma(H^0)$,
it follows that $\mu_1$ and $\mu_2$ must occur in the $KH^0$-composition
factor afforded by $\lambda$. But the weight 
$\lambda|_{H^0}-\beta_{2,2}$ occurs with 
multiplicity $1$ in this composition factor, which is a contradiction. We conclude that $a_4a_5=0$. 

Suppose $\lambda=\lambda_1+a_4\lambda_4$ with $a_4 \ge 1$. Let $\nu_1 = \lambda-\alpha_1-\alpha_2-\alpha_3-\alpha_4$, $\nu_2 = 
\lambda-\alpha_1-\alpha_2-\alpha_3-\alpha_5$ and observe that 
$m_{V}(\nu_1) = 3$ if $a_4+4$ is divisible by $p$, otherwise $m_{V}(\nu_1) = 4$ (see Lemmas \ref{l:s816} and \ref{l:sr}). We also note that $m_{V}(\nu_2)=1$. Now $\nu_1$ and $\nu_2$ both
restrict to $\lambda|_{H^0}-\beta_{1,1}-\beta_{1,2} = \mu|_{H^0}-\beta_{2,1}-\beta_{2,2}$, and by Lemma \ref{l:118} this weight has multiplicity $2$ (respectively, $1$) 
in the $KH^0$-composition factor afforded by $\lambda$ when $p$ does not divide $a_4+4$ (respectively, $p$ divides $a_4+4$), and the same multiplicity in the composition factor afforded by $\mu$. 
By comparing these multiplicities, we reach a contradiction via Lemma \ref{l:easy}. 
An entirely similar argument applies if $\lambda=\lambda_1+a_5\lambda_5$ and $a_5 \ge 1$.

Next assume $\lambda=\lambda_3+a_4\lambda_4$.  
Since the weights $\lambda-\alpha_3-\alpha_4$ and $\lambda-\alpha_3-\alpha_5$
both restrict to $\lambda|_{H^0}-\beta_{2,1}-\beta_{2,2}$, which has multiplicity at most $2$ in the $KH^0$-composition factor with highest weight $\l|_{H^0}$, and these weights do not 
occur in the composition factor afforded by $\mu = \l - \a_2-\a_3$, Lemma \ref{l:118} implies that $a_4=p-2$. Therefore, $\l|_{H^0} = p\o_{1,2}+\o_{2,1}+(p-2)\o_{2,2}$ and $\mu|_{H^0} = \o_{1,1}+(p-2)\o_{1,2}+p\o_{2,2}$. But $\rho=\lambda-\alpha_2-\alpha_3-\alpha_4 \in \L(V)$ and $\rho|_{H^0} = \l|_{H^0}-\b_{1,2} = \mu|_{H^0}-\b_{2,2}$, so $\rho|_{H^0}$ is not a weight of the $KH^0$-composition factor afforded by $\lambda$, nor in the one afforded by $\mu$. This contradicts Lemma \ref{l:easy}. The case $\lambda=\lambda_3+a_5\lambda_5$ is entirely similar. 
We have now reduced to the cases $\l=\l_1$ and $\l=\l_3$. If $\l=\l_1$ then $V$ is simply the natural module for $G$, and this case is recorded in Table \ref{t:c2}. On the other hand, if $\l=\l_3$ then $\dim V = 120$ (see \cite[Table A.42]{Lubeck}), but each $KH^0$-composition factor has dimension $50$, so $V|_{H}$ is reducible when $\l=\l_3$ since $2 \cdot 50 < 120$. This completes our analysis of the case $l=t=2$.

Next suppose $l=2$ and $t \ge 4$. Recall that we have already shown that $a_{2m+\lfloor m/2 \rfloor}=0$ for all $m \in \{1, \ldots, t-1\}$. In particular,
$a_{5k}=0$ for all $k<n/5$. Suppose $a_{5k+1} \neq 0$ for some $k \ge 1$. Then $\mu = \l - \a_{5k} - \a_{5k+1} \in \L(V)$ and $\mu|_{H^0} = \l|_{H^0}-\o_{2k+1,1}+2\o_{2k+1,2}$, but this contradicts Lemma \ref{l:eh} since $\mu|_{H^0}$ is the highest weight of a $KH^0$-composition factor. Similarly, if $a_{5k-1} \neq 0$ and $k<n/5$ then $\mu = \l - \a_{5k-1} - \a_{5k} \in \L(V)$ and $\mu|_{H^0} = \l|_{H^0}+\o_{2k,1}-2\o_{2k,2}+\o_{2k+1,1}$ is the highest weight of a $KH^0$-composition factor. If $\s \in S_t$ denotes the associated permutation then $\{1,\ldots, 2k\}$ is $\s$-invariant since $a_{5k-1} \neq 0$, but this contradicts Lemma \ref{l:eh}. Finally, if $a_{5k-2} \neq 0$ and $k<n/5$ then
$$\mu = \l - \a_{5k-3} - \a_{5k-2} - \a_{5k-1} - \a_{5k} \in \L(V)$$
affords the highest weight of a $KH^0$-composition factor and
$$\mu|_{H^0} = \l|_{H^0}+\o_{2k-1,1}-2\o_{2k-1,2}+\o_{2k+1,1}.$$
Here $\{1, \ldots, 2k-1\}$ has to be $\s$-invariant, but once again this contradicts Lemma \ref{l:eh}. 

For $l=2$ we have now reduced to the case
\begin{equation}\label{e:lam}
\l = a_1\l_1+a_{n-2}\l_{n-2}+a_{n-1}\l_{n-1}+a_n\l_n
\end{equation}
(with $t \ge 4$). We can also reduce to this configuration when $l >2$. To see this, let us assume $l>2$ and recall that $a_{ml+\lfloor m/2 \rfloor}=0$ for all $m \in \{1, \ldots , t-1\}$. 

For each odd integer $m \in \{1, \ldots, t-1\}$, say $m=2k-1$, suppose $s \in \{1, \ldots, l-2\}$ is minimal such that $a_{(2k-1)l+k-1+s} \neq 0$. Then
$$\mu = \l - \a_{(2k-1)l+k-1} - \a_{(2k-1)l+k} - \cdots - \a_{(2k-1)l+k-1+s} \in \L(V)$$
affords the highest weight of a $KH^0$-composition factor, but this contradicts Lemma \ref{l:eh} since
$$\mu|_{H^0} = \l|_{H^0}+\o_{2k-1,l-1}-2\o_{2k-1,l}-\o_{2k,s}+\o_{2k,s+1}.$$
Similarly, if $s \in \{1, \ldots, l-2\}$ is minimal such that $a_{(2k-1)l+k-1-s} \neq 0$ then
$$\mu = \l - \a_{(2k-1)l+k-1-s} - \a_{(2k-1)l+k-s} - \cdots - \a_{(2k-1)l+k-1} \in \L(V)$$
affords the highest weight of a $KH^0$-composition factor, but this is also ruled out by Lemma \ref{l:eh} since
$$\mu|_{H^0} = \l|_{H^0}+\o_{2k-1,l-s-1}-\o_{2k-1,l-s}+\o_{2k,1}.$$
Notice that if $t=2$ then we have reduced to the case \eqref{e:lam}, so let us assume $t \ge 4$.

For each even integer $m \in \{2, \ldots, t-2\}$, say $m=2k$, suppose $s \in \{1, \ldots, l-1\}$ is minimal such that $a_{2kl+k-s} \neq 0$ then
$$\mu = \l - \a_{2kl+k-s} - \a_{2kl+k-s+1} - \cdots - \a_{2kl+k} \in \L(V)$$
affords the highest weight of a $KH^0$-composition factor, with associated permutation $\s \in S_t$. If $s \ge 2$ then this contradicts Lemma \ref{l:eh} since 
$$\mu|_{H^0} = \l|_{H^0} +\o_{2k,l-s}-\o_{2k,l-s+1}+\o_{2k+1,1}.$$ 
On the other hand, if $s=1$ then 
$$\mu|_{H^0} = \l|_{H^0} +\o_{2k,l-1}-2\o_{2k,l}+\o_{2k+1,1}.$$ 
However, if $\s(i)>2k$ for some $i \le 2k$ then $v_{2k}=v_{2k+1}$, which contradicts the assumption $a_{2kl+k-1} \neq 0$. Therefore $\{1, \ldots, 2k\}$ is $\s$-invariant, but this contradicts Lemma \ref{l:eh}.

Finally, suppose $m \in \{2, \ldots, t-2\}$ is the largest even integer, $m=2k$ say, such that $a_{2kl+k +s} \neq 0$ for some $s \in \{1, \ldots, l-2\}$. Assume 
$s \in \{1, \ldots, l-2\}$ is minimal such that $a_{2kl+k +s} \neq 0$.
Then
$$\mu = \l - \a_{2kl+k} - \a_{2kl+k+1} - \cdots - \a_{2kl+k+s} \in \L(V)$$
affords the highest weight of a $KH^0$-composition factor and
$\mu|_{H^0} = \l|_{H^0} - \o_{2k+1,s} + \o_{2k+1,s+1}$. Let $\s \in S_t$ be the associated permutation. By the maximality of $m$, and our earlier analysis, we have $a_{(i-1)l+\lfloor (i-1)/2\rfloor+s}=0$ for all $i>2k+1$, whence $\s(2k+1)\le 2k$ and thus $v_{2k+1} = v_{2k}$. This is a contradiction since $a_{2kl+k+s} \neq 0$. 

This justifies the claim, and we have reduced to the configuration for $\l$ given in \eqref{e:lam}. (Also recall that we may assume $t \ge 4$ if $l=2$.)

If $a_1 \neq 0$ then $\mu = \l - \a_1 - \a_2 - \cdots - \a_l \in \L(V)$ affords the highest weight of a $KH^0$-composition factor and $\mu|_{H^0} = \l|_{H^0}-\o_{1,1}+\o_{2,1}$. We quickly deduce that $a_1=1$ is the only possibility. Next suppose $a_{n-1}a_n \neq 0$. Set $\mu_1 = \l - \a_{n-1}$, $\mu_2 = \l - \a_{n}$ and note that $m_V(\mu_i)=1$ and $\mu_i|_{H^0} = \l|_{H^0} - \b_{t,l}$, $i=1,2$. Fix non-zero vectors $x \in V_{\mu_1}$ and $y \in V_{\mu_2}$. Then $x_{\b_{i,j}}(c)x = x$ and  $x_{\b_{i,j}}(c)y = y$
for $(i,j) \neq (t,l)$, and we have $x_{\b_{t,l}}(c)x = x+ca v^+$ and $x_{\b_{t,l}}(c)y = y+cb v^+$, for some scalars $a,b \in K$.
If $ab=0$ then $x$ or $y$ is a maximal vector for the Borel subgroup $B_{H^0}$, but this contradicts Lemma \ref{l:eh} since $h(\mu_i|_{H^0}) \neq h(\l|_{H^0})$. Similarly, if $ab \neq 0$ then $w = bx-ay$ is non-zero ($x$ and $y$ are 
linearly independent), and $w$ is a maximal vector for $B_{H^0}$. Once again, we reach a contradiction via Lemma \ref{l:eh}. Hence $a_{n-1}a_n=0$.

For now, let us assume $t>2$. If $a_{n-2} \neq 0$ then $\mu = \l - \a_{n-2l-1} - \a_{n-2l} - \cdots - \a_{n-2} \in \L(V)$ affords the highest weight of a $KH^0$-composition factor, but this contradicts Lemma \ref{l:eh} since $\mu|_{H^0} = \l|_{H^0}-\o_{t,l-1}+2\o_{t,l}$. 

Next suppose $a_{n-1} \ge 2$. Then $\mu_1 = \l - \a_{n-2l-1} - \a_{n-2l} - \cdots - \a_{n-2} - 2\a_{n-1}$ and $\mu_2 =  \l - \a_{n-2l-1} - \a_{n-2l} - \cdots - \a_{n}$ are both weights of $V$ with $m_V(\mu_i)=1$ (it is easy to see this via Lemma \ref{l:sr}, working in the $A_{n-1}$ Levi factor). Fix non-zero vectors $x \in V_{\mu_1}$ and $y \in V_{\mu_2}$. Also set 
$\nu = \l - \a_{n-2l-1} - \a_{n-2l} - \cdots - \a_{n-1} \in \L(V)$ and note that $m_{V}(\nu)=1$, say $V_{\nu} = \la w \ra$. Now $x_{\b_{i,j}}(c)x = x$ and  $x_{\b_{i,j}}(c)y = y$
for $(i,j)\neq (t,l)$. By Lemma \ref{l:eh}, neither $\mu_1$ nor $\mu_2$ afford the highest weight of a $KH^0$-composition factor (both weights restrict to $\l|_{H^0}-\b_{t,l}$), so $x_{\b_{t,l}}(c)x\ne x$ and $x_{\b_{t,l}}(c)y\ne y$. Therefore $x_{\b_{t,l}}(c)x = x_{\a_{n-1}}(c) x = x+caw$ and $x_{\b_{t,l}}(c)y = x_{\a_{n}}(c) y = y+cbw$ for some $a,b\in K^*$. But this means that  $0 \neq bx-ay$ is the maximal vector of a $KH^0$-composition factor with highest weight $\l|_{H^0} - \b_{t,l}$, contradicting Lemma \ref{l:eh}. We conclude that $a_{n-1} \le 1$. An entirely similar argument shows that $a_n \le 1$, so for $t>2$ we have reduced to the following cases:
\begin{equation}\label{e:cases}
\l_1, \l_{n-1}, \l_n, \l_1+\l_{n-1}, \l_1+\l_n.
\end{equation}

Now assume $t=2$ and $l \ge 3$. Recall that $\l$ is the unique weight in $\L(V)$ that restricts to $\l|_{H^0}$ (see \eqref{e:rho}), so if $V|_{H^0}$ is reducible then $V|_{H^0}$ has exactly two non-isomorphic composition factors.

First observe that if $a_{n-2} \neq 0$ then $\mu = \l - \a_{l} - \cdots - \a_{n-2} \in \L(V)$ affords the highest weight of a $KH^0$-composition factor, and we have
$$\mu|_{H^0} = \l|_{H^0} + \o_{1,l-1}-2\o_{1,l} -\o_{2,l-1}+2\o_{2,l}$$
so $(a_{1},a_{n-2})=(0,1)$ is the only possibility. 
Suppose $a_{n-1}=a_n=0$. If $a_{n-2}=0$ then $\l=\l_1$ so let us assume $a_{n-2}\neq 0$. By the previous observation we have $\l = \l_{n-2}$. Now $\mu_1 = \l - \a_{n-2}-\a_{n-1}$ and $\mu_2 = \l - \a_{n-2}-\a_n$ are both weights of $V$ with multiplicity $1$ and which restrict to $\l|_{H^0} - \b_{2,l-1}  - \b_{2,l}$. Given the form of $\mu|_{H^0}$ above, it follows that 
$\mu_1$ and $\mu_2$ only occur in the $KH^0$-composition factor afforded by $\l$, but this is not possible (by Lemma \ref{l:easy}) since the $T_{H^0}$-weight $\l|_{H^0} - \b_{2,l-1}  - \b_{2,l}$ has multiplicity $1$ in this composition factor. Therefore, we may assume that $a_{n-1}+a_n \ge 1$. Without loss of generality, we will assume $a_{n-1} \neq 0$ (so $a_{n}=0$ since $a_{n-1}a_n=0$; an entirely similar argument applies if we assume $a_n \neq 0$). 

Suppose $a_{n-2} \neq 0$. By the previous analysis we have $\l = \l_{n-2}+a_{n-1}\l_{n-1}$. Define $\mu, \mu_1$ and $\mu_2$ as in the previous paragraph, so the $\mu_i$ are weights of $V$ that only occur in the $KH^0$-composition factor afforded by $\l$. Now both weights restrict to $\l|_{H^0} - \b_{2,l-1}  - \b_{2,l}$, which has multiplicity at most $2$ in the $KH^0$-composition factor afforded by $\l$ (see Lemma \ref{l:118}), so Lemma \ref{l:easy} implies that $m_{V}(\mu_1) = 1$ and thus $a_{n-1} = p-2$ by Lemma \ref{l:118}(i). Let $\nu = \lambda - \a_{l} - \cdots - \a_{n-1} \in \L(V)$.
 Now $\nu|_{H^0} = \l|_{H^0} - \b_{1,l} = \mu|_{H^0} - \b_{2,l}$, but neither $\l|_{H^0} - \b_{1,l}$ nor $\mu|_{H^0} -\b_{2,l}$ are weights in the respective $KH^0$-composition factors since the condition $a_{n-1}=p-2$ implies that $\l|_{H^0} = p\o_{1,l}+\o_{2,l-1}+(p-2)\o_{2,l}$ and $\mu|_{H^0} = \o_{1,l-1}+(p-2)\o_{1,l}+p\o_{2,l}$. This contradicts Lemma \ref{l:easy},  so we conclude that $a_{n-2}=0$.

Next suppose $a_{n-1} \ge 2$, so $\l = a_1\l_1+a_{n-1}\l_{n-1}$ with $a_1 \le 1$. If  $a_1=0$ then $\l|_{H^0} = a_{n-1}(\o_{1,l}+\o_{2,l})$ and we see that $\l$ is the only weight in $\L(V)$ that affords the highest weight of a $KH^0$-composition factor. Therefore $V|_{H^0}$ is irreducible, and by inspecting \cite[Table 1]{Seitz2}, we deduce that $\l=\l_{n-1}$ is the only possibility. Now assume $a_1 \neq 0$. Here $V|_{H^0}$ is reducible and we calculate that there are precisely two $KH^0$-composition factors, which are afforded by $\l$ and $\mu=\l-\a_1 - \cdots - \a_l$. Set
$$\mu_1 = \l - \a_{l} - \cdots - \a_{n-3} - 2\a_{n-2} - \a_{n-1} - \a_{n}$$
and
$$\mu_2 = \l - \a_{l} - \cdots - \a_{n-3} - 2\a_{n-2} - 2\a_{n-1}.$$
Then $\mu_1$ and $\mu_2$ are weights of $V$ that restrict to $\l|_{H^0} - \b_{1,l}-\b_{2,l-1}-\b_{2,l}$. Both of these weights can only occur in the $KH^0$-composition factor afforded by $\l$ (indeed, note that $\l|_{H^0}-\mu|_{H^0} = \o_{1,1} - \o_{2,1} = \sum_{i=1}^{l}(\b_{1,i}-\b_{2,i})$), but the $T_{H^0}$-weight $\l|_{H^0} - \b_{1,l}-\b_{2,l-1}-\b_{2,l}$ has multiplicity $1$ in this factor. This contradiction implies that $a_{n-1} \le 1$. Similarly, if we assume $a_{n} \neq 0$ then an entirely similar argument yields $a_{n} \le 1$. Therefore, we have also reduced to the cases labelled \eqref{e:cases} when $t=2$ (with $l \ge 3$).

Clearly, the case $\l=\l_1$ is an example, while Lemma \ref{l:c2_special} implies that the highest weights $\l_{n-1}$ and $\l_n$ also provide examples with $V|_{H}$ irreducible.
Finally, suppose $\l=\l_1+\l_{n-1}$. Here Lemma \ref{l:dm} states that 
\begin{equation}\label{e:star2}
\dim V = \left\{\begin{array}{ll}
2^n(n-1) & \mbox{if $p \mid n$} \\
2^{n-1}(2n-1) & \mbox{otherwise,}
\end{array} \right.
\end{equation} 
and we have $\l|_{H^0} = \o_{1,1}+\sum_{i=1}^{t}\o_{i,l}$. By applying Lemma \ref{l:bn2} we see that $\dim L_{B_l}(\o_{1,1}+\o_{1,l}) = d$, where
$$d= \left\{\begin{array}{ll} 2^{l}(2l-1) & \mbox{if $p \mid 2l+1$} \\
2^{l+1}l & \mbox{otherwise,} \end{array}\right.$$
whence each $KH^0$-composition factor has dimension $d\cdot 2^{l(t-1)}$. (Note that if $p$ divides $2l+1$ then $p$ also divides $n$ since $n$ is divisible by $2l+1$.) 

Suppose $t=2$, so $n=2l+1$. Here each $KH^0$-composition factor has dimension $2^{n-1}(n-\a)$, where $\a=2$ if $p$ divides $n$, otherwise $\a=1$. In particular, if $V|_{H}$ is irreducible then $\dim V =2^{n-1}(n-\a)$ or $2^{n}(n-\a)$, but this is incompatible with \eqref{e:star2}, so we can eliminate the case $\l=\l_1+\l_{n-1}$ when $t=2$. Finally, if $t \ge 4$ then Lemma \ref{l:c2_special_2} implies that $V|_{H^0}$ has a composition factor with highest weight $\sum_{i=1}^{t}\o_{i,l}$. But this is not conjugate to $\l|_{H^0}$, which is a contradiction. The case $\l=\l_1+\l_n$ is entirely similar.
\end{proof}

\begin{lem}\label{c2:p5}
Proposition \ref{TH:C2} holds in case (v) of Table \ref{t:c2s}.
\end{lem}

\begin{proof}
Here $G=D_n$, $H=(D_l^t.2^{t-1}).S_t$, $n=lt$ and $n \ge 4$.
Note that $H$ is a maximal rank subgroup of $G$, so $T_{H^0} = T$.
If $l=1$ then $H=N_{G}(T)$ is the normalizer of a maximal torus of $G$, so $\l$ is \emph{minimal} (see \cite[Lemma 2.4]{g_paper} and the paragraph preceding Lemma \ref{l:dmspin}), and so the only examples are $\l=\l_1$, $\l_{n-1}$ and $\l_{n}$ (see \eqref{e:mini}). For the remainder we may assume $l \ge 2$.

Up to conjugacy we have $H^0 = X_1 \cdots X_t$, where
$$X_i = \la U_{\pm \a_{(i-1)l+1}}, U_{\pm \a_{(i-1)l+2}}, \ldots, U_{\pm \a_{(i-1)l+l-1}}, U_{\pm \gamma_{i}} \ra,$$
$\gamma_{t}=\a_n$ and
$$\gamma_{i} =\a_n+\a_{n-1}+\a_{il-1}+2\sum_{j=il}^{n-2}{\a_j}$$
for all $i<t$.
In particular, if $\mu = \sum_{i=1}^{n}{b_i\l_i}$ is a weight for $T$ then
$$\mu|_{X_{i}} = \sum_{j=1}^{l-1}{b_{(i-1)l+j}\o_{i,j}}+v_i(\mu)\o_{i,l}$$
with $v_t(\mu) = b_n$ and
$$v_i(\mu) = b_n+b_{n-1}+b_{il-1}+2\sum_{j=il}^{n-2}{b_j}$$
for $i<t$.
Clearly, two weights $\mu,\nu \in \L(V)$ have the same restriction $\mu|_{H^0} = \nu|_{H^0}$ if and only if $\mu=\nu$. Also note that the standard graph automorphism of $X_{i}$ acts on the set of simple roots of $X_i$ by swapping $\a_{(i-1)l+l-1}$ and $\gamma_i$. We set $v_{i}=v_{i}(\l)$ for all $i$.

Suppose $\mu \in \L(V)$ affords the highest weight of a composition factor of $V|_{H^0}$. Recall from Section \ref{ss:prel_c2} that there exists $\s \in S_t$ and a collection of permutations $\{\rho_1, \ldots, \rho_t\}$ in $S_{l}$ such that $\rho_i = (l-1,l)$ if $\rho_{i}\neq 1$, and precisely an even number $k \ge 0$ of the $\rho_i$ are non-trivial, and we have
$$\mu|_{X_{\s(i)}} = \sum_{j=1}^{l-2}{a_{(i-1)l+j}\o_{\s(i),j}}+a_{(i-1)l+l-1}\o_{\s(i),\rho_i(l-1)}+v_{i}\o_{\s(i),\rho_i(l)}.$$

For now, let us assume $l \ge 3$. If $a_{ml} \neq 0$ for some $m \in \{1, \ldots, t-1\}$ then
$\mu = \l - \a_{ml} \in \L(V)$ affords the highest weight of a $KH^0$-composition factor but
$\mu|_{H^0} = \l|_{H^0}+\o_{m,l-1}-\o_{m,l}+\o_{m+1,1}$ and this contradicts Lemma \ref{l:eh}.

Next suppose $r \in \{1, \ldots, l-2\}$ is minimal such that $a_{ml-r} \neq 0$ for some $m \in \{1, \ldots, t-1\}$, so $\mu = \l-\a_{ml-r} - \a_{ml-r+1} - \cdots - \a_{ml} \in \L(V)$
affords the highest weight of a $KH^0$-composition factor. If $r>1$ then $h(\mu|_{H^0}) = h(\l|_{H^0})+1$, which is a contradiction. Now assume $r=1$, so
$$\mu|_{H^0} = \l|_{H^0}+\o_{m,l-2}-\o_{m,l-1}-\o_{m,l}+\o_{m+1,1}.$$
Let $\s \in S_t$ be the associated permutation and suppose $\s(i)>m$ for some $i \le m$.   
If $\rho_{i}$ is trivial then $v_{i}=v_{\s(i)}$ and thus $v_{m}=v_{m+1}$. Similarly, if $\rho_{i} \neq 1$ then $v_{i} = a_{\s(i)l-1} \le v_{\s(i)}$ and so again we get $v_{m}=v_{m+1}$, which is not possible since $a_{ml-1} \neq 0$. Therefore $\{1, \ldots, m\}$ is $\s$-invariant, but this contradicts Lemma \ref{l:eh}.

Now suppose $r \in \{1, \ldots, l-2\}$ is minimal such that $a_{ml+r} \neq 0$ for some $m \in \{1, \ldots, t-1\}$. Here $\mu = \l-\a_{ml}  - \cdots - \a_{ml+r} \in \L(V)$
affords the highest weight of a $KH^0$-composition factor. For now, let us assume $(m,r) \neq (t-1,l-2)$, so $ml+r \le n-3$ and $\mu = \l+\l_{ml+r+1}-\l_{ml+r}-\l_{ml}+\l_{ml-1}$. If $r=l-2$ then
$$\mu|_{H^0} = \l|_{H^0}+\o_{m,l-1}-\o_{m,l}-\o_{m+1,l-2}+\o_{m+1,l-1}+\o_{m+1,l}$$
and this contradicts Lemma \ref{l:eh}. Now assume $r<l-2$, so
$$\mu|_{H^0} = \l|_{H^0}+\o_{m,l-1}-\o_{m,l}-\o_{m+1,r}+\o_{m+1,r+1}.$$
Let $\s \in S_t$ be the associated permutation and observe that
$\{1, \ldots, m\}$ is $\s$-invariant, so $\s(i)=m$ for some $i<m$.
If $\rho_i=1$ then $v_{i}=v_{m}-1$ and thus
$$2\sum_{j=il}^{ml-1}{a_j} = a_{ml-1}-a_{il-1}-1,$$
which is absurd since we have already established $a_{ml-1}=a_{il-1}=0$. Similarly, if $\rho_i \neq 1$ then $v_{i} = a_{ml-1}+1$, so $v_{i}=1$ since $a_{ml-1}=0$. Therefore $v_{m} \le 1$, but this is not possible since $a_{ml+r}\neq 0$. Finally, if $(m,r) = (t-1,l-2)$ then
$$\mu|_{H^0} = \l|_{H^0}+\o_{t-1,l-1}-\o_{t-1,l}-\o_{t,l-2}+\o_{t,l-1}+\o_{t,l},$$
which contradicts Lemma \ref{l:eh}. Notice that we have now reduced to the case
$$\l = a_1\l_1+a_{n-1}\l_{n-1}+a_n\l_n.$$

If $a_1 \neq 0$ then
$\mu = \l - \a_1 - \cdots - \a_{l} \in \L(V)$
affords the highest weight of a $KH^0$-composition factor and $\mu|_{H^0} = \l|_{H^0}-\o_{1,1}+\o_{2,1}$, so $a_1=1$ is the only possibility.

Next suppose $a_{n-1}\neq 0$. Then $\mu = \l - \a_{l}- \cdots - \a_{n-1} \in \L(V)$
affords the highest weight of a $KH^0$-composition factor and
$$\mu|_{H^0} = \l|_{H^0}+\o_{1,l-1}-\o_{1,l}-\o_{t,l-1}+\o_{t,l}.$$
By considering the associated permutation $\s \in S_t$ we quickly deduce that $a_{n-1}=1$.
Moreover, if $a_{n-1}a_n \neq 0$ then $\s(t)=1$ is the only possibility, so $a_1=0$.

Similarly, if $a_n \neq 0$ then $\mu = \l - \a_{l}- \cdots - \a_{n-2} - \a_n \in \L(V)$
affords the highest weight of a $KH^0$-composition factor and
$$\mu|_{H^0} = \l|_{H^0}+\o_{1,l-1}-\o_{1,l}+\o_{t,l-1}-\o_{t,l}.$$
It follows that $a_n=1$ is the only possibility. We have now reduced to the following specific list of cases:
$$\l_1, \l_{n-1}, \l_{n}, \l_1+\l_{n-1}, \l_{1}+\l_{n}, \l_{n-1}+\l_{n}.$$

Suppose $\l=\l_{n-1}+\l_{n}$. Then
$\mu_1 = \l-\a_{l}-\cdots - \a_{n-1}$ and $\mu_2 = \l-\a_{l}-\cdots - \a_{n-2} - \a_n$ are weights of $V$
that afford the highest weights of $KH^0$-composition factors.
Set $\mu = \l-\a_{l}-\cdots - \a_{n} \in \L(V)$ and note that $\mu|_{H^0}$ is not conjugate to $\l|_{H^0}$, so $\mu$ does not afford the highest weight of a $KH^0$-composition factor. Since
$m_{V}(\mu)=m_{V}(\l-\a_{n-2}-\a_{n-1}-\a_{n})$, Lemma \ref{l:sr} implies that
$m_{V}(\mu)$ is equal to the multiplicity of the zero weight in the action of $A_{3}$ on the non-trivial irreducible constituent of its Lie algebra. In other words, $m_{V}(\mu) = 3-\delta_{2,p}$.

By inspecting the above root groups for $H^0$, we deduce that $\mu$ can only occur in the $KH^0$-composition factors afforded by $\mu_1$ and $\mu_2$. If $p \neq 2$ then $\mu$ has multiplicity $1$ in each factor so $m_{V}(\mu) =2$, which is a contradiction. Now assume $p=2$. Since  
%On the other hand, if $p=2$ then we claim that $\mu$ is the highest weight of a $KH^0$-composition factor. Indeed, since
$$\mu|_{H^0} = \mu_1|_{H^0} + \o_{t,l-2}-2\o_{t,l} = \mu_2|_{H^0} + \o_{t,l-2} - 2\o_{t,l-1}$$
we observe that $\mu = \mu_1-\a_{n}$ is not a weight of the $KH^0$-composition factor afforded by $\mu_1$, and nor is $\mu = \mu_2-\a_{n-1}$ a weight of the factor afforded by $\mu_2$. This final contradiction eliminates the case $\l = \l_{n-1}+\l_n$. 
%This justifies the claim. However,
%$$\mu|_{H^0} = \l|_{H^0}+\o_{1,l-1}-\o_{1,l}+\o_{t,l-2}-\o_{t,l-1}-\o_{t,l}$$
%and this contradicts Lemma \ref{l:eh}.

Now let $\l=\l_{1}+\l_{n-1}$ and suppose $\nu$ is the highest weight of a $KH^0$-composition factor. If $V|_{H}$ is irreducible then $\nu$ is conjugate to $\l|_{H^0}$ and the corresponding $KH^0$-composition factor is of the form
$$U \otimes V_1 \otimes \cdots \otimes V_{t-1},$$
where $U=L_{D_{l}}(\l_1+\l_{l-1})$ or $L_{D_{l}}(\l_1+\l_{l})$, and each $V_i$ is a spin module for $D_{l}$.
In view of Lemmas \ref{l:dmspin} and \ref{l:dm}, each composition factor has dimension
$$2^{l-1}(2l-\a)\cdot 2^{(t-1)(l-1)},$$
where $\a = 2$ if $p$ divides $l$, otherwise $\a=1$. By applying all possible permutations in $S_{t}$, and all possible combinations of transpositions $\{\rho_{1}, \ldots, \rho_{t}\}$, we calculate that there are precisely $N$ distinct conjugates of $\l|_{H^0}$, where
$$N = \sum_{i=1}^{t}{\frac{t!}{(i-1)!(t-i)!}} = 2^{t-1}t.$$
Moreover, this is precisely the number of $KH^0$-composition factors since we have previously observed that two weights in $V$ have the same restriction to $T_{H^0}$ if and only if they are equal. Now Lemma \ref{l:dm} states that
$$\dim V = 2^{n-1}(2n-\b),$$
where $\b=2$ if $p$ divides $n$, otherwise $\b=1$. Since
$$N\cdot 2^{l-1}(2l-\a)\cdot 2^{(t-1)(l-1)} = 2^{n-1}(2n-t\a)$$
it follows that $V|_{H}$ is irreducible if and only if $t\a=\b$, so we must have $t=2$, $\a=1$ and $\b=2$, which implies that $l$ is odd and $p=2$. This is recorded in Table \ref{t:c2}. The case $\l=\l_1+\l_n$ is entirely similar.

Finally, let us assume $\l=\l_1, \l_{n-1}$ or $\l_n$. If $\l=\l_1$ then $V$ is the natural $KG$-module and $V|_{H}$ is irreducible. Next suppose $\l=\l_{n-1}$ and $\mu$ is the highest weight of a $KH^0$-composition factor. If $V|_{H}$ is irreducible then $\mu|_{H^0}$ is conjugate to $\l|_{H^0}$ and the corresponding composition factor is a tensor product of $t$ spin modules for $D_{l}$, and therefore has dimension $2^{t(l-1)}$.
By applying all possible permutations in $S_{t}$, and all possible combinations of transpositions $\{\rho_{1}, \ldots, \rho_{t}\}$, we calculate that there are precisely $2^{t-1}$ distinct conjugates of $\l|_{H^0}$ and we conclude that $V|_{H}$ is irreducible since
$$2^{t-1}\cdot 2^{t(l-1)} = 2^{n-1} = \dim V.$$
An entirely similar argument applies when $\l=\l_n$. Again, these cases are listed in Table \ref{t:c2}.

To complete the proof of the lemma we may assume $l=2$. Suppose $a_{2m} \neq 0$ with $m<t$. Then $\mu=\l-\a_{2m}\in \L(V)$ affords the highest weight of a $KH^0$-composition factor, but
$$\mu|_{H^0} = \l|_{H^0}+\o_{m,1}-\o_{m,2}+\o_{m+1,1}+\o_{m+1,2}$$
and this contradicts Lemma \ref{l:eh}. Similarly, if $a_{2m+1}\neq 0$ for some $m \in \{1, \ldots, t-2\}$ then $\mu = \l-\a_{2m}-\a_{2m+1}-\a_{2m+2} \in \L(V)$ affords the highest weight of a $KH^0$-composition factor, but this is not possible since
$$\mu|_{H^0} = \l|_{H^0}+\o_{m,1}-\o_{m,2}+\o_{m+2,1}+\o_{m+2,2}.$$
We have now reduced to the case
$\l=a_1\l_1+a_{n-1}\l_{n-1}+a_{n}\l_{n}$. Arguing as before, we deduce that $a_{i} \le 1$ for all $i$, so the remaining possibilities for $\l$ are the following:
$$\l_{1},\l_{n-1}, \l_{n}, \l_{1}+\l_{n-1}, \l_{1}+\l_{n}, \l_{n-1}+\l_{n}, \l_{1}+\l_{n-1}+\l_{n}.$$

Here the first three possibilities give examples as before, while our earlier argument in the case $l \ge 3$ rules out the case $\l=\l_{n-1}+\l_n$.

Let $\l=\l_1+\l_{n-1}$ and assume $t \ge 3$. If $V|_{H}$ is irreducible then each composition factor of $V|_{H^0}$ is of the form $U \otimes V_1 \otimes \cdots \otimes V_{t-1}$, where $U = 1 \otimes 2$ or $2 \otimes 1$, and $V_i=1 \otimes 0$ or $0 \otimes 1$ for all $i$ (as modules for $D_2=A_1A_1$). In particular, $\dim U=6-2\delta_{2,p}$ and $\dim V_i=2$ for all $i$, so
$\dim V = 2^{t-1}t \cdot 2^{t-1}(6-2\delta_{2,p})$ since there are exactly $2^{t-1}t$ distinct conjugates of $\l|_{H^0}$, as before. However, Lemma \ref{l:dm} gives $\dim V = 2^{2t-1}(4t-\b)$, where $\b=2$ if $p$ divides $n$, otherwise $\b=1$, which is a contradiction since $t \ge 3$. Similar reasoning applies when $t=2$, and the case $\l=\l_1+\l_n$ is ruled out in an entirely similar fashion.

Finally, suppose $\l=\l_{1}+\l_{n-1}+\l_{n}$. Here $\mu= \l-\a_{2}- \cdots - \a_{n-1} \in \L(V)$
affords the highest weight of a $KH^0$-composition factor and
$$\mu|_{H^0} = \l|_{H^0}+\o_{1,1}-\o_{1,2}-\o_{t,1}+\o_{t,2}.$$
Now $\l|_{X_{1}} = \o_{1,1}+3\o_{1,2}$, while $\mu|_{X_{1}} = 2\o_{1,1}+2\o_{1,2}$ and $\mu|_{X_{i}} = 2\o_{i,2}$ for all $i>1$. Therefore $\mu|_{H^0}$ is not conjugate to $\l|_{H^0}$, so the case $\l=\l_1+\l_{n-1}+\l_{n}$ can also be eliminated. 
\end{proof}

This completes the proof of Proposition \ref{TH:C2}.

\chapter{Tensor product subgroups, I}\label{s:c4i}

In this section we prove Theorem \ref{main} in the case where $H$ is a tensor product subgroup in the 
$\C_{4}(i)$ collection. Recall from Section \ref{ss:c4} that such a subgroup $H$ stabilizes a tensor product decomposition
$$W = W_{1}\otimes W_{2}$$
of the natural $KG$-module $W$,  and $H$ is of the form $H=N_G(Cl(W_1) \otimes Cl(W_2))$ for certain non-isomorphic classical groups $Cl(W_1)$ and $Cl(W_2)$. Here $Cl(W_1) \otimes Cl(W_2)$ denotes the image of the central product $Cl(W_1) \circ Cl(W_2)$ acting naturally on the tensor product. The relevant cases are listed in Table \ref{t:c4is} (see
Proposition \ref{p:cdisc}).

\renewcommand{\arraystretch}{1.2}
\begin{table}[h]
$$\begin{array}{llll} \hline
& G & H & \mbox{Conditions} \\ \hline
{\rm (i)} & C_n & C_aD_b.2 & \mbox{$n = 2ab$, $b \ge 2$, $p\ne2$} \\
{\rm (ii)} & D_n & D_aD_b.2^2 & \mbox{$n=2ab$, $a>b \ge 2$, $p\ne2$} \\
\hline
\end{array}$$
\caption{The disconnected $\C_{4}(i)$ subgroups}
\label{t:c4is}
\end{table}
\renewcommand{\arraystretch}{1}

\section{The main result}

\begin{prop}\label{TH:C4I}
Let $V$ be an irreducible tensor-indecomposable $p$-restricted $KG$-module with
highest weight $\l$ and let $H$ be a disconnected maximal 
$\C_{4}(i)$-subgroup of $G$. Then $V|_{H}$ is irreducible if and only if $(G,H,\l)$ is one of the cases recorded in Table \ref{t:c4i}.
\end{prop}

\begin{remk}\label{f:c4i}
In Table \ref{t:c4i} we adopt the standard labelling $\{\o_{1},\o_2, \ldots\}$ for the fundamental dominant weights of each factor in $H^0$. In the final column, $\kappa$ denotes the number of $KH^0$-composition factors in $V|_{H^0}$. In particular, we note that $V|_{H}$ is irreducible if and only if $V$ is the natural module for $G$.
\end{remk}

\renewcommand{\arraystretch}{1.2}
\begin{table}[h] \small
$$\begin{array}{lllll} \hline
G  & H & \l & \l|_{H^0} & \kappa  \\  \hline
C_n & C_aD_b.2 & \l_1 & \o_{1} \otimes \o_1 & 1  \\
D_n & D_aD_b.2^2 & \l_1 & \o_{1} \otimes \o_1 & 1  \\ \hline
\end{array}$$
\caption{The $\C_4(i)$ examples}
\label{t:c4i}
\end{table}
\renewcommand{\arraystretch}{1}

\section{Proof of Proposition \ref{TH:C4I}}

\begin{lem}\label{l:c4i1}
Proposition \ref{TH:C4I} holds in case (i) of Table \ref{t:c4is}.
\end{lem}

\begin{proof}
Here $G=C_n$ and $p \neq 2$, where $n=2ab$ with $a \ge 1$ and $b \ge 2$. We have 
$H^0=X_1X_2$ with $X_1=C_a$ and $X_2=D_b$, and $H=H^0\la z\ra$ where $z$ centralizes $X_1$ and acts as an involutory graph automorphism on $X_2$. Let $\Pi(X_1)=\{\b_1,
\dots,\b_a\}$ and $\Pi(X_2)=\{\gamma_1,\dots,\gamma_b\}$
be bases of the root systems $\Sigma(X_1)$ and $\Sigma(X_2)$, respectively.

The natural $KG$-module $W$ restricts to $X_1$ as $2b$ copies of the natural module for
$X_1$, and hence up to conjugacy, we may assume that $X_1$ lies in the subgroup
$$\langle U_{\pm\a_i}\mid 2ja+1\leq i\leq 2(j+1)a-1,
0\leq j\leq b-1\rangle,$$
which is the derived subgroup of an $A_{2a-1}\times\cdots\times A_{2a-1}$ ($b$ factors)
Levi subgroup
of $G$. The projection of $X_1$ into each of the factors of this group
is the natural embedding of a symplectic group $C_a$ in $A_{2a-1}$. 
Therefore, up to conjugacy, we may assume that
$$\a_{2ja+i}|_{H^0}=\b_i, \;\; \a_{2ja+a+k}|_{H^0}=\b_{a-k}$$
for all $0\leq j\leq b-1$, $1\leq i\leq a$ and $1\leq k \leq a-1$.
Moreover, by considering the action of $H$ on $W$, we
obtain the restrictions of the remaining simple roots:
$$\a_{2ja}|_{H^0}=\gamma_j-\b_0,\;\; \a_{2ba}|_{H^0} = \gamma_b-\gamma_{b-1}-\b_{0}$$
for all $1\leq j\leq b-1$, where 
$\b_0 = 2(\b_1+\cdots+\b_{a-1})+\b_a$ 
is the root of maximal height in $\Sigma(X_1)$. Note that if $\a  = \sum_{i}c_i\a_i$ with $c_i \in \mathbb{N}_0$, then
\begin{equation}\label{e:ci0}
\alpha|_{H^0}=0 \mbox{ if and only if $\a=0$.}
\end{equation}
In particular, $\l$ is the unique weight of $V$ that restricts to $\l|_{H^0}$.

If $V|_{H^0}$ is irreducible, the result follows by the main theorem of
\cite{Seitz2}. So we now assume $V|_{H^0}$ is reducible. Here
 $V|_{H^0}$ has exactly two $KH^0$-composition factors of highest weights
$\l|_{H^0}$ and $\sigma(\l|_{H^0})$, where $\sigma$ is
induced by the graph automorphism of $X_2$. 
In particular, all $KX_1$-composition factors of $V$ have highest weight
$\l|_{X_{1}}$, and consequently if $\mu$ is a weight of $V$ then
\begin{equation}\label{restriction1}
\mu|_{X_1} = \l|_{X_1}-\sum_{j=1}^{a}n_j\b_j
\end{equation}
for some non-negative integers $n_j$.

Write $\l=\sum_{i=1}^{n}a_i\l_i$ and suppose $a_{2ja}\ne0$ for some $1\leq j\leq b$. Then
$\mu=\l-\a_{2ja}\in \L(V)$ and $\mu|_{X_1} = \l|_{X_1}+\b_0$, contradicting \eqref{restriction1}.
Indeed, we can argue more generally: suppose $a_{2ja+k}\ne 0$
for some $1\leq j<b$,
$1\leq k\leq 2a-1$. Then using Lemma \ref{l:pr} (for the group $A_{2a-1}$)
we see that 
$$\mu_1=\l-\a_{2ja}-\a_{2ja+1}-\cdots-\a_{2ja+k}, \;\; 
\mu_2=\l-\a_{2ja+k}-\a_{2aj+k+1}-\cdots-\a_{2(j+1)a}$$
are both weights of $V$. But for $i=1$ or $i=2$ we have $\mu_i|_{X_1} = \l|_{X_1}+r_i$ for some positive root
$r_i\in\Sigma(X_1)$, again contradicting \eqref{restriction1}.

Hence we now have that $a_i=0$ for all $i\geq 2a$. Arguing in this way, we quickly deduce that $a_i=0$ for all $i \ge 2$, so $\l=a_1\l_1$. Here
$\l|_{H^0} = a_1(\l_1|_{H^0})$
is fixed by $\sigma$,  
contradicting the irreducibility of $V|_H$.
\end{proof}

\begin{lem}\label{l:c4i3}
Proposition \ref{TH:C4I} holds in case (ii) of Table \ref{t:c4is}.
\end{lem}

\begin{proof}
Here $G=D_n$ and $p \neq 2$, where $n=2ab$ and $a >b \ge 2$. We have 
$H^0=X_1X_2$ with $X_1=D_a$, $X_2=D_b$ and $H=H^0\la z_1, z_2\ra = H^0.2^2$ where $z_i$ induces an involutory graph automorphism on $X_i$. As before, let $\Pi(X_1)=\{\b_1,
\dots,\b_a\}$ and $\Pi(X_2)=\{\gamma_1,\dots,\gamma_b\}$
be bases of the root systems $\Sigma(X_1)$ and $\Sigma(X_2)$, respectively, and note that
the natural $KG$-module $W$ restricts to $X_1$ as $2b$ copies of the natural module for
$X_1$. Up to conjugacy, we may assume that $X_1$ lies in the subgroup
$$\langle U_{\pm\a_i}\mid 2ja+1\leq i\leq 2(j+1)a-1,
0\leq j\leq b-1\rangle,$$
the derived subgroup of an $A_{2a-1}\times\cdots\times A_{2a-1}$ ($b$ factors)
Levi subgroup
of $G$. The projection of $X_1$ into each of the factors of this group
is the natural embedding of an orthogonal group $D_a$ in $A_{2a-1}$. In particular, up to conjugacy, we may assume that
$$\a_{2ja+i}|_{H^0} = \b_i, \;\; \a_{2ja+a}|_{H^0} = \b_a-\b_{a-1},\;\; \a_{2aj+a+i}|_{H^0}=\b_{a-i}$$
for all $0\leq j\leq b-1$ and $1\leq i\leq a-1$. Similarly, for the remaining roots we have
$$\a_{2ja}|_{H^0} = \gamma_j-2\b_1-2\b_2-\cdots-2\b_{a-2}-\b_{a-1}-\b_a$$
for all $1\leq j\leq b-1$, while
$$\a_{2ba}|_{H^0} = \gamma_b-\gamma_{b-1}-\b_1-2\b_2-\cdots-2\b_{a-2}-\b_{a-1}-\b_a.$$
Therefore, \eqref{e:ci0} holds for all $\a  = \sum_{i}c_i\a_i$ with $c_i \in \mathbb{N}_0$, and thus $\l$ is the unique weight of $V$ that restricts to $\l|_{H^0}$.

We now argue by contradiction, as in the proof of the previous lemma. Write $\l = \sum_{i}a_i\l_i$ and suppose $V|_{H^0}$ is reducible. (If $V|_{H^0}$ is irreducible then the result follows from the main theorem of \cite{Seitz2}.)
Then $V$ has exactly two or four $KH^0$-composition factors, with highest weights
$\l|_{H^0}$, and one, or all, of $\sigma_1(\l|_{H^0})$,
$\sigma_2(\l|_{H^0})$,
$\sigma_1\sigma_2(\l|_{H^0})$,  where $\sigma_i$ is induced by the
graph automorphism of $X_i$, for $i=1,2$. In particular, if $\mu \in \L(V)$ then
\begin{equation}\label{restriction2}
\mu|_{X_1} = \eta|_{X_1} - \sum_{j=1}^{a}n_j\b_j
\end{equation}
for some non-negative integers $n_j$, where we may take
$\eta$ to be $\l$ or $\sigma_1\l$.  It is useful to note that if $\l|_{X_1}=\sum_{i=1}^ab_i\omega_i$,
where $\{\omega_1,\ldots,\omega_a\}$ are the fundamental dominant weights of $X_1$
corresponding to $\Pi(X_1)$, then
$$(\l-\sigma_1\l)|_{X_1}=\frac{1}{2}(b_{a-1}-b_a)(\b_{a-1}-\b_a).$$
In particular, if $\mu \in \L(V)$ and $\mu|_{X_1}$ is of the form $\l|_{X_1}+r$, where
$r$ is a positive integer linear combination of the roots
in $\Pi(X_1)$, then $\mu|_{X_1}$ cannot be of the form
$(\sigma_1\l)|_{X_1}-\sum_{j=1}^an_j\b_j$ with $n_j \in \mathbb{N}_0$, as otherwise this contradicts \eqref{restriction2}.

Suppose $a_i\ne 0$ for some $a\leq i\leq (b-1)2a+a$. Choosing $j$ such that
 $|2aj-i|$ is minimal, we see that there exists a positive root
$r\in\Sigma(G)$ of the form $r=\sum_{k=i}^{2aj}\a_k$, or $r=\sum_{k=2aj}^i\a_k$,
 such that $\nu=\l-r\in\L(V)$. But then
$\nu|_{X_1} = \l|_{X_1}+s$, where $s$ is a
positive integer linear combination of roots in $\Pi(X_1)$, contradicting
the above remarks. Therefore $a_i=0$ for all $a\leq i\leq (b-1)2a+a$.

Next suppose $a_i\ne 0$ with $2\leq i\leq a-1$. Then
$\mu = \l-\a_i-\a_{i+1}-\cdots-\a_{2a}\in\L(V)$ and it is straightforward to check that $\mu|_{X_1} = \l|_{X_1}+\b_1+\b_2+\cdots+\b_{i-1}$, which once again contradicts \eqref{restriction2}. Similarly, suppose that $a_i\ne 0$ for some $i\geq (b-1)2a+a+1$. If $i=n$ then $\l - \a_n \in \L(V)$, but this weight restricts to $\l|_{X_1}+s$ for some $s \in \Sigma^+(X_1)$, which contradicts \eqref{restriction2}. Similarly, if 
$i\le n-2$ then $\l-\a_i-\a_{i+1}-\cdots-\a_{n-2}-\a_n \in \L(V)$, while $\l-\a_{n-1}-\a_{n-2}-\a_n \in\L(V)$ if $i=n-1$. In both cases, these weights restrict to a weight of the form $\l|_{X_1}+s$ for some $s \in \Sigma^+(X_1)$, so once again we reach a contradiction.

We have now reduced to the case $\l=a_1\l_1$. Here
$\l|_{H^0} = a_1(\l_1|_{H^0})$
is stable under the graph automorphisms $\sigma_1$ and $\sigma_2$, but this contradicts the irreducibility of $V|_{H}$.
\end{proof}

This completes the proof of Proposition \ref{TH:C4I}.

\chapter{Tensor product subgroups, II}\label{s:c4ii}

To complete the proof of Theorem \ref{main}, it remains to deal with the tensor product subgroups in the $\C_{4}(ii)$ collection. Let $H$ be such a subgroup, so $H$ stabilizes a tensor product decomposition
$$W = W_{1}\otimes W_{2} \otimes \cdots \otimes W_{t}$$
of the natural $KG$-module $W$, with $t \ge 2$. As noted in Section \ref{ss:c4}, 
$H=N_G(\prod_iCl(W_i))$ where the classical groups $Cl(W_i)$ are simple and isomorphic, and the central product $\prod_iCl(W_i)$ acts naturally on the tensor product. The various cases that arise are described in Section \ref{ss:c4}, and they are listed in Table \ref{t:c4iis}.

\renewcommand{\arraystretch}{1.2}
\begin{table}[h]
$$\begin{array}{llll} \hline
& G & H & \mbox{Conditions} \\ \hline
{\rm (i)} & A_n & A_l^t.S_t & \mbox{$n+1 = (l+1)^t$, $l \ge 2$, $t \ge 2$} \\
{\rm (ii)} & B_n & B_l^t.S_t & \mbox{$2n+1=(2l+1)^t$, $l \ge 1$, $t \ge 2$} \\
{\rm (iii)} & C_n & C_l^t.S_t & \mbox{$2n=(2l)^t$, $l \ge 1$, $t \ge 3$ odd, $p \neq 2$} \\
{\rm (iv)} & D_n & C_l^t.S_t & \mbox{$2n=(2l)^t$, $l \ge 1$, $t \ge 2$ even or $p=2$} \\
{\rm (v)} & D_n & (D_l^t.2^t).S_t & \mbox{$2n=(2l)^t$, $l \ge 3$, $t \ge 2$, $p \neq 2$} \\
\hline
\end{array}$$
\caption{The $\C_{4}(ii)$ subgroups}
\label{t:c4iis}
\end{table}
\renewcommand{\arraystretch}{1}

\section{The main result}

\begin{prop}\label{TH:C4II}
Let $V$ be an irreducible tensor-indecomposable $p$-restricted $KG$-module with highest weight $\l$ and let $H$ be a maximal $\C_{4}(ii)$-subgroup of $G$. Then 
$V|_{H}$ is irreducible if and only if $(G,H,\l)$ is one of the cases recorded in Table \ref{t:c4ii}.
\end{prop}

\begin{remk}
Let us make a couple of comments on the statement of Proposition \ref{TH:C4II}, in particular concerning Table \ref{t:c4ii}.
\begin{itemize}\addtolength{\itemsep}{0.3\baselineskip}
\item[(a)] Consider the case $(G,H)=(D_n, C_l^t.S_t)$ in Table \ref{t:c4ii}, where $n=8$, $\l=\l_7$, and $(t,l)=(4,1)$ or $(2,2)$. If $\tilde{H}$ denotes the image of $H$ under a non-trivial graph automorphism of $G$ then $\l=\l_8$ is an example for 
the pair $(G,\tilde{H})$. Similarly, $\l=\l_1+\l_3$ and $\l_4$ are examples for $(G,\tilde{H})$ when $G=D_4$ and $H$ is of type $C_1^3.S_3$ (with $p=2$). 
\item[(b)] In the fourth column of Table \ref{t:c4ii}, we give the restriction of $\l$ to a suitable maximal torus of $H^0$ (as before, we denote this restriction by $\l|_{H^0}$) in terms of a set of fundamental dominant weights $\{\o_{i,1}, \ldots, \o_{i,l}\}$ for the $i$-th factor $X_i$ in $H^0 = X_1\cdots X_t$. In addition, $\kappa$ denotes the number of $KH^0$-composition factors in $V|_{H^0}$. Of course, any condition appearing in the final column of Table \ref{t:c4iis} must also hold for the relevant cases listed in Table \ref{t:c4ii}. 
\end{itemize}
\end{remk}

\renewcommand{\arraystretch}{1.2}
\begin{table}[h] \small
$$\begin{array}{llllll} \hline
G  & H & \l & \l|_{H^0} & \kappa & \mbox{Conditions} \\  \hline
A_n & A_l^t.S_t & \l_1 & \o_{1,1}+ \cdots + \o_{t,1} & 1 & \\
& & \l_n & \o_{1,l}+ \cdots + \o_{t,l}	& 1 & \\
& & \l_2 & \o_{1,2} + 2\o_{2,1} & 2 & \mbox{$t=2$, $p \neq 2$}\\
& & \l_{n-1} & \o_{1,l-1} + 2\o_{2,l} & 2 & \mbox{$t=2$, $p \neq 2$}\\

B_n & B_l^t.S_t & \l_1 & 	\o_{1,1}+ \cdots + \o_{t,1} & 1	& \\
	&	& \l_4 & \o_{1,1} + 3\o_{2,1} & 2 & \mbox{$(n,t,l)=(4,2,1)$, $p \neq 3$} \\
	
C_n & C_l^t.S_t & \l_1 & 	\o_{1,1}+ \cdots + \o_{t,1} & 1	& \\
	&	& \l_2 & 2\o_{2,1}+ 2\o_{3,1} & 3 & \mbox{$(n,t,l)=(4,3,1)$}\\
	&	& \l_3 & \o_{1,1}+ \o_{2,1}+ 3\o_{3,1} & 3 & \mbox{$(n,t,l)=(4,3,1)$, $p \neq 3$}\\
	
D_n & C_l^t.S_t & \l_1	& 	\o_{1,1}+ \cdots + \o_{t,1} & 1		& \\
	&	& \l_3 & \o_{1,1} + \o_{2,1} + \o_{3,1} & 1 & \mbox{$(n,t,l) = (4,3,1)$, $p=2$} \\
& & \l_1+\l_4 & \o_{1,1} + \o_{2,1} + 3\o_{3,1} & 3 & \mbox{$(n,t,l)=(4,3,1)$, $p=2$} \\
& & \l_3+\l_4 & \o_{1,1} + \o_{2,1} + 3\o_{3,1} & 3 & \mbox{$(n,t,l)=(4,3,1)$, $p=2$} \\
	&	& \l_7 & \o_{1,1} + \o_{2,1} + \o_{3,1} + 3\o_{4,1} & 4 			& \mbox{$(n,t,l)=(8,4,1)$, $p\ne3$}\\
	&	& \l_7 & \o_{1,1} + \o_{2,1}+\o_{2,2} & 2 & \mbox{$(n,t,l)=(8,2,2)$, $p\ne5$}\\
D_n & (D_l^t.2^t).S_t & \l_1 & \o_{1,1}+ \cdots + \o_{t,1} & 1	& \\ \hline
\end{array}$$
\caption{The $\C_4(ii)$ examples}
\label{t:c4ii}
\end{table}
\renewcommand{\arraystretch}{1}

\section{Preliminaries}\label{ss:prel_c4}
 
It is worth noting that the techniques used in this section
are similar to those used in the analysis of the $\C_4(i)$ subgroups in Section \ref{s:c4i},  but they differ
significantly from those used in Section \ref{s:c2}. In particular,
as it is quite difficult to give explicit expressions
for the root elements of $H^0$ in terms
of those for $G$, we will only work with the embedding of a maximal torus
of $H^0$ in a maximal torus of $G$, and the arguments will be purely at the level
of weights. This means that we cannot immediately see when a particular weight
affords the highest weight of a $KH^0$-composition factor. Nevertheless,
one can establish criteria that the weights must satisfy, and this is how we proceed.

Let $H$ be a $\C_{4}(ii)$-subgroup of $G$, with $H^0 = X_{1}\cdots X_{t}$, where the $X_i$ are isomorphic simple classical groups of rank $l$. Let $\{\b_{i,1}, \ldots, \b_{i,l}\}$ and $\{\o_{i,1}, \ldots, \o_{i,l}\}$ respectively denote a set of simple roots and the corresponding fundamental dominant weights for $X_i$. Note that for each $j \in\{1,\dots,l\}$, there exist rational numbers $d_{j,k}$ such that
$\o_{i,j}=\sum_{k=1}^{l}d_{j,k}\b_{i,k}$ for all $1 \le i \le t$. As the notation indicates,  the coefficients $d_{j,k}$ are independent of $i$.

Let $V$ be an irreducible $p$-restricted $KG$-module with highest weight $\l=\sum_{i=1}^{n}a_{i}\l_{i}$
and suppose $\l|_{X_{i}} = \sum_{j=1}^{l}{a_{i,j}\o_{i,j}}$ for each $i$, so
$$\l|_{H^0} = \sum_{i=1}^{t}{\sum_{j=1}^{l}{a_{i,j}\o_{i,j}}}.$$
Suppose $V|_{H}$ is irreducible and assume for now that we are not in case (v) of Table \ref{t:c4iis}. If $\mu \in \L(V)$ affords the highest weight of a $KH^0$-composition factor then there exists a  permutation $\s \in S_t$ such that
\begin{equation}\label{e:yb}
\mu|_{H^0} = \sum_{i=1}^{t}\sum_{j=1}^{l}{a_{i,j}\o_{\s(i),j}} =\l|_{H^0}+\sum _{i=1}^{t}\sum_{j=1}^{l}c_{i,j}\b_{\s(i),j},
\end{equation}
where
$$c_{i,j}=\sum_{k=1}^{l}d_{k,j}(a_{i,k}-a_{\s(i),k}).$$
Following the terminology introduced in Section \ref{s:c2}, we call $\s$ the \emph{associated permutation} of $\mu$. Observe that if a subset $S$ of $\{1, \ldots, t\}$ is $\s$-invariant then
$$\sum_{i \in S}c_{i,j}=\sum_{k=1}^{l}d_{k,j}\sum_{i \in S}(a_{i,k}-a_{\s(i),k})=0$$
for all $j \in \{1,\dots,l\}$.
In the special case $S=\{1, \ldots, t\}$, we have $\sum_{i=1}^{t}c_{i,j}=0$. Equivalently, if we set
\begin{equation}\label{e:ell}
\ell(\mu|_{H^0}) = \sum_{j=1}^l{\ell_j(\mu|_{H^0})} \mbox{ and } \ell_j(\mu|_{H^0}) = \sum_{i=1}^{t}{c_{i,j}},
\end{equation}
then $\ell(\mu|_{H^0})=0$.

The situation in case (v) of Table \ref{t:c4iis} is very similar. Here we use the standard labelling of simple roots for $X_i$ so that the standard graph automorphism of $X_{i}$ swaps the simple roots $\b_{i,l-1}$ and $\b_{i,l}$. Now, if $V|_{H}$ is irreducible and $\mu \in \L(V)$ affords the highest weight of a $KH^0$-composition factor then there exists an associated permutation, $\s \in S_t$ say, and a collection of permutations $\{\rho_1, \ldots, \rho_t\}$ in $S_{l}$ such that $\rho_i = (l-1,l)$ if $\rho_{i}\neq 1$,
and
$$\mu|_{X_{\s(i)}} = \sum_{j=1}^{l-2}{a_{i,j}\o_{\s(i),j}}+a_{i,l-1}\o_{\s(i),\rho_i(l-1)}+a_{i,l}\o_{\s(i),\rho_i(l)}$$
for all $i$. Write
$\mu|_{H^0}=\l|_{H^0}+\sum_{i=1}^{t}\sum_{j=1}^{l}c_{i,j}\b_{\s(i),j}$, where
$$c_{i,j}=
\sum_{k=1}^{l-2}(a_{i,k}-a_{\s(i),k})d_{k,j}
+(a_{i,l-1}-a_{\s(i),\rho_i(l-1)})d_{\rho_i(l-1),j}
+(a_{i,l}-a_{\s(i),\rho_i(l)})d_{\rho_i(l),j},$$
and observe that $d_{l-1,j}=d_{l,j}$ for all $j \le l-2$, $d_{l-1,l-1}=d_{l,l}$ and $d_{l-1,l}=d_{l,l-1}$.
We deduce that
\begin{align*}
\ell_j(\mu|_{H^0}) = & \sum_{k=1}^{l-2}d_{k,j}\sum_{i=1}^{t}(a_{i,k}-a_{\s(i),k}) \\
& + d_{l-1,j}\sum_{i=1}^{t}\left(a_{i,l-1}-a_{\s(i),\rho_i(l-1)}+a_{i,l}-a_{\s(i),\rho_i(l)}\right)=0
\end{align*}
for all $j \le l-2$, and 
$$\ell_{l-1}(\mu|_{H^0})+\ell_{l}(\mu|_{H^0}) =
(d_{l-1,l-1}+d_{l-1,l})\sum_{i=1}^{t}
\left(
a_{i,l-1}+a_{i,l}-a_{\s(i),l-1}-a_{\s(i),l}
\right)
=0,$$
so $\ell(\mu|_{H^0})=0$.

Finally, recall that if $\nu \in \L(V)$ and $\nu|_{H^0}$ occurs as a weight of $L_{H^0}(\mu|_{H^0})$ then 
$\nu|_{H^0}=\mu|_{H^0}-\sum _{i=1}^{t}\sum_{j=1}^{l}e_{i,j}\b_{i,j}$ for some non-negative integers $e_{i,j}$, so $\ell_j(\nu|_{H^0}) \le \ell_j(\mu|_{H^0})$ for all $j$. For easy reference, let us record these general observations.

\begin{lem}\label{l:sumc}
Suppose $V|_{H}$ is irreducible and $\mu \in \L(V)$.
\begin{itemize}\addtolength{\itemsep}{0.3\baselineskip}
\item[{\rm (i)}] We have $\ell(\mu|_{H^0})\le 0$, with equality if and only if $\mu$ affords the highest weight of a $KH^0$-composition factor.
\item[{\rm (ii)}] Moreover, if $\ell(\mu|_{H^0})= 0$ then $\ell_j(\mu|_{H^0})=0$ for all $j$, unless $H=(D_l^t.2^t).S_t$ and $j\in\{l-1,l\}$, in which case $\ell_{l-1}(\mu|_{H^0})+\ell_l(\mu|_{H^0})=0$.
\end{itemize}
\end{lem}

\section{Proof of Proposition \ref{TH:C4II}}\label{ss:pc4}

\begin{lem}\label{c4ii:p1}
Proposition \ref{TH:C4II} holds in case (i) of Table \ref{t:c4iis}.
\end{lem}

\begin{proof}
Here $G=A_n$, $n+1=(l+1)^t$ and $H^0 = X_1 \cdots X_t$, where $X_{i} \cong A_{l}$ and $l \ge 2$.
Recall that $\{\b_{i,1}, \ldots, \b_{i,l}\}$ is a set of simple roots for $X_{i}$, with corresponding fundamental dominant weights $\{\o_{i,1},\ldots,\o_{i,l}\}$. It will be convenient to set $\l_0=\l_{n+1}=0$ and $\o_{i,0}=\o_{i,l+1}=0$ for all $1 \le i \le t$. 

Let $0 \le k \le n$ be an integer. Then there exist unique integers $r_k(i) \in \{0,\ldots,l\}$ such that
\begin{equation}\label{e:rik1}
k=\sum_{i=0}^{t-1}r_k(i)(l+1)^i.
\end{equation}

For $t' \in \{1,\ldots,t\}$ let $k'=\sum_{i=0}^{t'-1}r_{k}(i)(l+1)^i$. By choosing an appropriate embedding of $H^0$ in $G$, we may assume that 
$$\langle\l_{k}|_{H^0},\b_{i,j}\rangle=\langle\l_{k'}|_{H^0},\b_{i,j}\rangle$$
for all $1 \le i \leq t'$ and $1 \le j \le l$. Let $\xi_k$ be the weight  $-\l_{k}+\l_{k+1}$ of the natural $KG$-module $W$. Then we deduce that
\begin{equation}\label{e:xi}
\xi_k|_{H^0} = \sum_{i=0}^{t-1}\left(-\o_{i+1,r_k(i)}+\o_{i+1,r_k(i)+1}\right).
\end{equation}
Now assume $1 \le k \le n$. Let $\mu=-\l_{k-1}+2\l_{k}-\l_{k+1}$ and let $0 \le i_k \le t-1$ be minimal such that $r_k(i_k) \neq 0$ in \eqref{e:rik1}. Note that $\mu = \xi_{k-1}-\xi_{k}$.
If $i_k=0$ then 
$$k-1=(r_k(0)-1)+\sum_{i=1}^{t-1}r_k(i)(l+1)^i$$ 
and therefore \eqref{e:xi} yields
$$\mu|_{H^0}= -\o_{1,r_k(0)-1}+2\o_{1,r_k(0)}-\o_{1,r_k(0)+1}.$$
Similarly, if $i_k>0$ then
$$k-1=\sum_{i=0}^{i_k-1}l(l+1)^i+(r_k(i_k)-1)(l+1)^{i_k}+\sum_{i=i_k+1}^{t-1}r_k(i)(l+1)^i$$ and thus
\begin{align*}
\mu|_{H^0}  = & \sum_{i=0}^{i_k-1}\left(-\o_{i+1,l}+\o_{i+1,l+1}+\o_{i+1,0}-\o_{i+1,1}\right)
-\o_{i_k+1,r_k(i_k)-1} \\
& +2\o_{i_k+1,r_k(i_k)}-\o_{i_k+1,r_k(i_k)+1} \\
 = & -\sum_{i=1}^{i_k}(\o_{i,1}+\o_{i,l})-\o_{i_k+1,r_k(i_k)-1}+2\o_{i_k+1,r_k(i_k)}-\o_{i_k+1,r_k(i_k)+1}.
\end{align*}
Since $\mu=\a_k$ we conclude that
\begin{equation}\label{e:akh1}
\a_{k}|_{H^0}=\left\{\begin{array}{ll}
\displaystyle \b_{i_k+1,r_k(i_k)}-\sum_{i=1}^{i_k}(\b_{i,1}+\dots+\b_{i,l}) & \mbox{if $i_k>0$} \\
\b_{1,r_k(0)} & \mbox{otherwise.} \\
\end{array}\right.
\end{equation}
In particular, note that if $\a=\sum_{i}c_i\a_i$, with $c_i \in \mathbb{N}_0$, then \eqref{e:ci0} holds.

Recall that $V$ has highest weight $\l=\sum_{i=1}^{n}a_i\l_i$ and assume that $V|_{H}$ is irreducible. Suppose $1 \le k<n-l$ and $a_k \neq 0$. We claim that $k \le 2$. To see this, define the integers $r_k(i)$ as in \eqref{e:rik1} and write $k$ in the form $k=r_k(0)+(a-1)(l+1)$, where $a \in \{1,\dots,(l+1)^{t-1}-1\}$.
Note that $\mu=\l-\a_{k} \in \L(V)$.

If $r_k(0)=0$ then $i_k>0$, so \eqref{e:akh1} implies that
$$\mu|_{H^0} = \l|_{H^0}+\sum_{i=1}^{i_k}(\b_{i,1}+\dots+\b_{i,l}) -\b_{i_k+1,r_k(i_k)}$$
and thus $\ell(\mu|_{H^0})=li_k-1 \ge 1$ (see \eqref{e:ell}), which contradicts Lemma \ref{l:sumc}. Therefore $r_k(0) \ge 1$.
Consequently, $\nu = \l-\a_{k}-\a_{k+1}-\dots-\a_{a(l+1)} \in \L(V)$ and using \eqref{e:akh1} we deduce that
\begin{equation}\label{e:res}
\begin{aligned}
\nu|_{H^0} = & \, \l|_{H^0}-\b_{1,r_k(0)}-\dots-\b_{1,l}-\b_{i_{a(l+1)}+1,r_{a(l+1)}(i_{a(l+1)})} \\
& + \sum_{i=1}^{i_{a(l+1)}}(\b_{i,1}+\dots+\b_{i,l}).
\end{aligned}
\end{equation}
Now 
$$a(l+1)=(r_k(1)+1)(l+1)+\sum_{i=2}^{t-1}r_k(i)(l+1)^{i}$$ 
and $\ell(\nu|_{H^0})=l(i_{a(l+1)}-1)+r_k(0)-2$, so Lemma \ref{l:sumc} implies that $i_{a(l+1)}=1$ and $r_k(0) \le 2$, hence $r_k(1) \leq l-1$.

First assume $r_k(0)=1$. If $k \neq 1$
			then $\mu=\l-\a_{k}-\a_{k-1}=\l-\a_{(a-1)(l+1)+1}-\a_{(a-1)(l+1)} \in \L(V)$ and once again, using \eqref{e:akh1}, we get
			 $$\mu|_{H^0}=\l|_{H^0}-\b_{1,1}-\b_{i_{k-1}+1,r_{k-1}(i_{k-1})}+\sum_{i=1}^{i_{k-1}}(\b_{i,1}+\dots+\b_{i,l}),$$
			where $i_{k-1} \geq 1$.
			Since $l \geq 2$, Lemma \ref{l:sumc} implies that
			$i_{k-1}=1$ and $l=2$, so $\mu|_{H^0}$ is the highest weight of a $KH^0$-composition factor (see Lemma \ref{l:sumc}) and thus \eqref{e:ell} indicates that $r_k(1)=2$. This is a contradiction since $r_k(1) \leq l-1$. We conclude that $k=1$ if $r_k(0)=1$.
			
Now suppose $r_k(0)=2$. First observe that $\nu|_{H^0} = \l|_{H^0}+\b_{1,1}-\b_{2,r_k(1)+1}$, where $\nu = \l-\a_{k}-\a_{k+1}- \cdots - \a_{a(l+1)}$ (see \eqref{e:res}). Therefore
			\eqref{e:ell} implies that $r_k(1)=0$, so $k=2+\sum_{i=2}^{t-1}r_k(i)(l+1)^i$. If $k>2$ then
						$\mu = \l-\a_{k}-\a_{k-1}-\a_{k-2} \in \L(V)$ and
			$$\mu|_{H^0}= \displaystyle \l|_{H^0}-\b_{1,1}-\b_{1,2}-\b_{i_{k-2}+1,r_{k-2}(i_{k-2})}+\sum_{i=1}^{i_{k-2}}(\b_{i,1}+\dots+\b_{i,l}),$$
			where $i_{k-2} \geq 2$. Since $l \geq 2$, this contradicts Lemma \ref{l:sumc}, so $k=2$ is the only possibility. This justifies the claim.
	
Next suppose $a_k \neq 0$ with $k \geq n-l$. Write 
$$k=(l+1)^t-c=l+1-c+\sum_{i=1}^{t-1}l(l+1)^{i}.$$ 
Then $\mu = \l-\a_{k}-\a_{k-1}-\dots-\a_{n-l} \in \L(V)$ and \eqref{e:akh1} implies that
$$\mu|_{H^0} = \l|_{H^0}-\b_{1,l+1-c}-\b_{1,l-c}-\dots-\b_{1,1}
		-\b_{2,l}+\sum_{j=1}^{l}\b_{1,j}$$
		(note that $n-l=l\sum_{i=1}^{t-1}(l+1)^i$).
		Therefore Lemma \ref{l:sumc} yields $c \le 2$, whence $k=n-1$ or $n$.
Notice that we have now reduced to the case
$$\l=a_1\l_1+a_2\l_2+a_{n-1}\l_{n-1}+a_n\l_n.$$

Now
$$\l_1|_{H^0} = \xi_0|_{H^0} = \sum_{i=1}^{t}\o_{i,1},
\;\; \l_2|_{H^0} = (\xi_0+\xi_1)|_{H^0} = \o_{1,2}+2\sum_{i=2}^{t}\o_{i,1}$$
and
$$\l_{n-1}|_{H^0} = -(\xi_{n-1}+\xi_n)|_{H^0} = \o_{1,l-1}+2\sum_{i=2}^{t}\o_{i,l},
\;\; \l_n|_{H^0} = -\xi_n|_{H^0} = \sum_{i=1}^{t}\o_{i,l},$$
so
$$\l|_{H^0}=a_1\o_{1,1}+a_2\o_{1,2}+a_{n-1}\o_{1,l-1}+a_n\o_{1,l}+\sum_{i=2}^{t}\left((a_1+2a_2)\o_{i,1}+(2a_{n-1}+a_n)\o_{i,l}\right).$$

If $\mu \in \L(V)$ affords the highest weight of a $KH^0$-composition factor then, in view of \eqref{e:yb} and \eqref{e:ci0},
either $\mu=\l$ or there exists an associated permutation of $\mu$, say $\s \in S_t$, such that $\s(1)=i_0 \neq 1$ and $\mu$ restricts to
$$a_1\o_{i_0,1}+a_2\o_{i_0,2}+a_{n-1}\o_{i_0,l-1}+a_n\o_{i_0,l}
					+\sum_{i=1, i \neq i_0}^{t}\left((a_1+2a_2)\o_{i,1}+(2a_{n-1}+a_n)\o_{i,l}\right),$$
					so
\begin{equation}\label{e:mu1}
\mu|_{H^0}
=\l|_{H^0}+a_2(\b_{1,1}-\b_{i_0,1})+a_{n-1}(\b_{1,l}-\b_{i_0,l}).
\end{equation}

For now, let us assume $a_2 \neq 0$. Then $\mu = \l-\a_{2}-\a_{3}-\dots-\a_{l+1} \in \L(V)$ restricts to $\l|_{H^0}+\b_{1,1}-\b_{2,1}$, so Lemma \ref{l:sumc} implies that $\mu$ affords the highest weight of a $KH^0$-composition factor, and \eqref{e:mu1} yields $a_2=1$ and $a_{n-1}=0$.
	Consequently, if $\nu \in \L(V)$ affords the highest weight of a $KH^0$-composition factor then either $\nu=\l$ or $\nu|_{H^0}=\l|_{H^0}+\b_{1,1}-\b_{i_0,1}$ for some $i_0 \in \{2,\ldots,t\}$.
	Set $\mu_{i_0}=\l|_{H^0}+\b_{1,1}-\b_{i_0,1}$. One checks that $\l - \a_2-\cdots -\a_{(l+1)^{i_0-1}}$ is the unique weight in $\L(V)$ that restricts to $\mu_{i_0}$, so $\mu_{i_0}$ occurs with multiplicity $1$ in $V|_{H^0}$ and we conclude that $V|_{H^0}$ has exactly $t$ composition factors. 
	
	Set $\chi_1=\l- \a_1- \cdots - \a_{l+1}$ and $\chi_2=\l-\a_2 - \cdots - \a_{l+2}$.
	Then $\chi_1,\chi_2 \in \L(V)$ and $\chi_1|_{H^0}=\chi_2|_{H^0}=\l|_{H^0}-\b_{2,1}$.
	If $\mu \in \L(V)$ affords the highest weight of a $KH^0$-composition factor containing $\l|_{H^0}-\b_{2,1}$ then either $\mu=\l$ or $\mu|_{H^0}=\mu_{2}$.
	As $\l|_{H^0}-\b_{2,1}$ occurs with multiplicity at most $1$ under each of $\l|_{H^0}$ and $\mu_{2}$, we must have $m_{V}(\chi_1)=1$, so Lemma \ref{l:118}(i) implies that $a_1=0$ or $p-2$.
	However, if $a_1=p-2$ then $\langle \l|_{H^0},\b_{2,1} \rangle=p$
	and $\l|_{H^0}-\b_{2,1}$ does not occur as a weight of $L_{H^0}(\l|_{H^0})$, which is impossible. Therefore $a_1=0$ is the only possibility, so $\l=\l_2+a_{n}\l_n$ and 
	$$\l|_{H^0}=\o_{1,2}+a_n\o_{1,l}+\sum_{i=2}^t(2\o_{i,1}+a_n\o_{i,l}).$$

	Suppose $a_n\geq 2$. Then
	$\l-\sum_{i=l(l+1)^{t-1}-1}^{n}\a_i$ and
	$\l-\sum_{i=l(l+1)^{t-1}}^{n-1}\a_i-2\a_{n} \in \L(V)$,
	both of which restrict to $\l|_{H^0}-\b_{1,l}-\b_{t,l}$.
	However, this latter weight can only occur in the composition factor with highest weight $\l|_{H^0}$, where it has multiplicity $1$,
	which is absurd.
	Hence $a_n \le 1$.
Similarly, if $a_n=1$ then each of the following $l$ distinct $T$-weights
$$\begin{array}{l}
\displaystyle \l - \sum_{j=2}^{l}\a_j - \sum_{j=1}^{l}\a_{n-l+j} \\
\displaystyle \l - \sum_{j=1}^{l}\a_j - \sum_{j=2}^{l}\a_{n-l+j} \\
\displaystyle \l - \sum_{j=1}^{l}\a_j - \sum_{j=2}^{k}\a_j - \sum_{j=k+1}^{l}\a_{n-l+j}\;\;(2 \le k \le l-1)
\end{array}$$
restricts to $\l|_{H^0}-\b_{1,1}-2\sum_{j=2}^{l}\b_{1,j}$. However, this weight can only occur in the composition factor with highest weight $\l|_{H^0}$, where it has multiplicity at most $l-1$ since it is conjugate to the $T_{H^0}$-weight $\l|_{H^0}-\sum_{j=2}^{l}\b_{1,j}$. This is a contradiction. We conclude that $a_n=0$.
	
We have now established that $\l=\l_2$ is the only possibility with $a_2 \neq 0$,
and thus $\l|_{H^0}=\o_{1,2}+2\sum_{i=2}^{t}\o_{i,1}$. Now $\dim V = \frac{1}{2}n(n+1)$ and since $\dim{L_{A_m}(\l_2)}=\frac{1}{2}m(m+1)$ and
$$\dim{L_{A_l}(2\l_1)}=\left\{
\begin{array}{ll}
l+1 & \mbox{if $p=2$} \\
\frac{1}{2}(l+2)(l+1) & \mbox{otherwise} \\
\end{array}\right.$$
(see Lemma \ref{l:s114}), we deduce that $\dim L_{H^0}(\l|_{H^0})= \frac{1}{2}l(l+1)^{t}$ if $p=2$, and $\frac{1}{2^{t}}l(l+1)^{t}(l+2)^{t-1}$ otherwise.
It follows that $\dim{L_{A_n}(\l_2)}=t \cdot \dim{L_{H^0}(\l|_{H^0})}$ if and only if $t=2$ and $p \neq 2$. These examples are recorded in Table \ref{t:c4ii}. Similarly, if $a_{n-1} \ge 1$ then $\l=\l_{n-1}$ is the only possibility, and again $V|_H$ is irreducible if and only if $t=2$ and $p \neq 2$.

Finally, let us assume $\l=a_1\l_1+a_n\l_n$. Then $\l$ is the unique weight in $\L(V)$ that affords the highest weight of a $KH^0$-composition factor (see \eqref{e:mu1} and \eqref{e:ci0}), so $V|_{H^0}$ is irreducible.
Therefore \cite[Theorem 1]{Seitz2} implies that the only examples are $\l=\l_1$ and $\l=\l_n$.
\end{proof}

\begin{lem}\label{c4ii:p2}
Proposition \ref{TH:C4II} holds in case (ii) of Table \ref{t:c4iis}.
\end{lem}

\begin{proof}
Here $G=B_n$, $2n+1=(2l+1)^t$ and $H^0 = X_1 \cdots X_t$, where $X_{i} \cong B_{l}$ and $p \neq 2$. Let $\{\b_{i,1}, \ldots, \b_{i,l}\}$ be a set of simple roots for $X_{i}$, with corresponding fundamental dominant weights   $\{\o_{i,1},\ldots,\o_{i,l}\}$.
For $1 \le i \le t$ and $0 \le m \le l$ we define $\b_{i,l+m}=\b_{i,l-m}$ and $\o_{i,l+m}=\o_{i,l-m}$, with $\b_{i,0}=\o_{i,0}=0$.
Similarly we set $\l_{n+m}=\l_{n-m}$ and $\l_0=0$ for all $0 \le m \le n$.

Let $0 \le k <n$ be an integer. Then there exist unique integers $r_k(i) \in \{0,\ldots,2l\}$ such that
\begin{equation}\label{e:rik2}
k=\sum_{i=0}^{t-1}r_k(i)(2l+1)^i.
\end{equation}
For $t' \in \{1,\ldots,t\}$ let $k'=\sum_{i=0}^{t'-1}r_{k}(i)(2l+1)^i$. By choosing an appropriate embedding of $H^0$ in $G$, we may assume that
$$\langle\l_{k}|_{H^0},\b_{i,j}\rangle=\langle\l_{k'}|_{H^0},\b_{i,j}\rangle$$
for all $1 \le i \leq t'$ and $1 \le j \le l$.
Set $\xi_{k} = -\l_{k}+\l_{k+1}$.
If $k \neq n-1$ then $\xi_k$ is a weight of the natural $KG$-module $W$ and we have
$$\langle \xi_{k}|_{H^0},\b_{i+1,j}\rangle=
\left\{\begin{array}{rl}
1 & \mbox{if $0<j=r_k(i)+1<l$ or $l<j=r_k(i)<2l$} \\
-1 & \mbox{if $0<j=r_k(i)<l$ or $l<j=r_k(i)-1<2l$} \\
2 & \mbox{if $j=l=r_k(i)+1$} \\
-2 & \mbox{if $j=l=r_k(i)-1$} \\
0 & \mbox{otherwise}
\end{array}\right.$$
for all $0 \le i \le t-1$.

Now write
$n+k=\sum_{i=0}^{t-1}r_{n+k}(i)(2l+1)^i$
and
$n-1-k=\sum_{i=0}^{t-1}r_{n-1-k}(i)(2l+1)^i$.
Let $j_0\in\{0,\dots,t-1\}$ be minimal such that $r_{n+k}(j_0) \neq 2l$ (equivalently, $j_0$ is minimal such that $r_{n-1-k}(j_0) \neq 2l$).
Then
$$
r_{n-1-k}(i)=
\left\{\begin{array}{ll}
r_{n+k}(i)=2l & \mbox{ if $i<j_0$} \\
2l-1-r_{n+k}(j_0) & \mbox{ if $i=j_0$}\\
2l-r_{n+k}(i) & \mbox{ if $i>j_0$.}
\end{array}\right.$$
Now $\xi_{n+k}=-\l_{n+k}+\l_{n+k+1}=-\xi_{n-1-k}$ so, for $k \neq 0$, we have
$$\langle \xi_{n+k}|_{H^0},\b_{i,j}\rangle=
\left\{\begin{array}{ll}
1 & \mbox{if $j=1<l$} \\
2 & \mbox{if $j=1=l$} \\
0 & \mbox{otherwise}
\end{array}\right.$$
for $0<i \le j_0$,
$$\langle \xi_{n+k}|_{H^0},\b_{j_0+1,j}\rangle=
\left\{\begin{array}{rl}
1 & \mbox{if $0<j=r_{n+k}(j_0)+2<l$ or $l<j=r_{n+k}(j_0)+1<2l$} \\
-1 & \mbox{if $0<j=r_{n+k}(j_0)+1<l$ or $l<j=r_{n+k}(j_0)<2l$} \\
2 & \mbox{if $j=l=r_{n+k}(j_0)+2$} \\
-2 & \mbox{if $j=l=r_{n+k}(j_0)$} \\
0 & \mbox{otherwise,}
\end{array}\right.$$
and
$$\langle \xi_{n+k}|_{H^0},\b_{i+1,j}\rangle=
\left\{\begin{array}{rl}
1 & \mbox{if $0<j=r_{n+k}(i)+1<l$ or $l<j=r_{n+k}(i)<2l$} \\
-1 & \mbox{if $0<j=r_{n+k}(i)<l$ or $l<j=r_{n+k}(i)-1<2l$} \\
2 & \mbox{if $j=l=r_{n+k}(i)+1$} \\
-2 & \mbox{if $j=l=r_{n+k}(i)-1$} \\
0 & \mbox{otherwise,}
\end{array}\right.$$
for $j_0<i<t$.
We also observe that
the weight $-\l_{n-1}+2\l_{n}$
restricts to
$-\o_{1,l-1}+2\o_{1,l}$,
while
$-2\l_{n}+\l_{n+1}$ restricts to
$\o_{1,l-1}-2\o_{1,l}$.

As in the proof of the previous lemma, for an integer $1 \le k \le n$, let $i_k$ be minimal such that  $r_k(i_k) \neq 0$ in \eqref{e:rik2}.
Notice that 
$$k-1=r_k(0)-1+\sum_{i=1}^{t-1}r_k(i)(2l+1)^i$$
if $i_k=0$, and
$$k-1=\sum_{i=0}^{i_k-1}(2l)(2l+1)^i+(r_k(i_k)-1)(2l+1)^{i_k}+\sum_{i=i_k+1}^{t-1}r_k(i)(2l+1)^i$$
if $i_k \geq 1$.
Moreover we have $\a_k=\xi_{k-1}-\xi_{k}$ for $1 \leq k < n-1$, $\a_{n-1}=\xi_{n-2}+\l_{n-1}-2\l_{n}$ and $\a_n=-\l_{n-1}+2\l_{n}$.
In view of the above restrictions, we deduce that
\begin{equation}\label{e:akh2}
\a_{k}|_{H^0} = \left\{\begin{array}{ll}
\b_{1,r_k(0)} & \mbox{if $i_k=0$ and $r_k(0) \leq l$}\\
\b_{1,r_k(0)-1} & \mbox{if $i_k=0$ and $r_k(0)>l$}\\
\b_{i_k+1,r_k(i_k)}-2\displaystyle \sum_{i=1}^{i_k}\b(i) & \mbox{if $i_{k} \ge 1$ and $r_k(i_k) \leq l$}\\	
\b_{i_k+1,r_k(i_k)-1}-2\displaystyle \sum_{i=1}^{i_k}\b(i) & \mbox{if $i_{k} \ge 1$ and $r_k(i_k)>l$,}\\	
\end{array}\right.
\end{equation}
where
\begin{equation}\label{e:betai}
\b(i)=\b_{i,1}+\dots+\b_{i,l}
\end{equation}
is the highest short root in $\Sigma(X_i)$. In particular, we deduce that \eqref{e:ci0} holds, where $\a=\sum_{i}c_i\a_i$ with $c_i \in \mathbb{N}_0$.

Recall that $V$ has highest weight $\l=\sum_{i=1}^na_i\l_i$ and let us assume $V|_{H}$ is irreducible.
Our first task is to deal with the case $n=4$, so $(l,t)=(1,2)$ and 
\begin{equation}\label{e1:lh04} \l|_{H^0}=(2a_1+2a_2+a_4)\o_{1,1}+(2a_1+4a_2+6a_3+3a_4)\o_{2,1}.
\end{equation}
By \eqref{e:ci0}, $\l$ is the unique weight of $V$ that restricts to $\l|_{H^0}$.
In particular, if $\mu \in \L(V)$ affords the highest weight of a $KH^0$-composition factor  then either $\mu=\l$, or
\begin{equation}\label{e:muh0B04}
\mu|_{H^0}=\l|_{H^0}+(a_2+3a_3+a_4)(\b_{1,1}-\b_{2,1}).
\end{equation}
	If $a_3 \neq 0$ then $\nu=\l-\a_3$ is a weight of $V$ and  $\nu|_{H^0}=\l|_{H^0}+2\b_{1,1}-\b_{2,1}$ so $\ell(\nu|_{H^0})=1$, which cannot happen in view of Lemma \ref{l:sumc}. Therefore  $a_3=0$. Now $\l|_{H^0}-\b_{1,1}$ occurs with multiplicity at most $1$ in $V|_{H^0}$, and $\a_1,\a_2$ and $\a_4$ all restrict to $\b_{1,1}$ (see \eqref{e:akh2}), so we deduce that $\l=a_i\l_i$ with $i=1,2$ or $4$. 
	
	If $\l=a_1\l_1$ then $\l$ is the unique weight in $\L(V)$ that affords the highest weight of a $KH^0$-composition factor (see \eqref{e:muh0B04}),  so $V|_{H^0}$ is irreducible.
	Therefore, using \cite{Seitz2}, we conclude that the only example is $\l=\l_1$.
	If $\l=a_2\l_2$ then $\l-\a_2-\a_3\in\L(V)$ restricts to $\l|_{H^0}+\b_{1,1}-\b_{2,1}$, so it affords the highest weight of a $KH^0$-composition factor (see Lemma \ref{l:sumc}(i)).
	Now \eqref{e:muh0B04} implies that $a_2=1$, so $\l=\l_2$ and  $\l|_{H^0}=2\o_{1,1}+4\o_{2,1}$.
	Since $t=2$, 
	$V|_{H^0}$ has exactly two composition factors, and $\l|_{H^0}-\b_{1,1}-\b_{2,1}$ occurs with multiplicity at most $2$ in $V|_{H^0}$.
	However, the weights 
	$\l-\a_1-\a_2-\a_3-\a_4$,
	$\l-\a_1-2\a_2-\a_3$
	and
	$\l-\a_2-\a_3-2\a_4$ in $\L(V)$
	all restrict to $\l|_{H^0}-\b_{1,1}-\b_{2,1}$.
	This contradiction eliminates the case $\l=a_2\l_2$.
	Similarly, if $\l=a_4\l_4$ then
	$a_4=1$ is the only possibility, and $V|_{H^0}$ has exactly two composition factors of highest weights $\l|_{H^0}$ and $\l|_{H^0}+\b_{1,1}-\b_{2,1}$ respectively.
	So $\l=\l_4$, $\l|_{H^0}=\o_{1,1}+3\o_{2,1}$ and $\dim V = 16$ (see Lemma \ref{l:dmspin}). Therefore, arguing by dimension, we deduce that $V|_H$ is irreducible if and only if $p \neq 3$. This case is recorded in Table \ref{t:c4ii}.

For the remainder we may assume $n>4$.
Suppose $a_k \neq 0$ for some integer $k$ in the range $1 \le k \le n$, so $\mu = \l - \a_k \in \L(V)$.
Define the integers $r_k(i)$ as in \eqref{e:rik2} and write $k=r_k(0)+(a-1)(2l+1)$, where
$a \in \{1,\dots,((2l+1)^{t-1}+1)/2\}$.
Recall that $i_k$ is minimal such that $r_k(i_k) \neq 0$.
If $i_k \ge 1$ then \eqref{e:akh2} implies that $\ell_1(\mu|_{H^0})\ge 2i_k-1>0$, so this case is ruled out by Lemma \ref{l:sumc}.
Therefore $i_k =0$ and thus $r_k(0) \neq 0$.

For now let us assume $l=1$ and $a_k \neq 0$. By the above remarks we have $i_k=0$, so $r_k(0)=1$ or $2$. First suppose $r_k(0)=1$.
	If $k=n$ then $\mu=\l-\sum_{i=n-4}^{n}\a_i \in \L(V)$ and
		$\mu|_{H^0}=\l|_{H^0}+\b_{1,1}+\b_{2,1}-\b_{3,1}$, so $\ell(\mu|_{H^0})=1$ and this contradicts Lemma \ref{l:sumc}. Therefore $k \le n-3$. 
	If $k \neq 1$ then
		$\mu_1=\l-\a_{k}-\a_{k-1}$
		and
		$\mu_2=\l-\a_{k}-\a_{k+1}-\a_{k+2}$
		are weights of $V$,
		and \eqref{e:akh2} implies that
$$\mu_1|_{H^0} =  \l|_{H^0}-\b_{1,1}-\b_{i_{k-1}+1,1}+2\sum_{i=1}^{i_{k-1}}\b_{i,1}$$
and
$$\mu_2|_{H^0} = \l|_{H^0}-2\b_{1,1}-\b_{i_{k+2}+1,1}+2\sum_{i=1}^{i_{k+2}}\b_{i,1},$$
		where $i_{k-1},i_{k+2} \geq 1$.
		Now Lemma \ref{l:sumc} yields $i_{k-1}=i_{k+2}=1$, so $k \equiv 4 \imod{9}$ and $k \le n-9$. Then $\nu=\l-\sum_{i=k-1}^{k+5}\a_i \in \L(V)$ and
		$$\nu|_{H^0} =  \l|_{H^0}-2\b_{2,1}-\b_{i_{k+5}+1,1}+2\sum_{i=1}^{i_{k+5}}\b_{i,1},$$
		where $i_{k+5} \geq 2$,
		so $\ell(\nu|_{H^0})=2i_{k+5}-3\geq 1$, which is impossible by Lemma \ref{l:sumc}.
		Therefore $k=1$.
		
Now assume $r_k(0)=2$, so $k \le n-2$ since $n \equiv 1 \imod{3}$.
	If $k \neq 2$ then
		$\mu_1=\l-\a_{k}-\a_{k-1}-\a_{k-2}$
		and
		$\mu_2=\l-\a_{k}-\a_{k+1}$
		are weights of $V$,
		and \eqref{e:akh2} implies that
$$\mu_1|_{H^0} =  \l|_{H^0}-2\b_{1,1}-\b_{i_{k-2}+1,1}+2\sum_{i=1}^{i_{k-2}}\b_{i,1}$$
and
$$\mu_2|_{H^0} = \l|_{H^0}-\b_{1,1}-\b_{i_{k+1}+1,1}+2\sum_{i=1}^{i_{k+1}}\b_{i,1},$$
		where $i_{k-2},i_{k+1} \geq 1$.
		Here Lemma \ref{l:sumc} implies that $i_{k-2}=i_{k+1}=1$, so $k \equiv 5\imod{9}$ and $k \le n-8$. Then $\nu=\l-\sum_{i=k}^{k+4}\a_i \in \L(V)$ and
		$$\nu|_{H^0} =  \l|_{H^0}-\b_{1,1}-\b_{2,1}-\b_{i_{k+4}+1,1}+2\sum_{i=1}^{i_{k+4}}\b_{i,1},$$
		where $i_{k+4} \geq 2$,
		so $\ell(\nu|_{H^0})=2i_{k+4}-3\geq 1$, which contradicts Lemma \ref{l:sumc}.
		Therefore $k=2$. Note that for $l=1$ we have now reduced to the case $\l=a_1\l_1+a_2\l_2$.

Now assume $l \ge 2$ and $a_k \neq 0$. If $k \geq n-l$ then $r_k(0) \le l$,
$\nu=\l-\sum_{i=n-l}^{k}\a_{i} \in \L(V)$ and $\ell(\nu|_{H^0})=2l-r_k(0)-1 \ge 1$, so Lemma \ref{l:sumc} implies that $a_k=0$. Now suppose $k<n-l$, so $\mu=\l-\sum_{i=k}^{(2l+1)a}\a_{i} \in \L(V)$. Since 
$$(2l+1)a=(r_k(1)+1)(2l+1)+\sum_{i=2}^{t-1}r_k(i)(2l+1)^{i}$$ 
we deduce that
$$\mu|_{H^0} = 
\left\{\begin{array}{ll}
\displaystyle\l|_{H^0}-\sum_{j=r_k(0)}^{2l-1}\b_{1,j}-\b_{1,l}-\b_{i_{2la+a}+1,r}+2\sum_{i=1}^{i_{2la+a}}\b(i) & \mbox{if $r_k(0) \leq l$, $r \leq l$}\\
\displaystyle\l|_{H^0}-\sum_{j=r_k(0)-1}^{2l-1}\b_{1,j}-\b_{i_{2la+a}+1,r}+2\sum_{i=1}^{i_{2la+a}}\b(i) & \mbox{if $r_k(0)>l$, $r \leq l$}\\
\displaystyle\l|_{H^0}-\sum_{j=r_k(0)}^{2l-1}\b_{1,j}-\b_{1,l}-\b_{i_{2la+a}+1,r-1}+2\sum_{i=1}^{i_{2la+a}}\b(i) & \mbox{if $r_k(0) \leq l$, $r>l$}\\
\displaystyle\l|_{H^0}-\sum_{j=r_k(0)-1}^{2l-1}\b_{1,j}-\b_{i_{2la+a}+1,r-1}+2\sum_{i=1}^{i_{2la+a}}\b(i) & \mbox{if $r_k(0)>l$, $r>l$,}\\
\end{array}\right.$$
where $r=r_{2la+a}(i_{2la+a})$ and $\b(i)$ is defined as in \eqref{e:betai}.
As $\ell(\mu|_{H^0})=2l(i_{(2l+1)a}-1)+r_k(0)-2$, we must have $i_{(2l+1)a}=1$ and $r_k(0)\leq 2$.

Suppose $r_k(0)=1$. If $k\neq 1$ then $i_{k-1} \geq 1$,
			$\nu_1=\l-\a_{k-1}-\a_{k} \in \L(V)$
			and $\ell(\nu_1|_{H^0})=2li_{k-1}-2 \ge 2$, which is impossible by Lemma \ref{l:sumc}.
		 	Similarly, if $r_k(0)=2$ and $k \neq 2$ then $i_{k-2} \geq 1$,
			$\nu_2=\l-\a_{k-2}-\a_{k-1}-\a_{k} \in \L(V)$
			and $\ell(\nu_2|_{H^0})=2li_{k-2}-3\geq 1$, which again contradicts Lemma \ref{l:sumc}.
			Therefore, for the remainder of the proof, we may assume $l \ge 1$, $n>4$ and $\l=a_1\l_1+a_2\l_2$,
			so
$$\l|_{H^0} = \left\{\begin{array}{ll}
\displaystyle 2(a_1+a_2)\o_{1,1}+2\sum_{i=2}^{t}(a_1+2a_2)\o_{i,1} & \mbox{if $l=1$}\\
\displaystyle a_1\o_{1,1}+2a_2\o_{1,2}+\sum_{i=2}^{t}(a_1+2a_2)\o_{i,1} & \mbox{if $l=2$}\\
\displaystyle a_1\o_{1,1}+a_2\o_{1,2}+\sum_{i=2}^{t}(a_1+2a_2)\o_{i,1} & \mbox{if $l \ge 3$.}\\
\end{array}\right.$$

In particular, if $\mu \in \L(V)$ affords the highest weight of a $KH^0$-composition factor then either $\mu=\l$, or there exists an associated permutation of $\mu$, say $\s \in S_t$, such that $\s(1)=i_0 \neq 1$ and
\begin{equation}\label{e:muh0}
\mu|_{H^0}=\l|_{H^0}+a_2(\b_{1,1}-\b_{i_0,1}).
\end{equation}

Suppose $a_2 \neq 0$. Then $\l-\a_{2}-\a_{3}-\dots-\a_{2l+1} \in \L(V)$ restricts to $\l|_{H^0}+\b_{1,1}-\b_{2,1}$, so it affords the highest weight of a $KH^0$-composition factor (by Lemma \ref{l:sumc}).
	Therefore $a_2=1$ (see \eqref{e:muh0}). Moreover, it follows that if $\mu \in \L(V)$ affords the highest weight of a $KH^0$-composition factor then either $\mu=\l$, or $\mu|_{H^0}=\l|_{H^0}+\b_{1,1}-\b_{i_0,1}$ for some $i_0 \in \{2,\dots,t\}$. By \eqref{e:ci0}, $\l$ is the unique weight with restriction $\l|_{H^0}$, so we see that $V|_{H^0}$ has exactly $t$ composition factors.
		
	Set $\chi_1=\l- \a_1 - \cdots - \a_{2l+1}$ and $\chi_2=\l-\a_2 - \cdots - \a_{2l+2}$.
	Then $\chi_1,\chi_2 \in \L(V)$ and $\chi_1|_{H^0}=\chi_2|_{H^0}=\l|_{H^0}-\b_{2,1}$.
	If $\mu \in \L(V)$ affords the highest weight of a $KH^0$-composition factor containing $\l|_{H^0}-\b_{2,1}$ then either $\mu=\l$ or $\mu|_{H^0}=\l|_{H^0}+\b_{1,1}-\b_{2,1}=: \nu$.
	As $\l|_{H^0}-\b_{2,1}$ occurs with multiplicity $1$ under both $\l|_{H^0}$ and $\nu$, it follows that $m_{V}(\chi_1)=1$ so Lemma \ref{l:118}(i) implies that $a_1 \in \{0,p-2\}$.
	If $a_1=p-2$ then $\langle\l|_{H^0},\b_{2,1}\rangle \in \{p,2p\}$ and $\l|_{H^0}-\b_{2,1}$ does not occur as a weight of $L_{H^0}(\l|_{H^0})$, which is 
a contradiction. Therefore $a_1=0$ and thus $\l=\l_2$, so 
$$\l|_{H^0} = \left\{\begin{array}{ll}
\displaystyle 2\o_{1,1}+4\sum_{i=2}^{t}\o_{i,1} & \mbox{if $l=1$}\\
\displaystyle 2\o_{1,2}+2\sum_{i=2}^{t}\o_{i,1} & \mbox{if $l=2$}\\
\displaystyle \o_{1,2}+2\sum_{i=2}^{t}\o_{i,1} & \mbox{if $l \ge 3$.}\\
\end{array}\right.$$
	By Proposition \ref{p:dims}, since $p \neq 2$, we have $\dim{L_{B_2}(2\l_2)}=10$,
	$$\dim{L_{B_l}(2\l_1)}=
		\left\{\begin{array}{ll}
		3& \mbox{if $l=1$} \\
		l(2l+3)-1& \mbox{if $p\mid 2l+1$} \\
        	l(2l+3)	 & \mbox{otherwise}
		\end{array}\right.$$
	and
	$$\dim{L_{B_m}(\l_2)}=
		\left\{\begin{array}{ll}
		4 & \mbox{if $m=2$} \\
		m(2m+1) & \mbox{if $m>2$.}
		\end{array}\right.$$
	Therefore $\dim V = n(2n+1)$ and $\dim L_{H^0}(\l|_{H^0}) \le l^{t}(2l+1)(2l+3)^{t-1}$.	
	It is easy to check that $n(2n+1)>t \cdot l^{t}(2l+1)(2l+3)^{t-1}$ for all $t \ge 2$ and $l \ge 1$, whence $V|_{H}$ is reducible.

Finally, let us assume $\l=a_1\l_1$. Here $\l$ is the unique weight in $\L(V)$ that affords the highest weight of a $KH^0$-composition factor (see \eqref{e:muh0}),  so $V|_{H^0}$ is irreducible.
Therefore, by using \cite{Seitz2}, we see that the only example is $\l=\l_1$.
\end{proof}

\begin{lem}\label{c4ii:p3}
Proposition \ref{TH:C4II} holds in case (iii) of Table \ref{t:c4iis}.
\end{lem}

\begin{proof}
Here $G=C_n$, $2n=(2l)^t$ and $H^0 = X_1 \cdots X_t$ where $X_{i} \cong C_{l}$, $t \ge 3$ is odd and $p \neq 2$. As before, let $\{\b_{i,1}, \ldots, \b_{i,l}\}$ be a set of simple roots for the factor $X_{i}$, with corresponding fundamental dominant weights $\{\o_{i,1},\ldots,\o_{i,l}\}$.
In addition, for $1 \le i \le t$ and $0 \le m \le l$ we define $\b_{i,l+m}=\b_{i,l-m}$ and $\o_{i,l+m}=\o_{i,l-m}$, where $\b_{i,0}=\o_{i,0}=0$.
Then for $1 \le j \le 2l-1$ we have $\b_{i,j}=-\o_{i,j-1}+2\o_{i,j}-\o_{i,j+1}$.
Similarly, for all $0 \le m \le n$ we set $\l_{n+m}=\l_{n-m}$ and $\l_0=0$.

Let $0 \le k < n$ be an integer. Then there exist unique integers $r_k(i) \in \{0,\ldots,2l-1\}$ such that
\begin{equation}\label{e:rik3}
k=\sum_{i=0}^{t-1}r_k(i)(2l)^i.
\end{equation}
Note that $r_k(t-1) \le l-1$ since $n=l(2l)^{t-1}$.
For $t' \in \{1,\ldots,t\}$ let $k'=\sum_{i=0}^{t'-1}r_{k}(i)(2l)^i$. By choosing an appropriate embedding of $H^0$ in $G$, we may assume that
$$\langle\l_{k}|_{H^0},\b_{i,j}\rangle=\langle\l_{k'}|_{H^0},\b_{i,j}\rangle$$
for all $1 \le i \leq t'$ and $1 \le j \le l$.
Let $\xi_k$ be the weight $-\l_{k}+\l_{k+1}$ of the natural $KG$-module $W$. Now
\begin{equation}\label{e:n1k}
n-k-1=\sum_{i=0}^{t-2}(2l-1-r_k(i))(2l)^i+(l-1-r_k(t-1))(2l)^{t-1},
\end{equation}
$$n+k = \sum_{i=0}^{t-2}r_k(i)(2l)^i+(l+r_k(t-1))(2l)^{t-1}$$
and $\xi_{n+k}=-\l_{n+k}+\l_{n+k+1}=-\xi_{n-k-1}$, so
$$\xi_{n+k}|_{H^0} = \sum_{i=0}^{t-2}\left(-\o_{i+1,r_k(i)}+\o_{i+1,r_k(i)+1}\right)-\o_{t,l+r_k(t-1)}+\o_{t,l+1+r_k(t-1)}$$
and we see that \eqref{e:xi} holds for any integer $0 \le k <2n$, where the $r_k(i)$ are the unique integers in  \eqref{e:rik3}.

As before, for an integer $1 \le k \le n$, let $i_k$ be minimal such that $r_k(i_k) \neq 0$ in \eqref{e:rik3}.
If $i_k=0$ then $k-1=r_k(0)-1+\sum_{i=1}^{t-1}r_k(i)(2l)^i$ and using \eqref{e:xi} we deduce that $-\l_{k-1}+2\l_{k}-\l_{k+1}$ restricts to $-\o_{1,r_k(0)-1}+2\o_{1,r_k(0)}-\o_{1,r_k(0)+1}$.
Now assume $i_k \geq 1$. Here
$$k-1=\sum_{i=0}^{i_k-1}(2l-1)(2l)^i+(r_k(i_k)-1)(2l)^{i_k}+\sum_{i=i_k+1}^{t-1}r_k(i)(2l)^i$$ and
we find that the weight $-\l_{k-1}+2\l_{k}-\l_{k+1}$ restricts to
\begin{align*}
& -\sum_{i=1}^{i_k}(\o_{i,1}+\o_{i,2l-1})-\o_{i_k+1,r_k(i_k)-1}+2\o_{i_k+1,r_k(i_k)}-\o_{i_k+1,r_k(i_k)+1} \\
= & \; \b_{i_k+1,r_k(i_k)}-2\sum_{i=1}^{i_k}\o_{i,1}.
\end{align*}
It follows that
\begin{equation}\label{e:akh3}
\a_{k}|_{H^0}=
	\left\{
	\begin{array}{ll}
	\displaystyle \b_{i_k+1,1} -\sum_{i=1}^{i_k}\b_{i,1} & \mbox{if $i_k>0$ and $l=1$}\\
	\displaystyle \b_{i_k+1,r_k(i_k)}-\sum_{i=1}^{i_k}\b(i) &  \mbox{if $i_k>0$ and $l \ge 2$} \\
	\b_{1,r_k(0)} & \mbox{if $i_k=0$,}
	\end{array}
	\right.
	\end{equation}
	where 
	\begin{equation}\label{e:betai2}
	\b(i) = 2\b_{i,1}+\cdots+2\b_{i,l-1}+\b_{i,l}
	\end{equation}
	is the highest long root in $\Sigma(X_i)$.
In particular, we deduce that \eqref{e:ci0} holds, where $\a=\sum_{i}c_i\a_i$ with $c_i \in \mathbb{N}_0$.

Let $\l=\sum_{i=1}^na_i\l_i$ be the highest weight of $V$ and suppose $V|_{H}$ is irreducible.
For now, let us assume $l=1$ and suppose $a_k \neq 0$. Define the integers $r_k(i)$ as in \eqref{e:rik3}, and recall that $i_k$ is minimal such that $r_k(i_k) \neq 0$.

Suppose $k$ is even, so $i_k>0$. Here $\mu = \l-\a_{k} \in \L(V)$ and \eqref{e:akh3} implies that
$$\mu|_{H^0} = \l|_{H^0}-\b_{i_k+1,1}+\sum_{i=1}^{i_k}\b_{i,1},$$
so Lemma \ref{l:sumc} yields $i_k=1$ (since $\ell(\mu|_{H^0})=i_k-1$) and thus $k \equiv 2 \imod{4}$.
	If $k \neq 2$, then
		$\mu_1=\l-\a_{k}-\a_{k-1}-\a_{k-2}$
		and
		$\mu_2=\l-\a_{k}-\a_{k+1}-\a_{k+2}$
		are weights of $V$ (note that $k \le n-2$ since $n=2^{t-1}$ is divisible by $4$),
		and \eqref{e:akh3} implies that
$$\mu_1|_{H^0} =  \l|_{H^0}-\b_{2,1}-\b_{i_{k-2}+1,1}+\sum_{i=1}^{i_{k-2}}\b_{i,1}$$
and
$$\mu_2|_{H^0} = \l|_{H^0}-\b_{2,1}-\b_{i_{k+2}+1,1}+\sum_{i=1}^{i_{k+2}}\b_{i,1},$$
where $i_{k-2},i_{k+2} \geq 2$ (since $k-2$ and $k+2$ are divisible by $4$). In view of  Lemma \ref{l:sumc}, we deduce that $i_{k-2}=i_{k+2}=2$, but this is a contradiction since
$k-2$ and $k+2$ are not both congruent to $4$ modulo $8$. We conclude that $k=2$ is the only possibility.

Next suppose $k$ is odd, so $i_k=0$ and \eqref{e:akh3} implies that $\l-\a_{k} \in \L(V)$ restricts to $\l|_{H^0}-\b_{1,1}$.
If $\mu \in \L(V)$ affords the highest weight of a $KH^0$-composition factor in which $\l|_{H^0}-\b_{1,1}$ occurs, then by Lemma \ref{l:sumc}{(ii)}, $\mu|_{H^0}=\l|_{H^0}+\b_{i,1}-\b_{1,1}$ for some $i \in \{1,\dots,t\}$.
 One checks that there is no weight of $V$ that restricts to $\l|_{H^0}+\b_{i,1}-\b_{1,1}$ with $i>1$, hence $\l|_{H^0}-\b_{1,1}$ occurs with multiplicity $1$ in $V|_{H^0}$, so for 
$l=1$ we have reduced to the case 
	$\l=a_2\l_2+a_k\l_k$ with $k$ odd. 
	
	The next step is to reduce to the case $k=1$. Let us assume that $a_k \neq 0$ with $k \ge 3$ odd, and continue to assume that $l=1$. Here
		$\nu_1=\l-\a_{k-1}-\a_{k}$
		and
		$\nu_2=\l-\a_{k}-\a_{k+1}$ are weights of $V$ and we have
		$$\nu_1|_{H^0} = \l|_{H^0}-\b_{1,1}-\b_{i_{k-1}+1,1}+\sum_{i=1}^{i_{k-1}}\b_{i,1}$$
		and
		$$\nu_2|_{H^0} =  \l|_{H^0}-\b_{1,1}-\b_{i_{k+1}+1,1}+\sum_{i=1}^{i_{k+1}}\b_{i,1},$$
		where $i_{k-1},i_{k+1} \ge 1$ and $|i_{k-1}-i_{k+1}| \ge 1$ since $k+1$ and $k-1$ are both even, but only one is divisible by $4$. Now Lemma \ref{l:sumc} implies that $i_{k-1}.i_{k+1}=2$,
		so $k=3+\sum_{i=3}^{t-1}r_k(i)2^i$ or $5+\sum_{i=3}^{t-1}r_k(i)2^i$.
		It follows that $\mu=\l-\a_{k-1}-\a_{k}-\a_{k+1} \in \L(V)$
		restricts to $\l|_{H^0}+\b_{1,1}-\b_{3,1}$,
		so $\mu$ affords the highest weight of a $KH^0$-composition factor (see Lemma \ref{l:sumc})
		and thus $\langle\mu|_{H^0},\b_{1,1}\rangle=\langle\l|_{H^0},\b_{3,1}\rangle$,
		so $\la \l|_{H^0},\b_{1,1}\ra +2=\la \l|_{H^0},\b_{3,1}\ra$.
		Write $\l|_{H^0}=\sum_{i=1}^{t}{a_{i,1}\o_{i,1}}$ and let us calculate $a_{1,1}$ and $a_{3,1}$.
		
		Now $k \ge 3$ and $\l_k=k\l_1-\sum_{i=1}^{k-1}(k-i)\a_i$
		so
		$$\l=a_2\l_2+a_k\l_k=(2a_2+ka_k)\l_1-(a_2+(k-1)a_k)\a_1-a_k\sum_{i=2}^{k-1}(k-i)\a_i.$$
		Recall that $\l_1|_{H^0} = \sum_{i=1}^{t}\o_{i,1} = \frac12\sum_{i=1}^{t}\b_{i,1}$.
		Since $a_{i,1}$ is the coefficient of $\frac12\b_{i,1}$ in $\l|_{H^0}$, using \eqref{e:akh3} we get
		$$a_{1,1}=2a_2+ka_k-2(a_2+(k-1)a_k)+2a_k\sum_{i=2}^{k-1}(-1)^{i}(k-i)=a_k$$
		and
		$$a_{3,1}=2a_2+ka_k+2a_k\sum_{i=1}^{c(k)}(-1)^{i}(k-4i)=2a_2+3a_k,$$
		where
		$$c(k)= \left\{\begin{array}{ll}
		(k-3)/4 & \mbox{if $k \equiv 3 \imod{8}$} \\
		(k-1)/4 & \mbox{if $k \equiv 5 \imod{8}$.}
		\end{array}\right.$$
		Therefore $a_k+2=2a_2+3a_k$ and thus $\l=\l_k$ is the only possibility. 
				
		Suppose $\l=\l_k$, where $k \ge 3$ is odd. If $t=3$ then $k=3=n-1$, $\dim V = 48-8\delta_{3,p}$ (see \cite[Table A.33]{Lubeck}) and $\l|_{H^0} = \l_3|_{H^0} = \o_{1,1}+\o_{2,1}+3\o_{3,1}$. Therefore, arguing by dimension, we deduce that $V|_H$ is irreducible if and only if $p \neq 3$. This case is recorded in Table \ref{t:c4ii}.
		
		Now assume $t>3$ and $\l=\l_k$ (with $l=1$ and $k \ge 3$ odd), so $n \ge 8$ and $k \le 5+\sum_{i=3}^{t-2}2^i = n-3$. Set
		$$\nu_1=\l-\a_{k-1}-\a_{k}-\a_{k+1},\;\; \nu_2=\l-(i_{k-1}-1)\a_{k-1}-\a_{k}-(i_{k+1}-1)\a_{k+1}$$ 
		and $\nu_3=\l$. Note that $\nu_i$ is the unique weight of $V$ such that $\nu_i|_{H^0}=\l|_{H^0}+\b_{i,1}-\b_{3,1}$.
		Also observe that there is no weight of $V$ that restricts to $\l|_{H^0}+\b_{i,1}-\b_{3,1}$ with $i>3$. 
		In particular, if $\nu \in \L(V)$ affords the highest weight of a $KH^0$-composition factor containing $\l|_{H^0}-\b_{3,1}$ then $\nu|_{H^0}=\nu_i|_{H^0}$ for $i=1,2$ or $3$, and $\l|_{H^0}-\b_{3,1}$ occurs with multiplicity at most $1$ under each $\nu_i|_{H^0}$, so this weight has multiplicity at most $3$ in $V|_{H^0}$. However, it is easy to see that the following four weights of $V$
		$$\l-\a_{k-1}-2\a_{k}-\a_{k+1}, \;\; \l-\a_{k-2}-\a_{k-1}-\a_{k}-\a_{k+1},\;\; \l-\a_{k-1}-\a_{k}-\a_{k+1}-\a_{k+2},$$
		 $$(i_{k-1}-1)\left(\l-\a_{k-3}-\a_{k-2}-\a_{k-1}-\a_{k}\right)+(i_{k+1}-1)\left(\l-\a_{k}-\a_{k+1}-\a_{k+2}-\a_{k+3}\right)$$
all restrict to $\l|_{H^0}-\b_{3,1}$, so this is a contradiction. We conclude that $a_k=0$ for all $k \ge 3$ odd, so for $l=1$ we have reduced to the case $\l=a_1\l_1+a_2\l_2$.

We can also reduce to the case $\l=a_1\l_1+a_2\l_2$ when $l \ge 2$. To see this, suppose $l \ge 2$ and $a_k \neq 0$ for some $k$.
Define the integers $r_k(i)$ as in \eqref{e:rik3} and write $k=r_k(0)+2l(a-1)$, where
$a \in \{1,\dots,l(2l)^{t-2}+1\}$.
		Note that $\mu = \l-\a_{k} \in \L(V)$.
		If $r_k(0)=0$ then
		$$\mu|_{H^0} = \l|_{H^0}-\b_{i_k+1,r_k(i_k)}+\sum_{i=1}^{i_k}\b(i),$$
		where $\b(i)$ is defined in \eqref{e:betai2}. However, $\ell(\mu|_{H^0})=(2l-1)i_k-1$ and $l \ge 2$, so this contradicts Lemma \ref{l:sumc}. Therefore $r_k(0) \neq 0$. 
		Let $\nu = \l-\a_{k}-\a_{k+1}-\dots-\a_{2la} \in \L(V)$ and note that
			\begin{equation}\label{e:nu09}
			\nu|_{H^0} =
			 \l|_{H^0}-\b_{1,r_k(0)}-\dots-\b_{1,2l-1}-\b_{i_{2la}+1,r_{2la}(i_{2la})}+\sum_{i=1}^{i_{2la}}\b(i).
			 \end{equation}
			Therefore $\ell(\nu|_{H^0})=(2l-1)(i_{2la}-1)+r_k(0)-2$, so Lemma \ref{l:sumc} implies that $i_{2la}=1$ and $r_k(0) \le 2$. In particular, since
			$2la=(r_k(1)+1)2l+\sum_{i=2}^{t-1}r_k(i)(2l)^{i}$ and $i_{2la}=1$, we deduce that $r_k(1) \leq 2l-2$.
			
			First assume $r_k(0)=1$. If $k \neq 1$
			then $\mu = \l-\a_{k}-\a_{k-1}=\l-\a_{2l(a-1)+1}-\a_{2l(a-1)} \in \L(V)$ and
			$$\mu|_{H^0} = \l|_{H^0}-\b_{1,1}-\b_{i_{k-1}+1,r_{k-1}(i_{k-1})}+\sum_{i=1}^{i_{k-1}}\b(i),$$
			where $i_{k-1} \geq 1$. This is ruled out by Lemma \ref{l:sumc}, so $k=1$ is the only possibility.
			
			Now assume $r_k(0)=2$. By considering the restriction of the weight $\nu = \l-\a_{k}-\a_{k+1}-\dots-\a_{2la}$ as given in \eqref{e:nu09}, and using Lemma \ref{l:sumc}, we deduce that $r_k(1)=0$ so $k=2+\sum_{i=2}^{t-1}r_k(i)(2l)^i$. If $k \neq 2$ then
			$\mu = \l-\a_{k}-\a_{k-1}-\a_{k-2} \in \L(V)$ and
			$$\mu|_{H^0}= \l|_{H^0}-\b_{1,1}-\b_{1,2}-\b_{i_{k-2}+1,r_{k-2}(i_{k-2})}+\sum_{i=1}^{i_{k-2}}\b(i),$$
			where $i_{k-2} \geq 2$. Again, this contradicts Lemma \ref{l:sumc}, so $k=2$.
			
			Therefore, for the remainder of the proof, we may assume that $l \ge 1$ and $\l=a_1\l_1+a_2\l_2$,
		so
		 $$\l|_{H^0}=a_1\o_{1,1}+a_2\o_{1,2}+\sum_{i=2}^{t}(a_1+2a_2)\o_{i,1}.$$
		
If $a_2=0$ then $\l=a_1\l_1$ and $\l$ is the unique weight in $\L(V)$ that affords the highest weight of a $KH^0$-composition factor (see \eqref{e:muh0}),  so $V|_{H^0}$ is irreducible. Therefore, by inspecting \cite{Seitz2}, we see that $\l=\l_1$ provides the only  example.

Finally, let us assume $a_2 \neq 0$. First note that if $\mu \in \L(V)$ affords the highest weight of a $KH^0$-composition factor then either $\mu=\l$, or there exists a permutation $\s \in S_t$ such that $\s(1)=i_0 \neq 1$ and \eqref{e:muh0} again holds.

The weight $\l-\a_{2}-\a_{3}-\dots-\a_{2l} \in \L(V)$ restricts to $\l|_{H^0}+\b_{1,1}-\b_{2,1}$, so it affords the highest weight of a $KH^0$-composition factor.
	Therefore $a_2=1$ (see \eqref{e:muh0}). Moreover, it follows that if $\mu \in \L(V)$ affords the highest weight of a $KH^0$-composition factor then either $\mu=\l$, or $\mu|_{H^0}=\l|_{H^0}+\b_{1,1}-\b_{i_0,1}$ for some $i_0 \in \{2,\dots,t\}$. By \eqref{e:ci0}, $\l$ is the unique weight with restriction $\l|_{H^0}$, so $V|_{H^0}$ has exactly $t$ composition factors.
	
	Set $\chi_1=\l- \a_1 - \cdots - \a_{2l}$ and $\chi_2=\l-\a_2 - \cdots - \a_{2l+1}$.
	Then $\chi_1,\chi_2 \in \L(V)$ and $\chi_1|_{H^0}=\chi_2|_{H^0}=\l|_{H^0}-\b_{2,1}$.
	If $\mu \in \L(V)$ affords the highest weight of a $KH^0$-composition factor containing $\l|_{H^0}-\b_{2,1}$ then either $\mu=\l$ or $\mu|_{H^0}=\l|_{H^0}+\b_{1,1}-\b_{2,1} =: \nu$.
	As $\l|_{H^0}-\b_{2,1}$ occurs with multiplicity $1$ under both $\l|_{H^0}$ and $\nu$, it follows that $m_{V}(\chi_1)=1$ and thus Lemma \ref{l:118}(i) gives $a_1 \in \{0,p-2\}$.
	If $a_1=p-2$ then $\langle\l|_{H^0},\b_{2,1}\rangle=p$ and $\l|_{H^0}-\b_{2,1}$ does not occur as a weight of $L_{H^0}(\l|_{H^0})$, which is absurd. Therefore $a_1=0$ and thus $\l=\l_2$. In particular,
	$\l|_{H^0}=2\sum_{i=2}^{t}\o_{i,1}$ if $l=1$, otherwise $\l|_{H^0}=\o_{1,2}+2\sum_{i=2}^{t}\o_{i,1}$.
	
	By Proposition \ref{p:dims}, since $p\neq 2$, we have
	$\dim{L_{C_l}(2\l_1)}=l(2l+1)$
	and
	$$\dim{L_{C_m}(\l_2)}=
		\left\{\begin{array}{ll}
			(m-1)(2m+1)-1 & \mbox{if $p\mid m$} \\
			(m-1)(2m+1) & \mbox{otherwise.}
		\end{array}\right.$$
	If $l=1$ then $n=2^{t-1}$, so $\dim V = (n-1)(2n+1)$ and we calculate that
	$\dim V = t \cdot (\dim L_{X_1}(2\o_{1,1}))^{t-1}$ if and only if $t=3$, so here $V|_{H}$ is irreducible if and only if $t=3$. This case is recorded in Table \ref{t:c4ii}. On the other hand, if $l \ge 2$ then
	$\dim V \ge (n-1)(2n+1)-1$
	and
	$$\dim L_{H^0}(\l|_{H^0}) = (\dim L_{X_1}(2\o_{1,1}))^{t-1} \cdot \dim L_{X_1}(\o_{1,2}) \le l^{t-1}(l-1)(2l+1)^{t}.$$
	It is easy to check that $(n-1)(2n+1)-1>l^{t-1}t(l-1)(2l+1)^{t}$ for all $t \ge 3$ and $l \ge 2$, whence $V|_{H}$ is reducible.
\end{proof}

\begin{lem}\label{c4ii:p4}
Proposition \ref{TH:C4II} holds in case (iv) of Table \ref{t:c4iis}.
\end{lem}

\begin{proof}
Here $G=D_n$, $2n=(2l)^t$ and $H^0 = X_1 \cdots X_t$ where $X_{i} \cong C_{l}$ and either $t$ is even or $p=2$. Note that $(l,t) \neq (1,2)$ since we are assuming $G$ is simple. As before, let $\{\b_{i,1}, \ldots, \b_{i,l}\}$ be a set of simple roots for $X_{i}$, with $\{\o_{i,1},\ldots,\o_{i,l}\}$ the corresponding fundamental dominant weights. In addition, for $1 \le i \le t$ and $0 \le m \le l$, set $\b_{i,l+m}=\b_{i,l-m}$ and $\o_{i,l+m}=\o_{i,l-m}$, where $\b_{i,0}=\o_{i,0}=0$.
Then for $1 \le j \le 2l-1$ we have $\b_{i,j}=-\o_{i,j-1}+2\o_{i,j}-\o_{i,j+1}$.
Also, for all $0 \le m \le n$ we set $\l_{n+m}=\l_{n-m}$, with $\l_0=0$.

Let $0 \le k <n$ be an integer and define the unique integers $r_k(i) \in \{0,\ldots,2l-1\}$ as before in \eqref{e:rik3}.
Note that $r_k(t-1) \le l-1$ since $n=l(2l)^{t-1}$.  For $t' \in \{1,\dots,t\}$ let $k'=\sum_{i=0}^{t'-1}r_{k}(i)(2l)^i$. By choosing an appropriate embedding of $H^0$ in $G$, we may assume that
$$\langle\l_{k}|_{H^0},\b_{i,j}\rangle=\langle\l_{k'}|_{H^0},\b_{i,j}\rangle$$
for all $1 \le i \leq t'$ and  $1 \le j \le l$.
Let $\xi_{k}$ denote the weight $-\l_{k}+\l_{k+1}$.
For $k \neq n-2$, $\xi_k$ is a weight of the natural $KG$-module $W$ and we deduce that \eqref{e:xi} holds.
Since
$\xi_{n+k}=-\xi_{n-k-1}$ and \eqref{e:n1k} holds, it follows that
$$\xi_{n+k}|_{H^0} =
\sum_{i=0}^{t-2}\left(-\o_{i+1,r_k(i)}+\o_{i+1,r_k(i)+1}\right)-\o_{t,l+r_k(t-1)}+\o_{t,l+1+r_k(t-1)}$$
for all $k \neq 1$.
Similarly, if $k=n-2$ then $-\l_{k}+\l_{k+1}+\l_{k+2}$
restricts to
$$\sum_{i=0}^{t-1}\left(-\o_{i+1,r_k(i)}+\o_{i+1,r_k(i)+1}\right)
=\o_{1,1}-\o_{1,2}-\sum_{i=2}^{t-1}\o_{i,1}-\o_{t,l-1}+\o_{t,l},$$
while the weight
$-\l_{n+k-1}-\l_{n+k}+\l_{n+k+1} = -\l_{n-k+1}-\l_{n-k}+\l_{n-k-1}$
 restricts to
\begin{align*}
& \sum_{i=0}^{t-2}\left(-\o_{i+1,r_k(i)}+\o_{i+1,r_k(i)+1}\right)-\o_{t,l+r_k(t-1)}+\o_{t,l+1+r_k(t-1)} \\
 = & -\o_{1,1}+\o_{1,2}+\sum_{i=2}^{t-1}\o_{i,1}-\o_{t,l}+\o_{t,l+1}
 \end{align*}
 when $k=1$.

As before, for an integer $1 \le k \le n$, let $i_k$ be minimal such that $r_k(i_k) \neq 0$ in the decomposition \eqref{e:rik3}. In view of the above restrictions, noting that $\a_{n}=(-\l_{n-2}+\l_{n-1}+\l_{n})+(-\l_{n-1}+\l_{n})$ and
$$n-1 = (2l-1)\sum_{i=0}^{t-2}(2l)^i+(l-1)(2l)^{t-1},$$
we deduce that
\begin{equation}\label{e:akh4}
\a_{k}|_{H^0}=
	\left\{
	\begin{array}{ll}
	\b_{i_k+1,1}-\displaystyle \sum_{i=1}^{i_k}\b_{i,1} & \mbox{if $i_k>0$, $k \neq n$ and $l=1$}\\
	 \b_{i_k+1,r_k(i_k)}-\displaystyle\sum_{i=1}^{i_k}\b(i) & \mbox{if $i_k>0$, $k \neq n$ and $l \ge 2$}\\
\b_{t,1} -\displaystyle\sum_{i=2}^{t-1}\b_{i,1} & \mbox{if $k=n$ and $l=1$}\\
	 \b_{1,1}+\b_{t,l}-\displaystyle\sum_{i=1}^{t-1}\b(i) & \mbox{if $k=n$ and $l \ge 2$}\\
\b_{1,r_k(0)} & \mbox{if $i_k=0$,}
	\end{array}
	\right.
	\end{equation}
	where $\b(i)$ is the highest long root in $\Sigma(X_i)$ (see \eqref{e:betai2}).
In particular, \eqref{e:ci0} holds.

Suppose $V$ has highest weight $\l=\sum_{i=1}^{n}a_i\l_i$ and $V|_{H}$ is irreducible.
We deal first with the case where $n=4$, so $(l,t,p)=(1,3,2)$ and we have
\begin{equation}\label{e2:lh04} \l|_{H^0}=(a_1+a_3)\o_{1,1}+(a_1+2a_2+a_3)\o_{2,1}+(a_1+2a_2+a_3+2a_4)\o_{3,1}
\end{equation}
	where $a_i \le 1$ for all $i$.
	First note that $\l|_{H^0}-\b_{1,1}$ occurs with multiplicity at most $1$ in $V|_{H^0}$, and both $\a_1$ and $\a_3$ restrict to $\b_{1,1}$ (see \eqref{e:akh4}), so we have $a_1+a_3\le 1$ (note that there are no $T_{H^0}$-weights of the form $\l|_{H^0} - \b_{1,1}+\b_{i,1}$ ($i=2,3$) in $V$).
	
	Now suppose $a_2 \neq 0$. By \eqref{e:ci0}, $\l-\a_2$ is the unique weight of $V$ that restricts to $\l|_{H^0}+\b_{1,1}-\b_{2,1}$, and thus $\l|_{H^0}-\b_{2,1}$ occurs with multiplicity at most $2$ in $V|_{H^0}$.
	As both $\l-\a_1-\a_2$ and $\l-\a_2-\a_3$ restrict to $\l|_{H^0}-\b_{2,1}$, Lemma \ref{l:118} implies that $a_1=a_3=0$.
	Moreover, $\mu=\l-\a_2-\a_4\in\L(V)$ restricts to $\l|_{H^0}+\b_{1,1}-\b_{3,1}$ and affords the highest weight of a $KH^0$-composition factor (see Lemma \ref{l:sumc}), so
	 $\langle\mu|_{H^0},\b_{1,1}\rangle=\langle\l|_{H^0},\b_{3,1}\rangle$ and thus $a_4=0$ (see \eqref{e2:lh04}).
	It follows that $\l=\l_2$, so $\l|_{H^0}=2\o_{2,1}+2\o_{3,1}$.
	However $\dim V = 26$ (see Proposition \ref{p:dims}) is not divisible by  $\dim L_{H^0}(\l|_{H^0})=4$,
	so $H$ acts reducibly on $V$. This contradiction implies that $a_2=0$.
	
	If $a_4 \neq 0$ then $\mu_1=\l-\a_4$ and $\mu_2=\l-\a_2-\a_4\in\L(V)$ restrict to $\l|_{H^0}+\b_{2,1}-\b_{3,1}$ and $\l|_{H^0}+\b_{1,1}-\b_{3,1}$ respectively, so they afford the highest weights of $KH^0$-composition factors.
	If $a_1=a_3=0$ then $\l=\l_4$, $\l|_{H^0}=2\o_{3,1}$ and the other two composition factors of $V|_{H^0}$ have highest weights $\l|_{H^0}+\b_{i,1}-\b_{3,1}$ for $i=1,2$.
	However $\nu=\l-\a_1-\a_2-\a_4\in\L(V)$ and $\nu|_{H^0}=\l|_{H^0}-\b_{3,1}$ is not conjugate to a weight occurring in any of the composition factors of $V|_{H^0}$, which is absurd.
	Therefore $a_1+a_3=1$, so $\l|_{H^0}=\o_{1,1}+\o_{2,1}+3\o_{3,1}$ and $V|_{H}$ is irreducible since $\dim V = 48$ and so
	$$V|_{H^0}=(3\otimes1\otimes1) \oplus (1\otimes3\otimes1) \oplus (1\otimes1\otimes3).$$
		This case is recorded in Table \ref{t:c4ii}.
	Finally, if $\l=\l_1$ or $\l=\l_3$ then $\l|_{H^0}=\o_{1,1}+\o_{2,1}+\o_{3,1}$, $V|_{H^0}=1\otimes1\otimes1$ and $V|_{H}$ is irreducible.

For the remainder we may assume $n>4$.
Suppose first that $l=1$ (so $t>3$) and assume $a_k \neq 0$. If $k \le n-3$ then arguing as in the proof of Lemma \ref{c4ii:p3} we deduce that $k \le 2$. If $a_{n-2} \ne 0$ then $\mu = \l-\a_{n-2}-\a_{n} \in \L(V)$ and
$$\mu|_{H^0} = \l|_{H^0}+\b_{1,1}+\sum_{i=3}^{t-1}\b_{i,1}-\b_{t,1}$$
(see \eqref{e:akh4}),
which contradicts Lemma \ref{l:sumc} since $t>3$. Similarly, if $a_n \ne 0$ then $\l-\a_{n} \in \L(V)$ restricts to $\l|_{H^0}+\sum_{i=2}^{t-1}\b_{i,1}-\b_{t,1}$, and once again we arrive at a contradiction via Lemma \ref{l:sumc}.

Next suppose $l=1$ and $a_{n-1} \ne 0$.
Here we can argue as in the proof of Lemma \ref{c4ii:p3} to get $\l=a_2\l_2+a_{n-1}\l_{n-1}$.
Now $\mu=\l-\a_{n-2}-\a_{n-1}-\a_{n} \in \L(V)$ restricts to $\l|_{H^0}+\sum_{i=3}^{t-1}\b_{i,1}-\b_{t,1}$, so
	Lemma \ref{l:sumc} implies that $t=4$ (so $n=8$)
	and $\mu$ affords the highest weight of a $KH^0$-composition factor,
	whence $\langle\mu|_{H^0},\b_{3,1}\rangle=\langle\l|_{H^0},\b_{4,1}\rangle$.
	Now
		$$\l_2|_{H^0} = 2(\o_{2,1}+\o_{3,1}+\o_{4,1}),\;\;
		\l_7|_{H^0} = \o_{1,1}+\o_{2,1}+\o_{3,1}+3\o_{4,1},$$
		 so $\langle\l|_{H^0},\b_{4,1}\rangle = 2a_2+3a_7$ and we deduce that
		$a_7=1$.
	If $a_2 \neq 0$ then $\nu=\l-\a_2 - \cdots - \a_7 \in \L(V)$ restricts to $\l|_{H^0}-\b_{2,1}-\b_{3,1}$.
	Set $\nu_1=\l$, $\nu_2=\l-\a_{2}$ and $\nu_3=\l-\a_{2}-\a_{3}-\a_{4}$. By \eqref{e:ci0},  $\nu_i$ is the unique weight of $V$ such that $\nu_i|_{H^0}=\l|_{H^0}+\b_{1,1}-\b_{i,1}$. Also note that $\nu|_{H^0}$ occurs with multiplicity $1$ under each $\nu_i|_{H^0}$, and it is easy to check that there is no weight of $V$ that restricts to one of the following
	$$\lambda|_{H^0} + \beta_{2,1} - \beta_{3,1}, \; \lambda|_{H^0} + \beta_{4,1} - \beta_{3,1}, \; \lambda|_{H^0} + \beta_{3,1} - \beta_{2,1},\; \lambda|_{H^0} + \beta_{4,1} - \beta_{2,1}.$$ 
	In particular, if $\chi \in \L(V)$ affords the highest weight of a $KH^0$-composition factor containing $\nu|_{H^0}$ then $\chi|_{H^0} = \nu_i|_{H^0}$ for some $i$,
		so $\nu|_{H^0}$ occurs with multiplicity at most $3$ in $V|_{H^0}$. However,
 Lemma \ref{l:s816} gives $m_{V}(\nu) \ge 5$, so this contradiction implies that $a_2=0$.
	Therefore $\l=\l_7$, so $\dim V = 2^7$ (see Lemma \ref{l:dmspin}), $\l|_{H^0}=\o_{1,1}+\o_{2,1}+\o_{3,1}+3\o_{4,1}$ and arguing by dimension, we deduce that $V|_H$ is irreducible if and only if $p \neq 3$. This case is recorded in Table \ref{t:c4ii}. For $l=1$ we have now reduced to the case $\l=a_1\l_1+a_2\l_2$ (with $t>3$).
	
Next suppose $l \ge 2$ and $a_k \neq 0$.
As before, if $k \le n-2l$ then we can argue as in the proof of Lemma \ref{c4ii:p3} to get $k \le 2$.
If $n-2l+3 \le k \le n-3$ then $\mu=\l-\sum_{i=k}^{n-2}\a_{i}-\a_{n} \in \L(V)$ and
$$\mu|_{H^0} = \l|_{H^0}-\sum_{j=2l+k-n}^{2l-2}\b_{1,j}-\b_{1,1}-\b_{t,l}+\sum_{i=1}^{t-1}\b(i),$$
where $\b(i)$ is given in \eqref{e:betai2}. However $n-k+1 <(t-1)(2l-1)$, so this contradicts Lemma \ref{l:sumc}.
Similarly, if $a_{n-2l+1} \neq 0$ then $\mu = \l-\a_{n-2l}-\a_{n-2l+1} \in \L(V)$
and
$$\mu|_{H^0} = \l|_{H^0}-\b_{1,1}+2\sum_{j=1}^{l-1}\b_{1,j}+\b_{1,l}-\b_{2,1},$$
which again is ruled out by Lemma \ref{l:sumc}.
If $a_{n-2l+2} \neq 0$ then $\mu = \l-\a_{n-2l}-\a_{n-2l+1}-\a_{n-2l+2} \in \L(V)$ and, if $l\ge3$, we have 
$$\mu|_{H^0} =
\l|_{H^0}-\b_{2,1}+\b_{1,1}+\b_{1,2}+2\sum_{j=3}^{l-1}\b_{1,j}+\b_{1,l},$$
contradicting Lemma \ref{l:sumc}.
Similarly, if $a_{n-2} \neq 0$ then $\mu = \l-\a_{n-2}-\a_{n} \in \L(V)$ and
$$\mu|_{H^0} =  \l|_{H^0}-\b_{1,1}-\b_{1,2}+\sum_{i=1}^{t-1}\b(i)-\b_{t,l},$$
so $\ell(\mu|_{H^0}) = (2l-1)(t-1)-3$. In particular, if $(t,l) \neq (2,2)$ then $\ell(\mu|_{H^0})>0$, which cannot happen in view of Lemma \ref{l:sumc}. On the other hand, if $(t,l) = (2,2)$ then
$\mu|_{H^0}=\l|_{H^0}+\b_{1,1}-\b_{2,2}$, but this contradicts Lemma \ref{l:sumc}(ii). Next suppose
$a_n \neq 0$. Then $\mu = \l-\a_{n} \in \L(V)$ and
$$\mu|_{H^0} = \l|_{H^0}-\b_{1,1}+\sum_{i=1}^{t-1}\b(i)-\b_{t,l},$$
which yet again contradicts Lemma \ref{l:sumc}. For $l \ge 2$, we have now reduced to the case $\l = a_1\l_1+a_2\l_2+a_{n-1}\l_{n-1}$.		

Continuing with the case $l \ge 2$, suppose $a_{n-1} \neq 0$. Then $\mu = \l-\a_{n-2}-\a_{n-1}-\a_{n} \in \L(V)$ and
$$\mu|_{H^0} = \l|_{H^0}-2\b_{1,1}-\b_{1,2}+\sum_{i=1}^{t-1}\b(i)-\b_{t,l},$$
so Lemma \ref{l:sumc} implies that $t=l=2$, and thus $n=8$.
One can check that there is no weight of $V$ that restricts to one of the following
$$\lambda|_{H^0} + \beta_{1,2} - \beta_{1,1},\;  \lambda|_{H^0} + \beta_{2,1} - \beta_{1,1}, \; \lambda|_{H^0} + \beta_{2,2} - \beta_{1,1},$$ 
so $\lambda|_{H^0} - \beta_{1,1}$ occurs with multiplicity at most $1$ in $V|_{H^0}$. In particular, since both
$\a_1$ and $\a_7$ restrict to $\b_{1,1}$, it follows that $a_1=0$ and $\l=a_2\l_2+a_7\l_7$.
Hence $\mu=\l-\a_4 - \cdots - \a_8 \in \L(V)$ restricts to $\l|_{H^0}+\b_{1,1}+\b_{1,2}-\b_{2,1}-\b_{2,2}$, so $\mu$ affords the highest weight of a $KH^0$-composition factor.
	Therefore
	 $\langle\mu|_{H^0},\b_{2,j}\rangle=\langle\l|_{H^0},\b_{1,j}\rangle$ for $j \le 2$, so $2a_2+a_7=a_7$ and $a_2=a_7-1$ since 
		$\l_2|_{H^0} = \o_{1,2}+2\o_{2,1}$
		and
		$\l_7|_{H^0} = \o_{1,1}+\o_{2,1}+\o_{2,2}$. Therefore
		 $a_2=0$ and $a_7=1$, so 
	$\l=\l_7$, $\dim V = 2^7$ and $\l|_{H^0}=\o_{1,1}+\o_{2,1}+\o_{2,2}$. Moreover, since $\dim L_{X_1}(\o_{1,1}+\o_{1,2})=16-4\delta_{5,p}$ (see \cite[Table A.22]{Lubeck}), we deduce that $V|_H$ is irreducible if and only if $p \neq 5$. This case is recorded in Table \ref{t:c4ii}.

To complete the proof of the lemma, we may assume $l \ge 1$, $n>4$ and $\l=a_1\l_1+a_2\l_2$. If $a_2=0$ then $\l=a_1\l_1$, and by arguing as in the proof of Lemma \ref{c4ii:p3} we deduce that $a_1=1$, so $\l=\l_1$
and $H$ acts irreducibly on $V$.

Finally, let us assume $a_2 \neq 0$. First note that if $\mu \in \L(V)$ affords the highest weight of a $KH^0$-composition factor then either $\mu=\l$, or there exists a permutation $\s \in S_t$ such that $\s(1)=i_0 \neq 1$ and \eqref{e:muh0} holds.

By \eqref{e:ci0}, $\l-\sum_{i=2}^{n-2}\a_{i}-\a_n$ is the unique weight of $V$ that restricts to $\l|_{H^0}+\b_{1,1}-\b_{t,1}$.
	 By arguing as in the proof of Lemma \ref{c4ii:p3}, we deduce that $a_2=1$, $a_1=0$ and $V|_{H^0}$ has exactly $t$ composition factors.
	Therefore $\l=\l_2$ and thus $\l|_{H^0}=2\sum_{i=2}^{t}\o_{i,1}$ if $l=1$, otherwise $\l|_{H^0}=
	\o_{1,2}+2\sum_{i=2}^{t}\o_{i,1}$.
By Proposition \ref{p:dims} we have
	$$\dim{L_{C_l}(2\l_1)}=\left\{\begin{array}{ll}
	2l  & \mbox{if $p=2$} \\
	l(2l+1) & \mbox{otherwise}
	\end{array}\right.$$
	$$\dim{L_{C_l}(\l_2)}=
		\left\{\begin{array}{ll}
			(l-1)(2l+1)-1 & \mbox{if $p\mid l$} \\
			(l-1)(2l+1) & \mbox{otherwise}
		\end{array}\right.$$
and 
$$\dim V = \dim{L_{D_n}(\l_2)} \ge n(2n-1)-2.$$
	If $l=1$ then $n=2^{t-1}$,
	so $\dim V=n(2n-1)-2\delta_{2,p}$,
	and $\dim L_{H^0}(\l|_{H^0})=(\dim L_{X_{1}}(2\o_{1,1}))^{t-1}=2^{t-1}$ if $p=2$ and $3^{t-1}$ otherwise.
	It is easy to check that $\dim V>t \cdot \dim L_{H^0}(\l|_{H^0})$ for all $t \ge 2$, whence $V|_{H}$ is reducible.
	If $l\ge2$ then
$$\dim L_{H^0}(\l|_{H^0}) = \dim L_{X_{1}}(\o_{1,2}) \cdot (\dim L_{X_{1}}(2\o_{1,1}))^{t-1} \le (l-1)l^{t-1}(2l+1)^{t}$$
		and it is easy to check that
		$t \cdot (l-1)l^{t-1}(2l+1)^{t}<n(2n-1)-2$
		for all $t \ge 2$ and $l \ge 1$, so once again $V|_{H}$ is reducible.
\end{proof}

\begin{lem}\label{c4ii:p5}
Proposition \ref{TH:C4II} holds in case (v) of Table \ref{t:c4iis}.
\end{lem}

\begin{proof}
Here $G=D_n$, $2n=(2l)^t$ and $H^0 = X_1 \cdots X_t$ with $X_{i} \cong D_{l}$, $l \ge 3$  and $p \neq 2$.
Define $\{\b_{i,1}, \ldots, \b_{i,l}\}$ and  $\{\o_{i,1},\ldots,\o_{i,l}\}$ as before, and for $0 \le m \le n$ set $\l_{n+m}=\l_{n-m}$ and $\l_0=0$. For all $1 \le i \le t$ and $0 \le m \le l$, set $\b_{i,l+m}=\b_{i,l-m}$ and $\o_{i,l+m}=\o_{i,l-m}$, where $\b_{i,0}=\o_{i,0}=0$.

Let $0 \le k < n$ be an integer and define the unique integers $r_k(i) \in \{0,\dots,2l-1\}$ as in \eqref{e:rik3}. For $t' \in \{1,\dots,t\}$ let $k'=\sum_{i=0}^{t'-1}r_{k}(i)(2l)^i$. By choosing an appropriate embedding of $H^0$ in $G$, we may assume that 
$$\langle\l_{k}|_{H^0},\b_{i,j}\rangle=\langle\l_{k'}|_{H^0},\b_{i,j}\rangle$$ 
for all $1 \le i \leq t'$, $1 \le j \le l$.

Set $\xi_{k} = -\l_{k}+\l_{k+1}$. If $k \neq n-2$ then $\xi_k$ is a weight of the natural $KG$-module $W$ and we have
$$\langle \xi_{k}|_{H^0},\b_{i+1,j}\rangle=
\left\{\begin{array}{rl}
1 & \mbox{if $j=r_k(i)+1 \neq 2l$ or $j=l=r_k(i)+2$} \\
-1 & \mbox{if $j=r_k(i) \neq 0$ or $j=l=r_k(i)-1$} \\ 
0 & \mbox{otherwise}
\end{array}\right.$$
for all $0 \le i \le t-1$.
Now $\xi_{n+k}=-\xi_{n-1-k}$ and
$$n+k=\sum_{i=0}^{t-2}r_k(i)(2l)^i+(l+r_k(t-1))(2l)^{t-1},$$
$$n-1-k=\sum_{i=0}^{t-2}(2l-1-r_k(i))(2l)^i+(l-1-r_k(t-1))(2l)^{t-1},$$
so for $k \neq 1$ and $0 \le i \le t-2$ we have
$$\langle \xi_{n+k}|_{H^0},\b_{i+1,j}\rangle=
\left\{\begin{array}{rl}
1 & \mbox{if $j=2l-1-r_k(i) \neq 0$ or $j=l=2l-2-r_k(i)$} \\ 
-1 & \mbox{if $j=2l-r_k(i) \neq 2l$ or $j=l=2l+1-r_k(i)$}
\end{array}\right.$$
and
$$\langle \xi_{n+k}|_{H^0},\b_{t,j}\rangle=
\left\{\begin{array}{rl}
1 & \mbox{if $j=l-r_k(t-1)-1 \neq 0$} \\ 
-1 & \mbox{if $j=l-r_k(t-1)$ or $j=l=l+1-r_k(t-1)$} \\ 
0 & \mbox{otherwise.}
\end{array}\right.$$
We also observe that the weight $-\l_{n-2}+\l_{n-1}+\l_{n}$
restricts to
$$\o_{1,1}-\o_{1,2}-\sum_{i=2}^{t-1}\o_{i,1}-\o_{t,l-1}+\o_{t,l},$$
while $-\l_{n}-\l_{n+1}+\l_{n+2}$ restricts to
$$-\o_{1,1}+\o_{1,2}+\sum_{i=2}^{t-1}\o_{i,1}-\o_{t,l}+\o_{t,l+1}.$$

As before, for an integer $1 \le k \le n$, let $i_k$ be minimal such that  $r_k(i_k) \neq 0$ in \eqref{e:rik3}. In view of the above restrictions, we deduce that
\begin{equation}\label{e:akh5}
\a_{k}|_{H^0} = \left\{\begin{array}{ll}
\b_{1,r_k(0)} & \mbox{if $i_k=0$ and $r_k(0) \neq l$}\\
-\b_{1,l-1}+\b_{1,l} & \mbox{if $i_k=0$ and $r_k(0)=l$}\\
-\displaystyle \sum_{i=1}^{i_k}\b(i)+\b_{i_k+1,r_k(i_k)} & \mbox{if $i_{k} \ge 1$, $k \neq n$ and $r_k(i_k) \neq l$}\\	
-\displaystyle \sum_{i=1}^{i_k}\b(i)-\b_{i_k+1,l-1}+\b_{i_k+1,l} & \mbox{if $i_k \ge 1$, $k \neq n$ and $r_k(i_k)=l$}\\
\displaystyle \b_{1,1}-\sum_{i=1}^{t-1}\b(i)-\b_{t,l-1}+\b_{t,l} & \mbox{if $k=n$,}
\end{array}\right.
\end{equation}
where
\begin{equation}\label{e:betai3}
\b(i) = 2\b_{i,1}+\dots+2\b_{i,l-2}+\b_{i,l-1}+\b_{i,l}.
\end{equation}
We also note that \eqref{e:ci0} holds.

Recall that $V$ has highest weight $\l=\sum_{i=1}^na_i\l_i$ and let us assume $V|_{H}$ is irreducible.
Suppose $a_k \neq 0$. Define the integers $r_k(i)$ as in \eqref{e:rik3} and write $k=r_k(0)+2l(a-1)$, where $a \in \{1,\dots,l(2l)^{t-2}+1\}$.
Note that $\l-\a_{k} \in \L(V)$.

If $i_k \ge 1$ then \eqref{e:akh5} implies that $\ell((\l-\a_{k})|_{H^0})\ge2(l-1)i_k-1 > 0$, which contradicts Lemma \ref{l:sumc}. Therefore $i_k =0$ and thus $r_k(0) \neq 0$. 

Suppose $k \le n-2l$.
Then $\mu=\l-\sum_{i=k}^{2la}\a_{i} \in \L(V)$ and since $2la=(r_k(1)+1)2l+\sum_{i=2}^{t-1}r_k(i)(2l)^{i}$ we deduce that
$$\mu|_{H^0} = \left\{\begin{array}{ll}
\displaystyle\l|_{H^0}-\sum_{j=r_k(0)}^{2l-1}\b_{1,j}+\b_{1,l-1}-\b_{\iota,r}+\sum_{i=1}^{i_{2la}}\b(i) & \mbox{if $r \neq l$, $r_k(0) \le l$}\\
\displaystyle\l|_{H^0}-\sum_{j=r_k(0)}^{2l-1}\b_{1,j}-\b_{\iota,r}+\sum_{i=1}^{i_{2la}}\b(i) & \mbox{if $r \neq l$, $r_k(0)>l$}\\
\displaystyle\l|_{H^0}-\sum_{j=r_k(0)}^{2l-1}\b_{1,j}+\b_{1,l-1}+\b_{\iota,l-1}-\b_{\iota,l}+\sum_{i=1}^{i_{2la}}\b(i) & \mbox{if $r = l$, $r_k(0)\le l$}\\
\displaystyle\l|_{H^0}-\sum_{j=r_k(0)}^{2l-1}\b_{1,j}+\b_{\iota,l-1}-\b_{\iota,l}+\sum_{i=1}^{i_{2la}}\b(i) & \mbox{if $r = l$, $r_k(0)>l$,}
\end{array}\right.$$
where $r=r_{2la}(i_{2la})$, $\iota = i_{2la}+1$ and $\b(i)$ is the $\mathbb{Z}$-linear combination of simple roots defined in \eqref{e:betai3}.

Therefore $\ell(\mu|_{H^0})\ge 2(l-1)(i_{2la}-1)+r_k(0)-3$ and, by applying Lemma \ref{l:sumc} we deduce that $i_{2la}=1$ and $r_k(0) \le 3$.
Therefore $r_k(0) \le l$ and thus $\ell(\mu|_{H^0})\ge r_k(0)-2$, so $r_k(0)=1$ or $2$. 
Suppose $r_k(0)=1$. If $k\neq 1$ then $i_{k-1} \geq 1$,
			$\nu=\l-\a_{k-1}-\a_{k} \in \L(V)$
			and $\ell(\nu|_{H^0})\ge2(l-1)i_{k-1}-2\ge 2$, so Lemma \ref{l:sumc} implies that $k=1$ is the only possibility. Similarly, if $r_k(0)=2$ and $k \neq 2$ then $i_{k-2} \geq 1$,
			$\nu=\l-\a_{k-2}-\a_{k-1}-\a_{k} \in \L(V)$
			and $\ell(\nu|_{H^0})\ge 2(l-1)i_{k-2}-3 \geq 1$,
			which is a contradiction. It follows that if $a_k \neq 0$ and $k \le n-2l$ then $k \le 2$.

Next suppose $n-2l+1 \le k \le n-l-1$. Then
$\mu=\l-\sum_{i=n-2l}^{k}\a_{i} \in \L(V)$ restricts to
$$\l|_{H^0}-\sum_{j=1}^{2l+k-n}\b_{1,j}-\b_{2,2l-1}+\b(1),$$
so $\ell(\mu|_{H^0})=n-k-3 \ge l-2 \ge 1$ and thus Lemma \ref{l:sumc} implies that $a_k = 0$.
Similarly, if $a_{n-l}\neq 0$ then
$\mu=\l-\sum_{i=n-2l}^{n-l}\a_{i} \in \L(V)$ and
$$\mu|_{H^0} = \l|_{H^0}-\b_{2,2l-1}+\b_{1,1}+\dots+\b_{1,l-2}+\b_{1,l-1}$$
so $\ell(\mu|_{H^0})\ge 1$ and again we conclude that $a_{n-l}=0$.
For $n-l+1 \le k \le n-2$ we have
$\mu=\l-\sum_{i=k}^{n}\a_{i}+\a_{n-1} \in \L(V)$
and
$$\mu|_{H^0} = \l|_{H^0}-\sum_{j=2l+k-n}^{2l-2}\b_{1,j}-\b_{1,1}+\b_{t,l-1}-\b_{t,l}+\sum_{i=1}^{t-1}\b(i)$$
so $\ell_1(\mu|_{H^0})=2t-3 \ge 1$, which contradicts Lemma~\ref{l:sumc}. Therefore $a_k=0$.
Similarly, since $\l-\a_{n}$ restricts to $$\l|_{H^0}-\b_{1,1}+\sum_{i=1}^{t-1}\b(i)+\b_{t,l-1}-\b_{t,l}$$
we deduce that $a_n=0$.		
Finally, if $a_{n-1} \neq 0$ then $\mu=\l-\a_{n-2}-\a_{n-1}-\a_{n} \in \L(V)$ and
$$\mu|_{H^0} = \l|_{H^0}-2\b_{1,1}-\b_{1,2}+\sum_{i=1}^{t-1}\b(i)+\b_{t,l-1}-\b_{t,l},$$
so $\ell(\mu|_{H^0})=(t-1)(2l-2)-3 \ge 1$. Applying  Lemma \ref{l:sumc} once again, we conclude that $a_{n-1}=0$.

We have now reduced to the case $\l=a_1\l_1+a_2\l_2$, so
$$\l|_{H^0} = \left\{\begin{array}{ll}
\displaystyle a_1\o_{1,1}+a_2(\o_{1,2}+\o_{1,3})+\sum_{i=2}^{t}(a_1+2a_2)\o_{i,1} & \mbox{if $l=3$}\\
\displaystyle a_1\o_{1,1}+a_2\o_{1,2}+\sum_{i=2}^{t}(a_1+2a_2)\o_{i,1} & \mbox{if $l \ge 4$.}\\
\end{array}\right.$$

Suppose $a_2 \neq 0$. As in the proof of Lemma \ref{c4ii:p3}, we can reduce  to the case $(a_1,a_2)=(0,1)$, so $V|_{H^0}$ has exactly $t$ composition factors with highest weights $\l|_{H^0}+\b_{1,1}-\b_{i,1}$, for $1 \le i \le t$. Therefore $\l=\l_2$ and $\l|_{H^0}=\o_{1,2}+\o_{1,3}+2\sum_{i=2}^{t}\o_{i,1}$ if $l=3$, and $\l|_{H^0}=\o_{1,2}+2\sum_{i=2}^{t}\o_{i,1}$ if $l \ge 4$.

	By Proposition \ref{p:dims}, since $p \neq 2$, we have
	$\dim{L_{D_m}(\l_2)}=m(2m-1)$
	and
	$$\dim{L_{D_l}(2\l_1)}=
	\left\{\begin{array}{ll}
	(l+1)(2l-1)-1 & \mbox{if $p\mid l$} \\
	(l+1)(2l-1) & \mbox{otherwise,}
	\end{array}\right.$$
	while $\dim{L_{D_3}(\o_{1,2}+\o_{1,3})}=15$.
	Hence $\dim V = n(2n-1)$ and $\dim L_{H^0}(\l|_{H^0}) \le l(2l-1)^{t}(l+1)^{t-1}$.
	It is easy to check that $n(2n-1)>t \cdot l(2l-1)^{t}(l+1)^{t-1}$ for all $t \ge 2$ and $l \ge 3$, whence $V|_{H}$ is reducible.

We have now reduced to the case $\l=a_1\l_1$. Here $\l$ is the unique weight in $\L(V)$ that affords the highest weight of a $KH^0$-composition factor (see \eqref{e:ci0}),  
so $V|_{H^0}$ is irreducible.
Therefore, using \cite{Seitz2}, we conclude that the only example is $\l=\l_1$.
\end{proof}

This completes the proof of Proposition \ref{TH:C4II}. Moreover, in view of Propositions \ref{TH:C1,3,6}, \ref{TH:C2} and \ref{TH:C4I}, the proof of Theorem \ref{main} is complete.

\backmatter

\bibliographystyle{amsalpha}

\end{document}